\documentclass{amsart}
\usepackage{amsfonts,amscd,amsthm,amsgen,amsmath,amssymb}
\usepackage[all]{xy}
\usepackage{epsfig,color}
\usepackage[vcentermath]{youngtab}
\usepackage{stmaryrd}
\newtheorem{theorem}{Theorem}[section]
\newtheorem{lemma}[theorem]{Lemma}
\newtheorem{corollary}[theorem]{Corollary}
\newtheorem{proposition}[theorem]{Proposition}

\theoremstyle{definition}
\newtheorem{definition}[theorem]{Definition}
\newtheorem{example}[theorem]{Example}

\theoremstyle{remark}
\newtheorem{remark}[theorem]{Remark}

\numberwithin{equation}{section}

\usepackage{hyperref}

\begin{document}


\title[On the Bauer-Furuta and Seiberg-Witten invariants]{On the Bauer-Furuta and Seiberg-Witten invariants of families of $4$-manifolds}
\author{David Baraglia, Hokuto Konno}

\address{School of Mathematical Sciences, The University of Adelaide, Adelaide SA 5005, Australia}
\email{david.baraglia@adelaide.edu.au}

\address{Graduate School of Mathematical Sciences, the University of Tokyo, 3-8-1 Komaba, Meguro, Tokyo 153-8914, Japan}
\email{konno@ms.u-tokyo.ac.jp}

\begin{abstract}
We show how the families Seiberg-Witten invariants of a family of smooth $4$-manifolds can be recovered from the families Bauer-Furuta invariant via a cohomological formula. We use this formula to deduce several properties of the families Seiberg-Witten invariants. We give a formula for the Steenrod squares of the families Seiberg-Witten invariants leading to a series of mod $2$ relations between these invariants and the Chern classes of the spin$^c$ index bundle of the family. As a result we discover a new aspect of the ordinary Seiberg-Witten invariants of a $4$-manifold $X$: they obstruct the existence of certain families of $4$-manifolds with fibres diffeomorphic to $X$. As a concrete geometric application, we shall detect a non-smoothable family of $K3$ surfaces. Our formalism also leads to a simple new proof of the families wall crossing formula. Lastly, we introduce $K$-theoretic Seiberg-Witten invariants and give a formula expressing the Chern character of the $K$-theoretic Seiberg-Witten invariants in terms of the cohomological Seiberg-Witten invariants. This leads to new divisibility properties of the families Seiberg-Witten invariants.
\end{abstract}


\date{\today}



\maketitle


\section{Introduction}
 
In \cite{bafu}, Bauer and Furuta constructed a refinement of the Seiberg-Witten invariants of smooth $4$-manifolds taking values in stable cohomotopy. The refined Bauer-Furuta invariant contains strictly more information than the ordinary Seiberg-Witten invariants. Moreover the existence of the refined invariant is useful even if one's primary interest is in the Seiberg-Witten invariants themselves. To understand this point, recall that the usual definition of the Seiberg-Witten invariants involves the construction of a smooth moduli space and one must use perturbations to achieve transversality. However the Bauer-Furuta stable cohomotopy refinement affords us the luxury of working in the setting of algebraic topology so that issues of transversality can be bypassed.

The Seiberg-Witten invariants of smooth $4$-manifolds have been extended to invariants of families of smooth manifolds \cite{rub1, rub3, liu, liliu, na1}. By a family of smooth $4$-manifolds we mean a smooth locally trivial fibre bundle over a smooth base manifold whose fibres are a fixed compact smooth $4$-manifold. Additionally the family is assumed to be equipped with a spin$^c$-structure on the vertical tangent bundle. Naturally one may also ask for a families extension of the Bauer-Furuta invariant. The existence of such an extension can already been seen implicitly in Bauer-Furuta (\cite[Theorem 2.6]{bafu}) and was further developed by Szymik in \cite{szy}. Neither of these works establish how, if at all, one can recover the families Seiberg-Witten invariants from the families Bauer-Furuta invariant. In this paper we answer this question in the affirmative, showing that the families Seiberg-Witten invariants can be recovered from the families Bauer-Furuta invariant and we give an explicit cohomological formula relating the two (Theorem \ref{thm:pdclass}). In subsequent sections of the paper we use our cohomological formula to extract a number of results concerning the families Seiberg-Witten invariants. Specifically:
\begin{itemize}
\item{We compute the Steenrod powers of the families Seiberg-Witten invariants (Section \ref{sec:steenrod}).}
\item{We give a simple new proof of the wall crossing formula of Li and Liu  \cite{liliu} for for families Seiberg-Witten invariants (Section \ref{sec:wallcrossing}).}
\item{We introduce $K$-theoretic families Seiberg-Witten invariants and show how they are related to the usual cohomological families invariants via the Chern character. This leads to certain divisibility properties of the families Seiberg-Witten invariants (Section \ref{sec:ktheory}).}
\end{itemize}

\subsection{Families Seiberg-Witten invariants}\label{sec:introfsw}
To state our main results we need to recall the construction of the families Seiberg-Witten and Bauer-Furuta invariants. The setting is as follows: let $\pi : E \to B$ be a smooth locally trivial fibre bundle over a compact smooth base manifold $B$ whose fibres are a fixed compact smooth $4$-manifold $X$. More precisely, fix a basepoint $b_0 \in B$ and a diffeomorphism $X \cong E_{b_0} = \pi^{-1}(b_0)$ of $X$ with the fibre of $E$ over $b_0$. Assume that the vertical tangent bundle $E$ is equipped with a spin$^c$-structure $\mathfrak{s}_E$ and let $\mathfrak{s}_X$ be the corresponding spin$^c$-structure on $X$ obtained by restriction to the fibre $E_{b_0}$. The data $(\pi : E \to B , \mathfrak{s}_E)$ (or $(E,\mathfrak{s}_E)$ for short) will be referred to as a {\em spin$^c$ family} of $4$-manifolds with fibres diffeomorphic to $(X , \mathfrak{s}_X)$. In this introduction we will consider only the case that $b_1(X)=0$. The necessary modifications to accommodate the case $b_1(X) > 0$ will be explained further in the paper.

Choose a smoothly varying fibrewise metric $g = \{g_b\}_{b\in B}$ on $E$ and a smoothly varying family of $2$-forms $\eta = \{ \eta_b\}_{b\in B}$, where $\eta_b$ is a $g_b$-self-dual $2$-form on the fibre $E_b$. The families Seiberg-Witten invariants will in general depend on a choice of chamber $\phi$ (see Definition \ref{def:chamber}). If $b^+(X) > dim(B)+1$ then there exists a unique chamber, otherwise the Seiberg-Witten invariants will generally depend on $\phi$. Assume that the metric $g$ and $2$-form perturbation $\eta$ are chosen so that $(g,\eta)$ lies in the chamber corresponding to $\phi$. 

Let $\mathcal{M} = \mathcal{M}(E,\mathfrak{s}_E , g , \eta )$ denote the moduli space of gauge equivalence classes of solutions to the Seiberg-Witten equations on the fibres of $E$ with respect to the fibrewise spin$^c$-structure $\mathfrak{s}_E$, the fibrewise metric $g$ and fibrewise perturbation $\eta$. For generic $(g,\eta)$ the moduli space $\mathcal{M}(E,\mathfrak{s}_E , g , \eta)$ is a smooth manifold of dimension $dim(B) + (2d-b^+-1)$, where $d = \frac{c_1(\mathfrak{s}_X)^2 - \sigma(X)}{8}$ and $b^+ = b^+(X)$ is the dimension of the space of harmonic self-dual $2$-forms on $X$. Let $\pi_{\mathcal{M}} : \mathcal{M} \to B$ denote the natural projection to $B$. The gauge theoretical construction of $\mathcal{M}$ determines a complex line bundle $\mathcal{L} \to \mathcal{M}$ as follows. Recall that the gauge group in Seiberg-Witten theory is $\mathcal{G} = Map(X , U(1))$. In the families setting this single group is replaced by a bundle of groups over $B$ and $\mathcal{M}$ is obtained as a fibrewise quotient. Let $\mathcal{G}_0$ be the subgroup of $\mathcal{G}$ given by:
\[
\mathcal{G}_0 = \left\{ g \in \mathcal{G} \; \left| \; g = e^{if}, \; \; \int_X f dvol_X = 0 \right. \right\}.
\]
Since $b_1(X) = 0$, we have that $\mathcal{G}_0/\mathcal{G} \cong S^1$. The definition of $\mathcal{G}_0$ also works in the families setting since the smoothly varying family of metrics $\{ g_b\}$ determines a smoothly varying family of fibrewise volume forms $\{ dvol_{E_b} \}$. Let $\widetilde{\mathcal{M}}$ be the families moduli space obtained by taking the quotient by the reduced gauge group $\mathcal{G}_0$. Then $\widetilde{\mathcal{M}} \to \mathcal{M}$ is a principal circle bundle and we let $\mathcal{L} \to \mathcal{M}$ be the associated line bundle.

Let $H^+ \to B$ be the vector bundle on $B$ whose fibre over $b$ is the space $H^+(E_b,g_b)$ of $g_b$-harmonic self-dual $2$-forms on $E_b$. To keep the introduction simple we will assume here that $H^+$ is oriented. The general case is dealt with in the paper using local systems. In the unparametrised setting, an orientation of the Seiberg-Witten moduli space corresponds to a choice of orientation on $H^+(X)$ (when $b_1(X)=0$). In the families setting this translates to the statement that a relative orientation of $\pi_{\mathcal{M}} : \mathcal{M} \to B$ is determined by a choice of orientation of the bundle $H^+$. In this paper we define the families Seiberg-Witten invariants of $(E , \mathfrak{s}_E)$ (with respect to the chamber $\phi$) to be the collection of cohomology classes $\{ SW_m(E , \mathfrak{s}_E , \phi) \}_{m \ge 0}$, where
\[
SW_m(E , \mathfrak{s}_E , \phi) = (\pi_{\mathcal{M}})_*(  c_1(\mathcal{L})^m ) \in H^{2m-(2d-b^+-1)}(B ; \mathbb{Z}).
\]
In the case that $H^+$ is not orientable, the families Seiberg-Witten invariants are instead valued in a local coefficient system. One needs to check that this definition does not depend on the particular choice of metric and perturbation $(g , \eta)$. The proof is much the same as in the unparametrised setting.

\subsection{Families Bauer-Furuta invariants}

In \cite{bafu} Bauer and Furuta constructed a stable cohomotopy refinement of the Seiberg-Witten invariant by taking a finite dimensional approximation of the Seiberg-Witten equations. This construction was extended to families of $4$-manifolds in \cite{szy}. We will recall how this finite dimensional approximation is constructed in Subsection \ref{sec:fda}. Suppose again that $\pi : E \to B$ is a spin$^c$ family of $4$-manifolds with fibres diffeomorphic to $X$. For simplicity we continue to assume in this introduction that $b_1(X) = 0$. The case $b_1(X)  > 0$ is explained in Section \ref{sec:fm}. The process of taking a finite dimensional approximation of the Seiberg-Witten equations parametrised by $B$ produces an $S^1$-equivariant map of sphere bundles over $B$:
\[
f : S_{V,U} \to S_{V',U'}.
\]
Here $V,V'$ are complex vector bundles over $B$ of ranks $a,a'$ and $U,U'$ are real vector bundles over $B$ of ranks $b,b'$. $S_{V,U}$ denotes the unit sphere bundle of $\mathbb{R}\oplus V \oplus U$, or equivalently the fibrewise one point compactification of $V\oplus U$ and $S_{V',U'}$ is defined similarly. The unit circle $S^1$ acts on $V,V'$ by fibrewise scalar multiplication and acts trivially on $U,U'$. The $S^1$ actions on $V\oplus U$ and $V' \oplus U'$ extend smoothly to the sphere bundles $S_{V,U}, S_{V',U'}$. 

As will be shown in Subsection \ref{sec:fda}, the vector bundles $V,V',U,U'$ satisfy:
\[
V - V' = D \in K^0(B), \quad U' - U = H^+ \in KO^0(B),
\]
where $D \in K^0(B)$ denotes the families index of the families spin$^c$ Dirac operator of $(E , \mathfrak{s}_{E})$ and $H^+$ is the vector bundle over $B$ whose fibres are the space of harmonic self-dual $2$-forms on the corresponding fibres of the family $E \to B$. Therefore
\[
a-a' = d = rank_{\mathbb{C}}(D) = \frac{c_1(\mathfrak{s}_X)^2 - \sigma(X)}{8}, \quad b'-b = b^+ = rank_{\mathbb{R}}(H^+).
\]
Lastly, we will see that the finite dimensional approximation $f$ can be constructed so as to have the property that $U' \cong U \oplus H^+$ and that the restriction $f|_U$ of $f$ to $U$ is given by the inclusion $U \to U'$.

For simplicity we assume in this introduction that the bundle $H^+$ is orientable. In this case $U$ and $U'$ can also be chosen to be orientable. The general case where $H^+$ is not necessarily orientable is dealt with in the body of the paper. By definition, a {\em chamber} of $f$ is a homotopy class of section $\phi : B \to U' \setminus U$, or equivalently a homotopy class of section of $H^+ \setminus \{ 0 \}$. As explained in Subsection \ref{sec:fda}, there is a canonical bijection between chambers for the families Seiberg-Witten equations of $(E , \mathfrak{s}|_E)$ and chambers of the finite dimensional approximation $f$. Thus to any chamber $[\phi]$ represented by a map $\phi : B \to U' \setminus U$, the families Seiberg-Witten invariants are defined: $SW_m(E , \mathfrak{s}_E , [\phi]) \in H^{2m-(2d-b^+-1)}(B ; \mathbb{Z})$. We prove a cohomological formula for these invariants in terms of $f$ and $\phi$. To state this formula we must introduce one additional construction.

Let us define $\widetilde{Y} = S_{V,U} \setminus N_{V,U}$, where $N_{V,U}$ is an $S^1$-invariant tubular neighbourhood of $S_{0,U}$ in $S_{V,U}$. Then $\widetilde{Y}$ is a smooth compact manifold with boundary $\partial \widetilde{Y}$. We also have that $S^1$ acts freely on $\widetilde{Y}$ and the quotient $Y = \widetilde{Y}/S^1$ is a compact manifold with boundary $\partial Y = \partial \widetilde{Y}/S^1$. Then we have isomorphisms of cohomology groups:
\[
H^*_{S^1}(S_{V,U} , S_{0,U} ; \mathbb{Z}) \cong H^*_{S^1}( \widetilde{Y} , \partial \widetilde{Y}) \cong H^*(Y , \partial Y),
\]
where the first isomorphism is excision and the second isomorphism follows since $S^1$ acts freely on $\widetilde{Y}$. Using Poincar\'e-Lefschetz duality we get a push-forward map
\[
(\pi_Y)_* : H^*( Y , \partial Y) \to H^{*-(2a+b-1)}(B ; \mathbb{Z}),
\]
since the fibres of $\pi_Y$ have dimension $2a+b-1$. 

The representative $\phi : B \to U' \setminus U$ of the chamber $[\phi]$ induces a pushforward map
\[
\phi_* : H^*_{S^1}(B ; \mathbb{Z}) \to H^{*+(2a'+b')}_{S^1}( S_{V',U'} , S_{0,U} ; \mathbb{Z})
\]
and the finite dimensional approximation $f$ induces a pullback map
\[
f^* : H^*_{S^1}(S_{V',U'} , S_{0,U} ; \mathbb{Z}) \to H^*_{S^1}(S_{V,U} , S_{0,U} ; \mathbb{Z}).
\]

Now we are ready to state our first main result:
\begin{theorem}\label{thm:cohomologyformula}
Let $(E , \mathfrak{s}_E)$ be a spin$^c$ family over $B$. Let $f : S_{V,U} \to S_{V',U'}$ be a finite dimensional approximation of the Seiberg-Witten monopole map, as described above. Let $[\phi]$ be a chamber of the families Seiberg-Witten equations represented by a section $\phi : B \to U' \setminus U$. Then:
\[
SW_m( E , \mathfrak{s}_E , [\phi]) = (\pi_Y)_*( x^m \smallsmile f^* \phi_*(1) ) \in H^{2m-(2d-b^+-1)}(B ; \mathbb{Z}),
\]
where $x$ is the standard generator of $H^2_{S^1}( pt ; \mathbb{Z})$.
\end{theorem}

Theorem \ref{thm:cohomologyformula} is a special case of the theorem that we prove in the paper (Theorem \ref{thm:pdclass}). Our general result applies also to families of $4$-manifolds with $b_1(X) > 0$ and without any orientability assumption on $H^+$.

\subsection{Applications}

The remainder of the paper is concerned with applications of our main formula, Theorem \ref{thm:cohomologyformula}, which we will now discuss.

Assume once again that $E \to B$ is a spin$^c$ family of $4$-manifolds with fibres diffeomorphic to $X$ and for simplicity we continue to assume that $b_1(X) = 0$. Recall that $D \in K^0(B)$ denote the virtual index bundle of the family of spin$^c$ Dirac operators determined by $E$ and that $H^+ \in KO^0(B)$ denotes the bundle of harmonic self-dual $2$-forms on the fibres of $E \to B$. Let $c(D) = 1 + c_1(D) + c_2(D) + \cdots $ denote the total Chern class of $D$. For each $j \ge 0$, define the {\em $j$-th Segre class} $s_j(D) \in H^{2j}(B ; \mathbb{Z})$ of $D$ as follows. Letting $s(D) = 1 + s_1(D) + s_2(D) + \cdots $ denote the total Segre class of $D$, then $s(D)$ is defined by the equation $c(D)s(D) = 1$.

Let $\phi$ denote a chamber for the families Seiberg-Witten equations of the family $(E , \mathfrak{s}_E)$. Consider the mod $2$ reductions of the families Seiberg-Witten invariants:
\[
SW_m^{\mathbb{Z}_2}( E , \mathfrak{s}_E , \phi ) \in H^{2m-(2d-b^+-1)}(B ; \mathbb{Z}_2).
\]
Then as a consequence of Theorem \ref{thm:cohomologyformula}, we compute the Steenrod squares of the mod $2$ families Seiberg-Witten invariants:

\begin{theorem}
The Steenrod squares of the mod $2$ families Seiberg-Witten invariants are given by:
\begin{equation*}
\begin{aligned}
& Sq^{2j}( SW_m^{\mathbb{Z}_2}(E,\mathfrak{s}_E,\phi) ) = \\
& \quad \quad \sum_{l=0}^j \sum_{k=0}^{j-l} \binom{d-1-m+l+k}{l} s_k(D) w_{2j-2l-2k}(H^+) SW_{m+l}^{\mathbb{Z}_2}(E,\mathfrak{s}_E , \phi), \\
& Sq^{2j+1}( SW_m^{\mathbb{Z}_2}(E,\mathfrak{s}_E,\phi) ) = \\
& \quad \quad \sum_{l=0}^j \sum_{k=0}^{j-l} \binom{d-1-m+l+k}{l} s_k(D) w_{2j-2l-2k+1}(H^+) SW_{m+l}^{\mathbb{Z}_2}(E,\mathfrak{s}_E , \phi).
\end{aligned}
\end{equation*}
for all $m,j \ge 0$.
\end{theorem}
These formulas imply some surprising mod $2$ relations between the families Seiberg-Witten invariants and the characteristic classes of $D$, $H^+$. For example, in the case of $Sq^2$, we get:
\[
Sq^2( SW_m^{\mathbb{Z}_2}(E , \mathfrak{s}_E , \phi) ) = (d+m)SW_{m+1}^{\mathbb{Z}_2}( E , \mathfrak{s}_E , \phi) + (c_1(D) + w_2(H^+))SW_m^{\mathbb{Z}_2}(E , \mathfrak{s}_E , \phi).
\]
Suppose that $b^+ = 2p+1$ is odd and set $m = d-p-1$. Then $SW_m^{\mathbb{Z}_b}( E , \mathfrak{s}_E , \phi) \in H^0(B ; \mathbb{Z}_2) \cong \mathbb{Z}_2$ is the mod $2$ reduction of the ordinary Seiberg-Witten invariant $SW(X , \mathfrak{s}_X)$ of the $4$-manifold $X$ with respect to the spin$^c$-structure $\mathfrak{s}|_X = \mathfrak{s}_E |_X$. In this case the above equation reduces to:
\[
(p+1) SW_{m+1}^{\mathbb{Z}_2}(E , \mathfrak{s}_E , \phi) = (c_1(D) + w_2(H^+)) SW(X , \mathfrak{s}_X)  \in H^2(B ; \mathbb{Z}_2).
\]
If $p$ is even, this gives a formula for $SW_{m+1}^{\mathbb{Z}_2}(E , \mathfrak{s}_E , \phi)$ in terms of $SW(X , \mathfrak{s}_X)$, $c_1(D)$ and $w_2(H^+)$. On the other hand if $p$ is odd (so $b^+ = 3 \; (mod \; 4)$), this says that $(c_1(D) + w_2(H^+))SW(X , \mathfrak{s}_X) = 0 \in H^2(B ; \mathbb{Z}_2)$, which leads to the following interesting consequence:
\begin{corollary}\label{cor:w2}
Let $(X , \mathfrak{s}_X)$ be a compact smooth spin$^c$ $4$-manifold with $b_1(X) = 0$ and $b^+(X) = 3 \; ({\rm mod} \; 4)$. Then:
\begin{itemize}
\item{If there exists a spin$^c$ family $E \to B$ with fibre $(X , \mathfrak{s}_X)$ and $c_1(D) \neq w_2(H^+) \; ({\rm mod} \; 2)$, then $SW(X , \mathfrak{s}_X)$ is even.}
\item{If $SW(X , \mathfrak{s}_X)$ is odd, then $c_1(D) = w_2(H^+) \; ({\rm mod} \; 2)$ for any spin$^c$ family $E \to B$ with fibre $(X , \mathfrak{s}_X)$.}
\end{itemize}
\end{corollary}
In particular, we see that the mod $2$ Seiberg-Witten invariants of $X$ can be used to obstruct the existence of certain spin$^c$ families with fibres diffeomorphic to $X$. We give an application of this result to $K3$ surfaces:

\begin{theorem}
\label{theo: family of K3 over torus intro}
There exists a continuous family $E$ of $K3$ surfaces over the two torus $T^2$
satisfying the following conditions:
\begin{itemize}
\item The total space of $E$ is smoothable as a manifold.
\item The restriction of $E$ to any $1$-dimensional submanifold of $T^2$ is smoothable as a family.
\item However, $E$ is not smoothable as a family.
\end{itemize}
\end{theorem}

\begin{remark} Here are three remarks on Theorem~\ref{theo: family of K3 over torus intro}.
\begin{enumerate}
\item The precise meaning of smoothablity as a family will be given in Definition~\ref{defi: smoothable as families}.
\item For comparison, see Theorem~1.4 of \cite{kkn}, in which other non-smoothable families have been detected by a different technique based on 10/8-type inequalities.
\item The proofs that $E$ satisfies the first and second conditions are based on \cite{kkn} and  \cite{mat} respectively. To prove that $E$ satisfies the third condition, we shall use Corollary~\ref{cor:w2}.
\end{enumerate}
\end{remark}

Theorem~\ref{theo: family of K3 over torus intro} immediately implies that the fundamental group of the homotopy fiber of the natural map $BDiff(K3) \to BHomeo(K3)$ is non-trivial.
Recall that this homotopy fiber is homotopy equivalent to the homotopy quotient 
\[
Homeo(K3)\sslash Diff(K3) := \left(EDiff(K3) \times Homeo(K3)\right)/Diff(K3).
\]
Therefore we have:

\begin{corollary}
\label{cor: family of K3 over torus intro}
$\pi_{1}\left(Homeo(K3)\sslash Diff(K3)\right) \neq 0$.
\end{corollary}

\begin{remark} Here are two remarks on Corollary~\ref{cor: family of K3 over torus intro}.
\begin{enumerate}
\item By the long exact sequence of homotopy groups it follows that $\pi_0( Diff(K3)) \to \pi_0(Homeo(K3))$ is not injective or $\pi_1( Diff(K3)) \to \pi_1(Homeo(K3))$ is not surjective. Building upon this result, we have shown in subsequent work that $\pi_1( Diff(K3)) \to \pi_1(Homeo(K3))$ is not surjective \cite{bako}. It follows that there exists a topological family of $K3$ surfaces over the $2$-sphere which can not be made into a smooth family. To the best of our knowledge, this is the first known example of a $4$-manifold $M$ where $\pi_1( Diff(M) ) \to \pi_1(Homeo(M))$ is not surjective.
\item It is interesting to compare Corollary~\ref{cor: family of K3 over torus intro} with Theorem~1.9 of \cite{kkn}, which states that
\[
\pi_{1}\left(Homeo(K3\#S^{2}\times S^{2})\sslash Diff(K3\#S^{2}\times S^{2})\right) \neq 0.
\]
\end{enumerate}
\end{remark}

Our next application of Theorem \ref{thm:cohomologyformula} is to give a simple new proof of the wall crossing formula for the families Seiebrg-Witten invariants. The families wall crossing formula was originally proven in \cite{liliu} using parametrised Kuranishi models and obstruction bundles. Using the technique of finite dimensional approximations of the Seiberg-Witten equations allows us to bypass these technicalities. Here we will state the result for spin$^c$ families $E \to B$ with fibres diffeomorphic to a $4$-manifold with $b_1(X) = 0$, however our general result also covers the case that $b_1(X) > 0$.

\begin{theorem}
Let $(E , \mathfrak{s}_E)$ be a spin$^c$ family with fibres diffeomorphic to a $4$-manifold $X$ with $b_1(X) = 0$. Let $D$ denote the virtual index bundle of the family of spin$^c$ Dirac operators determined by $E$ and set $d = rank_{\mathbb{C}}(D)$. Let $\phi, \psi$ be chambers of the families Seiberg-Witten equations for $(E , \mathfrak{s}_E)$. Then:
\[
SW_m(E , \mathfrak{s}_E,\phi)-SW_m(E , \mathfrak{s}_E,\psi) = \begin{cases} 0 & \text{if} \; \; m < d-1, \\ Obs(\phi,\psi) \smallsmile s_{m-(d-1)}(D) & \text{if} \; \; m \ge d-1, \end{cases}
\]
where $Obs(\phi , \psi) \in H^{b^+(X)-1}(B ; \mathbb{Z}_{w_1(H^+)})$ is the primary difference class of $\phi,\psi$ (see \textsection \ref{sec:relativeobstruction}) and $s_j(D)$ is the $j$-th Segre class of $D$.
\end{theorem}

In Subsection \ref{sec:uwcf}, we show how to recover the wall crossing formula for the ordinary Seiberg-Witten invariant of a $4$-manifold $X$ with $b^+(X)=1$ and $b_1(X)$ even as a special case of the families wall crossing formula. Note that even in this case the families formalism is still relevant: the parameter space of the family is the Jacobian torus $\mathbb{T} = H^1(X ; \mathbb{R}/\mathbb{Z})$.

Our final application concerns certain divisibility conditions of the families Seiberg-Witten invariants that arise from $K$-theoretic considerations. Recall from Subsection \ref{sec:introfsw} that the families Seiberg-Witten invariants of a spin$^c$ family $(E , \mathfrak{s}_E)$ over $B$ with fibres diffeomorphic to $X$ are certain cohomology classes $SW_m(E , \mathfrak{s}_E , \phi)$ valued in the cohomology of $B$:
\begin{equation}\label{equ:defsw}
SW_m(E , \mathfrak{s}_E , \phi) = (\pi_{\mathcal{M}})_*(  c_1(\mathcal{L})^m ) \in H^{2m-(2d-b^+-1)}(B ; \mathbb{Z}),
\end{equation}
where for simplicity we are assuming in this introduction that $b_1(X)=0$ and $H^+$ is orientable. Here $\mathcal{M}$ is the families Seiberg-Witten moduli space and $\pi_{\mathcal{M}} : \mathcal{M} \to B$ is the natural map to $B$. In Section \ref{sec:ktheory}, we introduce a new set of invariants of the family $(E , \mathfrak{s}_E)$, the {\em $K$-theoretic Seiberg-Witten invariants} $SW^K_m(E,\mathfrak{s}_E , \phi) \in K^{b^+-1}(B)$. To define them, assume that $H^+$ admits a spin$^c$-structure. Then one can show that $\pi_\mathcal{M} : \mathcal{M} \to B$ is relatively spin$^c$ (see Remark \ref{rem:kty}) and imitate (\ref{equ:defsw}) in $K$-theory:
\[
SW^K_m(E , \mathfrak{s}_E , \phi) = (\pi_{\mathcal{M}})^K_*( \mathcal{L}^m) \in K^{b^+-1}(B),
\]
where $(\pi_{\mathcal{M}})^K_*$ denotes the $K$-theoretic pushforward map induced by $\pi_\mathcal{M}$. The proof of Theorem \ref{thm:cohomologyformula} carries over to the $K$-theoretic setting, and thus we have:

\begin{theorem}
Let $(E , \mathfrak{s}_E)$ be a spin$^c$ family over $B$. Let $f : S_{V,U} \to S_{V',U'}$ be a finite dimensional approximation of the Seiberg-Witten monopole map, as described above. Let $[\phi]$ be a chamber of the families Seiberg-Witten equations represented by a section $\phi : B \to U' \setminus U$. Then:
\[
SW_m^K( E , \mathfrak{s}_E , \phi) = (\pi_Y)^K_*( \xi^m \smallsmile f^* \phi^K_*(1) ),
\]
where $\xi \in K^0_{S^1}(pt) = R[S^1]$ corresponds to the standard $1$-dimensional representation of $S^1$.
\end{theorem}

Then by an application of Grothendieck-Riemann-Roch, we obtain a formula relating the Chern character of the $K$-theoretic Seiberg-Witten invariants to the cohomological Seiberg-Witten invariants. To state the result, we introduce the following notation: for a complex vector bundle $V$ over $B$, let $Td^{S^1}(V) \in H^*_{S^1}(B ; \mathbb{Q}) \cong H^*(B ; \mathbb{Q}) \otimes H^*_{S^1}(pt ; \mathbb{Q})$ denote the $S^1$-equivariant Todd class of $V$, where $S^1$ acts on $V$ by scalar multiplication. This can be expanded as a formal power series in $x$, where $x$ is the standard generator of $H^2_{S^1}(pt ; \mathbb{Z})$:
\[
Td^{S^1}(V) = \sum_{j \ge 0} Td_j(V) x^j,
\]
for some characteristic classes $Td_j(V) \in H^{ev}(B ; \mathbb{Q})$.

\begin{theorem}
Let $\kappa \in H^2(B ; \mathbb{Z})$ be the first Chern class of the spin$^c$ line bundle associated to $H^+$. Then the $K$-theoretic and cohomological Seiberg-Witten invariants are related by:
\begin{equation*}
\begin{aligned}
& Ch( SW_m^K(E , \mathfrak{s}_E,\phi)) = \\
& \quad \quad e^{-\kappa/2}\hat{A}(H^+)^{-1} \sum_{j \ge 0} Td_j(D) \sum_{k \ge 0} \frac{m^k}{k!}SW_{j+k}(E , \mathfrak{s}_E,\phi) \in H^*(B ; \mathbb{Q})
\end{aligned}
\end{equation*}
\end{theorem}

This formula can be used to extract certain divisibility properties of the families Seiberg-Witten invariants. For example, we prove the following:

\begin{corollary}
Let $(E , \mathfrak{s}_E)$ be as above and suppose that $B = S^{2r}$ is an even dimensional sphere with $r \ge 5$. Suppose that fibres of $E$ are diffeomorphic to a $4$-manifold $X$ with $b_1(X)=0$ and $b^+(X) = 2p+1$ odd. Suppose that $n = r+d-p-1 \ge 0$. Then $SW_{n}(E,\mathfrak{s}_E,\phi) \in H^{2r}(S^{2r} ; \mathbb{Z}) \cong \mathbb{Z}$ is divisible by the denominators of $a_{p-r,l}$, for $l = 0,1, \dots n$, where the rational numbers $a_{p,l}$ are defined to be the coefficients of the Taylor expansion:
\[
log(1-y)^p = \sum_{l=0}^\infty a_{p,l} y^{p+l}.
\]
\end{corollary}

\subsection{Structure of paper}
A brief outline of the contents of this paper is as follows. In Subsection \ref{sec:familiesmonopole} we introduce the notion an infinite dimensional families monopole map, of which the families Seiberg-Witten monopole map is a special case. We define chambers and families Seiberg-Witten invariants for such maps. In Subsection \ref{sec:fsw} we introduce the notion of a finite dimensional monopole map and define chambers and Seiberg-Witten invariants for them. In Subsection \ref{sec:fda} we show that any infinite dimensional families monopole map has a finite dimensional approximation and in Subsection \ref{sec:equalitysw} we prove that the families Seiberg-Witten invariants are preserved under the process of taking a finite dimensional approximation (provided the approximation is ``sufficiently large"). In Section \ref{sec:cohom} we prove our main formula which expresses the families Seiberg-Witten invariants of a finite dimensional monopole map in terms of purely cohomological operations. The remaining sections of the paper are concerned with applications of the cohomological formula. In Section \ref{sec:steenrod} we compute the Steenrod powers of the families Seiberg-Witten invariants and in Subsection \ref{sec:k3} we give an application of these results to $K3$ surfaces. In Section \ref{sec:wallcrossing} we give a new proof of the wall crossing formula for the families Seiberg-Witten invariants. In Section \ref{sec:ktheory}, we introduce the $K$-theoretic families Seiberg-Witten invariants and compute their Chern character in terms of the cohomological Seiberg-Witten invariants. We also prove a wall crossing formula for these invariants.\\

\noindent{\bf Acknowledgments}. From Nobuo Iida, the second author heard Mikio Furuta's argument to relate the Seiberg-Witten moduli space to the corresponding moduli space of a finite dimensional approximation of the monopole map in the unparameterised setting. This helped the authors to consider its families version. 
After the version 1 of this paper had appeared on arXiv, the second author uploaded the paper \cite{kkn} with Tsuyoshi Kato and Nobuhiro Nakamura on arXiv.
The work on \cite{kkn} and a comment by Nakamura on the version 1 of this paper prompted the authors to conceive the idea of Theorem~\ref{theo: family of K3 over torus intro} and Corollary~\ref{cor: family of K3 over torus intro}, which are added after \cite{kkn} appeared on arXiv.
The authors are grateful to Kato and Nakamura for these.
D. Baraglia was financially supported by the Australian Research Council Discovery Early Career Researcher Award DE160100024 and Australian Research Council Discovery Project DP170101054. H. Konno was supported by JSPS KAKENHI Grant Number 16J05569 and the Program for Leading Graduate Schools, MEXT, Japan.

\section{The families monopole map and finite dimensional approximations}\label{sec:fm}

We now turn to the construction of the families monopole map and its finite dimensional approximations. While the primary motivation for this construction is the Bauer-Furuta and Seiberg-Witten invariants for a family of smooth $4$-manifolds, the construction holds in a more general setting which we outline below.

\subsection{Families monopole map and families Seiberg-Witten invariants}\label{sec:familiesmonopole}

In this subsection we introduce a general notion of a families monopole map and its associated families Seiberg-Witten invariants. In the case of the families Seiberg-Witten monopole map of a spin$^c$ family of $4$-manifolds, we recover the notion of the families Seiberg-Witten invariants. In the following subsections we will show how to construct a finite dimensional approximation of the families monopole map and recover the families Seiberg-Witten invariants from the approximation.

The general setting for the families monopole map is as follows. Let $B$ be a compact smooth manifold. Let $\mathbb{V}, \mathbb{W} \to B$ be infinite dimensional separable Hilbert space bundles, each a direct sum of a real and a complex Hilbert space bundle of infinite dimension:
\[
\mathbb{V} = \mathbb{V}_{\mathbb{C}} \oplus \mathbb{V}_{\mathbb{R}}, \quad \mathbb{W} = \mathbb{W}_{\mathbb{C}} \oplus \mathbb{W}_{\mathbb{R}}
\]
where $\mathbb{V}_{\mathbb{C}}, \mathbb{W}_{\mathbb{C}}$ are complex infinite dimensional Hilbert bundles and $\mathbb{V}_{\mathbb{R}}, \mathbb{W}_{\mathbb{R}}$ are real infinite dimensional Hilbert bundles.

Let $S^1$ act on $\mathbb{V}_{\mathbb{C}}, \mathbb{W}_{\mathbb{C}}$ by scalar multiplication and act trivially on $\mathbb{V}_{\mathbb{R}}, \mathbb{W}_{\mathbb{R}}$. Let $f : \mathbb{V} \to \mathbb{W}$ be an $S^1$-equivariant bundle map. We will say that $f$ is a {\em families monopole map} (or {\em monopole map} for short) if it satisfies the following conditions:
\begin{itemize}
\item[(M1)]{$f = l + c$, where $l : \mathbb{V} \to \mathbb{W}$ is an $S^1$-equivariant fibrewise linear Fredholm map and $c : \mathbb{V} \to \mathbb{W}$ is an $S^1$-equivariant family of compact operators.}
\item[(M2)]{$l$ and $c$ are both smooth as maps between Hilbert manifolds (hence $f$ itself is smooth)}
\item[(M3)]{$f$ satisfies a boundedness property: the preimage under $f$ of a disc bundle in $\mathbb{W}$ is contained in a disc bundle in $\mathbb{V}$.}
\item[(M4)]{By $S^1$-equivariance we can write $l = l_{\mathbb{C}} \oplus l_{\mathbb{R}}$, where $l_{\mathbb{C}} : \mathbb{V}_{\mathbb{C}} \to \mathbb{W}_{\mathbb{C}}$, $l_{\mathbb{R}} : \mathbb{V}_{\mathbb{R}} \to \mathbb{W}_{\mathbb{R}}$. We assume that $l_{\mathbb{R}}$ is injective.}
\item[(M5)]{We assume that $c(v) = 0$ for all $v \in \mathbb{V}_{\mathbb{R}}$.}
\end{itemize}

\begin{example}\label{ex:b1=0}
The following example is the main motivation for the above notion of a monopole map. Suppose we have a family $\pi : E \to B$ of $4$-manifolds over $B$ with fibres diffeomorphic to $X$. Here we discuss the simple case that $b_1(X) = 0$. The case $b_1(X)>0$ will be addressed in Example \ref{ex:b1>0}. Assume that the vertical tangent bundle $E$ is equipped with a spin$^c$-structure $\mathfrak{s}_E$. Choose a smoothly varying fibrewise metric $g = \{g_b\}_{b\in B}$ on $E$. Fix a smoothly varying family of $U(1)$-connections $A = \{ A_b \}_{b \in B}$ for the determinant line of the spin$^c$-structure (for instance one could choose a globally defined connection on the total space of $E$ and define $A_b$ as the restriction of this connection to the fibre over $b$). Fix an integer $k >2$. We define the following Hilbert bundles over $B$:
\begin{equation*}
\begin{aligned}
&\mathbb{V}_{\mathbb{C}} = L^2_k( S^+), &&\quad \mathbb{V}_{\mathbb{R}} = L^2_k(\wedge^1 T^*X), \\
&\mathbb{W}_{\mathbb{C}} = L^2_{k-1}(S^-), &&\quad \mathbb{W}_{\mathbb{R}} = L^2_{k-1}(\wedge^+ T^*X) \oplus L^2_{k-1}(\mathbb{R})_0
\end{aligned}
\end{equation*}
where $S^\pm$ denote the spinor bundles, $L^2_k(E)$ denotes the Sobolev space of $L^2_k$ sections of $E$ and $L^2_{k-1}(\mathbb{R})_0$ denotes the subspace of sections $f \in L^2_{k-1}(\mathbb{R})$ satisfying $\int_X f dvol_X = 0$.

We define the families Seiberg-Witten monopole map $f : \mathbb{V} \to \mathbb{W}$ to be given by
\[
f( \psi , a ) = ( D_{A+ia} \psi , -iF^+_{A+ia} +i \sigma(\psi) + iF^+_A , d^*a)
\]
where $D_{A+ia}$ denotes the spin$^c$ Dirac operator associated to $A+ia$ and $\sigma(\psi)$ is the quadratic spinor term in the Seiberg-Witten equations. The second term of $f$ is usually defined to be $-iF^+_{A+ia} + i\sigma(\psi)$, but we have added a harmless term $iF^+_A$ which merely shifts the level sets of $f$. The reason for including this extra term is twofold: first it makes $f$ satisfy (M5). Secondly it makes it easier to describe the chamber structure of the families Seiberg-Witten invariants associated to $f$. Note that $f$ can be re-written as:
\[
f(\psi , a) = (D_{A+ia} \psi , d^+a +i\sigma(\psi) , d^*a).
\]
Then $f =  l + c$, where
\[
l(\psi , a ) = ( D_A \psi , d^+a , d^*a), \quad c(\psi , a) = (\frac{i}{2} a \cdot \psi , i\sigma(\psi) , 0).
\]
So $l_{\mathbb{C}}(\psi) = D_A \psi$, $l_{\mathbb{R}}(a) = (d^+a , d^*a)$. It follows that $l_{\mathbb{R}}$ is injective since we assumed that $b_1(X) = 0$. So $f$ satisfies (M4). Clearly $c(0,a) = 0$, so $f$ also satisfies (M5). Conditions (M1), (M2) are clearly satisfied as well. Condition $(M3)$ was proven in \cite{bafu} for a single $4$-manifold. The same estimates easily extend to the case of a smoothly varying family over a compact base. Hence the families Seiberg-Witten monopole map satisfies conditions (M1)-(M5).
\end{example}

\begin{definition}
Let $f : \mathbb{V} \to \mathbb{W}$ be a monopole map. A {\em chamber} for $f$ is a homotopy class of section $\eta : B \to \mathbb{W}_{\mathbb{R}} \setminus l_{\mathbb{R}}(\mathbb{V}_{\mathbb{R}})$. We let $\mathcal{CH}(f)$ denote the set of chambers for $f$ (note that $\mathcal{CH}(f)$ could be empty).
\end{definition}

Let $f$ be a monopole map. Recall that $l_{\mathbb{R}}$ is Fredholm and injective by (M4). Hence $W_{0,\mathbb{R}} = Ker( l_{\mathbb{R}}^*)$ is a finite rank subbundle of $\mathbb{W}_{\mathbb{R}}$ which can be identified with $coker(l_{\mathbb{R}})$. Clearly $\mathbb{W}_{\mathbb{R}} \setminus l_{\mathbb{R}}( \mathbb{V}_{\mathbb{R}})$ admits a fibrewise deformation retraction to $W_{0,\mathbb{R}} \setminus \{0\}$, where $\{0\}$ denotes the zero section. Thus a chamber for $f$ is equivalent to a homotopy class of non-vanishing section of $W_{0,\mathbb{R}}$. By rescaling, we have that for any fixed $\delta > 0$, any chamber can be represented by a section $\eta : B \to \mathbb{W}_{\mathbb{R}}$ whose norm satisfies $|| \eta || < \delta$ (pointwise with respect to $B$).

In order to prove our main result we need to make two further assumptions on $f$:
\begin{itemize}
\item[(M6)]{Let $\delta > 0$ be given. Then any chamber of $f$ can be represented by a section $\eta : B \to \mathbb{W}_{\mathbb{R}}$ such that $||\eta || < \delta$ pointwise with respect to $B$ and such that $\eta$ is transverse to $f$. Furthermore any two such $\eta$ representing the same chamber can be joined by a path $\eta(t)$ such that $||\eta(t)|| < \delta$ for all $t$ and $\eta(t) : [0,1] \times B \to \mathbb{W}_{\mathbb{R}} \setminus l_\mathbb{R}( \mathbb{V}_\mathbb{R})$ is transverse to $f \circ pr_{\mathbb{V}} : [0,1] \times \mathbb{V} \to \mathbb{W}$, where $pr_{\mathbb{V}} : [0,1] \times \mathbb{V} \to \mathbb{V}$ is projection to the second factor.}
\item[(M7)]{There exists a smooth $\mathbb{R}$-bilinear map $q : \mathbb{V} \times_B \mathbb{V} \to \mathbb{W}$ such that $c(v) = q(v,v)$.}
\end{itemize}

\begin{example}
In the case that $f$ is the Seiberg-Witten monopole map of a family of $4$-manifolds with $b_1(X) = 0$, then (M6)-(M7) are satisfied. (M6) follows from the existence of regular perturbations of the families Seiberg-Witten equations (this is well known in the case of a single $4$-manifold and easily adapts to the case of families). (M7) is easy to verify by looking at the Seiberg-Witten equations.
\end{example}

\begin{example}\label{ex:b1>0}
We now explain how the construction of the families Seiberg-Witten monopole map extends to the case of a spin$^c$-family of $4$-manifolds $\pi : E \to B$ whose fibres are diffeomorphic to a $4$-manifold $X$ with $b_1(X) > 0$, equipped with a section $x : B \to E$. Let $\mathcal{H}^1(\mathbb{R}) \to B$ denote the vector bundle whose fibre $\mathcal{H}^1_b(\mathbb{R})$ over $b \in B$ is the space of harmonic $1$-forms on the fibre $E_b = \pi^{-1}(b)$. Then on the total space of $\mathcal{H}^1(\mathbb{R})$ we define the following Hilbert bundles:
\begin{equation*}
\begin{aligned}
&\mathbb{V}_{\mathbb{C}} = L^2_k( S^+), &&\quad \mathbb{V}_{\mathbb{R}} = L^2_k(\wedge^1 T^*X)_0, \\
&\mathbb{W}_{\mathbb{C}} = L^2_{k-1}(S^-), &&\quad \mathbb{W}_{\mathbb{R}} = L^2_{k-1}(\wedge^+ T^*X) \oplus L^2_{k-1}(\mathbb{R})_0
\end{aligned}
\end{equation*}
where $L^2_{k-1}(\mathbb{R})_0$ is defined as in Example \ref{ex:b1=0} and $L^2_k(\wedge^1 T^*X)_0$ is the subbundle of $L^2_k(\wedge^1 T^*X)$ consisting of $1$-forms $L^2$-orthogonal to the finite dimensional subbundle of harmonic $1$-forms. Define $f : \mathbb{V}_\mathbb{C} \oplus \mathbb{V}_\mathbb{R} \to \mathbb{W}_\mathbb{C} \oplus \mathbb{W}_\mathbb{R}$ over $\theta \in \mathcal{H}^1(\mathbb{R})$ by:
\[
f(\theta , ( \psi , a)) = (D_{A+ia+i\theta}(\psi) , d^+a +i\sigma(\psi) , d^*a)
\]
where $A$, $\sigma(\psi)$ are as in Example \ref{ex:b1=0}. 

The map $f$ is a combination of the Seiberg-Witten equations and the gauge fixing term $d^*a = 0$. To obtain the (families) Seiberg-Witten moduli space we still need to factor out by harmonic gauge transformations. This is done as follows. Let $\mathcal{H}^1(\mathbb{Z}) \to B$ denote the locally constant bundle of groups whose fibre $\mathcal{H}^1_b(\mathbb{Z})$ over $b \in B$ is the space of harmonic $1$-forms on $E_b$ with integral periods. To each $\omega \in \mathcal{H}^1(\mathbb{Z})$ we can define a unique harmonic map $f_\omega : E_b \to U(1)$ such that $f^{-1}_\omega df_\omega = 2\pi i \omega$ and $f_\omega(x(b)) = 1$, namely $f_\omega$ is given by
\[
f_\omega(y) = {\rm exp}\left( 2\pi i \int_{x(b)}^y \omega \right).
\]
We let $\mathcal{H}^1(\mathbb{Z})$ act fibrewise on $\mathcal{H}^1(\mathbb{R})$ by:
\[
\omega \cdot \theta = \theta + 4\pi \omega.
\]
Let $J = \mathcal{H}^1(\mathbb{R})/\mathcal{H}^1(\mathbb{Z})$ be the quotient with respect to this action. Then $J$ is a torus bundle over $B$ (in fact, $J$ is isomorphic to the bundle of Jacobians associated to $E \to B$, where the Jacobian $J_b$ of the fibre $E_b$ is defined to be the space of isomorphism classes of flat topologically trivial unitary line bundles on $E_b$).

The action of $\mathcal{H}^1(\mathbb{Z})$ on $\mathcal{H}^1(\mathbb{R})$ just defined lifts to fibrewise linear actions on the total spaces of $\mathbb{V} = \mathbb{V}_\mathbb{C} \oplus \mathbb{V}_\mathbb{R}$ and $\mathbb{W} = \mathbb{W}_\mathbb{C} \oplus \mathbb{W}_\mathbb{R}$ as follows:
\begin{equation*}
\begin{aligned}
\omega \cdot (\theta , (\psi , a)) &= (\theta + 4\pi \omega , ( f_\omega^{-1} \psi , a  )), && (\psi , a) \in \mathbb{V}_\theta \\ 
\omega \cdot (\theta , (\phi , \nu , g)) &= (\theta + 4\pi \omega , ( f_\omega^{-1} \phi , \nu , g )), && (\phi , \nu , g) \in \mathbb{W}_\theta.
\end{aligned}
\end{equation*}
This is precisely the action of the group of harmonic gauge transformations (which are normalised to equal the identity at $x(b)$). Note that $\mathcal{H}^1(\mathbb{Z})$ acts freely on $\mathcal{H}^1(\mathbb{R})$ with quotient space $J$ and so the vector bundles $\mathbb{V}, \mathbb{W}$ descend to vector bundles $\mathbb{V}/\mathcal{H}^1(\mathbb{Z}), \mathbb{W}/\mathcal{H}^1(\mathbb{Z})$ on $J$. Moreover $f$ is $\mathcal{H}^1(\mathbb{Z}) \times S^1$-equivariant so descends to an $S^1$-equivariant map
\[
f : \mathbb{V}/\mathcal{H}^1(\mathbb{Z}) \to \mathbb{W}/\mathcal{H}^1(\mathbb{Z})
\]
of Hilbert bundles over $J$. This is our desired families Seiberg-Witten monopole map in the case that $b_1(X)>0$. One can verify that $f$ satisfies (M1)-(M7) in much the same way as was done in the $b_1(X)=0$ case in Example \ref{ex:b1=0}.
\end{example}

\begin{lemma}\label{lem:compact}
Let $f$ satisfy (M1)-(M7). Then for any $v \in \mathbb{V}$, the linear map $\mathbb{V} \to \mathbb{W}$ given by $u \mapsto q(v,u)$ is compact.
\end{lemma}
\begin{proof}
Note that since $c(v) = q(v,v)$, it follows that:
\begin{equation}\label{equ:polarisation}
q(v,u) = \frac{1}{2}\left( c(u+v) - c(u) - c(v) \right).
\end{equation}
Fix $v \in \mathbb{V}_b$ and let $\{ u_i \}$ be a bounded sequence of elements in $\mathbb{V}_b$. Then by compactness of $c$, we can take a subsequence such that $\{ c(u_i) \}$ converges. Taking a further subsequence, we can also assume that $\{ c(v+u_i) \}$ converges. Then by Equation (\ref{equ:polarisation}), the sequence $q(v,u_i)$ also converges.
\end{proof}

Let $f$ be a monopole map satisfying (M1)-(M7). Let $\eta : B \to \mathbb{W}_{\mathbb{R}}$ be a representative of a chamber such that $\eta$ is transverse to $f$. By (M6) such an $\eta$ exists for any chamber and we can take $||\eta||$ to be as small as desired. By transversality $\widetilde{\mathcal{M}}_\eta = f^{-1}(\eta)$ is a smooth manifold equipped with a free $S^1$-action and an $S^1$-invariant projection map $\pi_{\widetilde{\mathcal{M}}_\eta} : \widetilde{\mathcal{M}}_\eta \to B$. We claim that $\widetilde{\mathcal{M}}_\eta$ is finite-dimensional and in fact $dim(\widetilde{\mathcal{M}}_\eta) = dim(B) + ind(l)$. To see this, note that the linearisation of $f(v) = \eta$ around $v$ (in the fibre directions) is given by $df_v(\dot{v}) = l(\dot{v}) + 2q(v , \dot{v})$. By Lemma \ref{lem:compact}, $2q(v , \, \cdot \,  )$ is a compact operator. It follows that the total derivative of $f : \mathbb{V} \to \mathbb{W}$ at any point is Fredholm with index equal to $ind(l)$. Then since $f$ is transverse to $\eta$, $dim(\widetilde{\mathcal{M}}_\eta) = dim( f^{-1}(\eta(B))) = dim(B) + ind(l)$. Moreover, since $\widetilde{\mathcal{M}}_\eta$ is finite-dimensional and contained in a bounded subset of $\mathbb{V}$ (by (M3)), it follows that $\widetilde{\mathcal{M}}_\eta$ is compact.

Since $S^1$ acts freely on $\widetilde{\mathcal{M}}_\eta$, the quotient $\mathcal{M}_\eta = \widetilde{\mathcal{M}}_\eta/S^1$ is a compact smooth manifold of dimension $dim(B) + ind(l)-1$. Moreover the map $\pi_{\widetilde{\mathcal{M}}_\eta}: \widetilde{\mathcal{M}}_\eta \to B$ is $S^1$-invariant and so descends to a map $\pi_{\mathcal{M}_\eta} : \mathcal{M}_\eta \to B$. Let $\mathcal{L} \to \mathcal{M}_\eta$ denote the complex line bundle associated to the principal $S^1$-bundle $\widetilde{\mathcal{M}}_\eta \to \mathcal{M}_\eta$.

\begin{lemma}\label{lem:orientation}
Let $Ind(l_{\mathbb{R}}) \in KO^0(B)$ denote the real $K$-theory class of the family of real Fredholm operators $l_\mathbb{R} : \mathbb{V}_\mathbb{R} \to \mathbb{W}_\mathbb{R}$. Then there is a natural isomorphism $\psi : det( T\mathcal{M}_\eta \oplus \pi^*_{\mathcal{M}_\eta}(TB)  ) \cong \pi^*_{\mathcal{M}_\eta}( det( Ind(l_\mathbb{R})  )$.
\end{lemma}
\begin{proof}
Using the $S^1$-action we have that $det(T\widetilde{\mathcal{M}}_\eta)$ descends to $\widetilde{M}_\eta$ and one easily sees that $det(T\widetilde{\mathcal{M}}_\eta)/S^1 = det(T\mathcal{M}_\eta)$. Therefore it suffices to show that there is an $S^1$-equivariant isomorphism $det(T\widetilde{\mathcal{M}}_\eta \oplus \pi^*_{\widetilde{\mathcal{M}}_\eta}(TB) ) \cong \pi_{\widetilde{\mathcal{M}}_\eta}^*( det( Ind(l_\mathbb{R}) ))$. However we have seen that the linearisation of $f$ around any point of $\widetilde{\mathcal{M}}_\eta$ differs from $l$ by a compact operator. In particular we can write down an $S^1$-equivariant homotopy through Fredholm operators from the family $\{df_v\}_{v \in \widetilde{\mathcal{M}}}$ to the family $\{ l \}_{v \in \widetilde{\mathcal{M}}}$. From this it follows that $det(T\widetilde{\mathcal{M}}_\eta)$ is $S^1$-equivariantly isomorphic to the pullback of $det(Ind(l) \oplus det(TB))$, where $Ind(l) \in KO^0(B)$ is the families index of $l$. But $l = l_\mathbb{R} + l_\mathbb{C}$, so we get a decomposition $Ind(l) = Ind(l_\mathbb{C}) + Ind(l_\mathbb{R})$. Moreover $l_\mathbb{C}$ is a complex Fredholm operator, so $det(Ind(l_\mathbb{C}))$ is trivial and $det(Ind(l)) = det(Ind(l_\mathbb{R}))$.
\end{proof}

Let $w_1 = w_1(Ind(l_\mathbb{R})) \in H^1(B ; \mathbb{Z}_2)$ be the first Stiefel-Whitney class of $Ind(l_\mathbb{R})$. Then from the above lemma, $w_1$ is the relative orientation class of $\pi_{\mathcal{M}_\eta} : \mathcal{M}_\eta \to B$. Let $\mathbb{Z}_{w_1}$ denote the local system with coefficient group $\mathbb{Z}$ determined by $w_1$. 
\begin{definition}
Let $m \ge 0$ be a non-negative integer. Let $[\eta] \in \mathcal{CH}(f)$ be a chamber of $f$ represented by a section $\eta : B \to \mathbb{W}_\mathbb{R} \setminus l_{\mathbb{R}}(\mathbb{V}_\mathbb{R})$ which is transverse to $f$. The $m$-th families Seiberg-Witten invariant of $f$ with respect to $\eta$ is defined as
\[
SW_m(f, \eta) = (\pi_{\mathcal{M}_\eta})_*( c_1(\mathcal{L})^m ) \in H^{2m-(ind(l)-1)}(B ; \mathbb{Z}_{w_1})
\]
where $w_1 = det( Ind(l_\mathbb{R})) \in H^1(B ; \mathbb{Z}_2)$.   
\end{definition}

\begin{remark}
If $Ind(l_{\mathbb{R}})$ is orientable, then we can regard $SW_m(f,\eta)$ as a cohomology class with integral coefficients. However this depends on the choice of orientation of $Ind(l_{\mathbb{R}})$. Reversing orientation on $Ind(l_{\mathbb{R}})$ will change the sign of $SW_m(f,\eta)$.
\end{remark}

\begin{lemma}\label{lem:cobordisminvariance}
Let $f : \mathbb{V} \to \mathbb{W}$ satisfy (M1)-(M7). Then $SW_m(f, \eta)$ depends only on the chamber $[\eta] \in \mathcal{CH}(f)$ in which $\eta$ lies and not on the particular choice of $\eta$.
\end{lemma}
\begin{proof}
Let $\eta_1, \eta_2 : B \to \mathbb{W}_\mathbb{R} \setminus l_\mathbb{R}(\mathbb{V}_\mathbb{R})$ represent the same chamber and assume $\eta_1, \eta_2$ are both transverse to $f$. Let $M > \sup_{b \in B} \{ ||\eta_1(b)|| , ||\eta_2(b)|| \}$. Then by (M6) we can find a generic path $\eta(t)$ joining $\eta_1$ to $\eta_2$ and such that $||\eta(t,b)|| < M$ for all $(t,b) \in [0,1] \times B$. Then $\widetilde{\mathcal{M}} = \coprod_{t\in [0,1]} f^{-1}(\eta(t))$ is a smooth manifold with free $S^1$-action and with boundary $\widetilde{\mathcal{M}_{\eta_2}} \coprod \widetilde{\mathcal{M}_{\eta_1}}$. Moreover $\widetilde{\mathcal{M}}$ is compact, by (M3). Let $\mathcal{M} = \widetilde{\mathcal{M}}/S^1$. Then $\mathcal{M}$ is a smooth compact manifold with boundary $\mathcal{M}_{\eta_2} \coprod \mathcal{M}_{\eta_1}$. Moreover the line bundles on $\mathcal{M}_{\eta_2}, \mathcal{M}_{\eta_1}$ extend to a common line bundle $\mathcal{L} \to \mathcal{M}$ (namely, the line bundle associated to $\widetilde{\mathcal{M}} \to \mathcal{M}$). It follows that
\[
SW_m(f , \eta_1) = (\pi_{\mathcal{M}_{\eta_1}})_*( c_1(\mathcal{L})^m|_{\mathcal{M}_{\eta_1}}) = (\pi_{\mathcal{M}_{\eta_2}})_*( c_1(\mathcal{L})^m|_{\mathcal{M}_{\eta_2}}) = SW_m(f,\eta_2).
\]
\end{proof}

\begin{remark}
In the case that $f$ is the families Seiberg-Witten monopole map of a spin$^c$ family of $4$-manifolds $(E , \mathfrak{s}_E )$ with $b_1(X)=0$, we have that $SW_m(f , \eta) = SW_m(E , \mathfrak{s}_E , [\eta])$ is exactly the $m$-th families Seiberg-Witten invariant of the family $(E , \mathfrak{s}_E)$ with respect to the chamber $[\eta]$, as defined in the introduction.

If $E \to B$ is a spin$^c$ family of $4$-manifolds with $b_1(X) \ge 0$ and which admits a section $x : B \to E$, we have constructed a monopole map defined on the total space of the Jacobian bundle $J \to B$ (Example \ref{ex:b1>0}). Associated to this monopole map are families Seiberg-Witten invariants $SW_m(f,\eta) \in H^{2m-(2d-b^+-1)}( J ; \mathbb{Z}_{w_1})$, where $w_1 = w_1(H^+)$. We interpret these as the families Seiberg-Witten invariants of the quadruple $(E,\mathfrak{s}_E,x,[\eta])$ and set $SW_m(E, \mathfrak{s}_E , x,[\eta]) = SW_m(f , \eta)$. 

When $b_1(X) = 0$ we clearly have that $SW_m(E,\mathfrak{s}_E,x , [\eta])$ agrees with our previous definition of $SW_m(E,\mathfrak{s}_E,[\eta])$ and in particular does not depend on the choice of section $x$ (and can be defined even if no such section exists).

If $b_1(X) > 0$, then in general $SW(E,\mathfrak{s}_E,x,[\eta])$ may depend on the choice of section $x$, but only up to homotopy. To see homotopy invariance, let $x : [0,1] \times B \to E$ be a homotopy of section from $x_0$ to $x_1$. Then we can view $x$ as a section of the pullback $[0,1] \times E$ of $E$ over $[0,1] \times B$ and obtain a families Seiberg-Witten invariant $SW_m([0,1] \times E , \mathfrak{s}_E , x , [\eta]) \in H^{2m-(2d-b^+-1)}([0,1] \times J ; \mathbb{Z}_{w_1})$. Restriction of $SW_m([0,1] \times E , \mathfrak{s}_E , x , [\eta])$ to $t=0,1$ recovers $SW_m(E , \mathfrak{s}_E , x_0 , [\eta])$ and $SW_m( E , \mathfrak{s}_E , x_1 , [\eta])$ respectively, but the homotopy equivalence $[0,1]\times J \cong J$ implies that $SW_m(E , \mathfrak{s}_E , x_0 , \eta) = SW_m(E , \mathfrak{s}_E , x_1 , [\eta])$.
\end{remark}

In order to compare the Seiberg-Witten invariants of $f$ with those of a finite dimensional approximation we would like to restate the definition of the invariants $SW_m(f , \eta)$ in a slightly different way. Fix a real line bundle $U \to B$ and consider triples $(M, \pi_M , \psi_M)$ consisting of:
\begin{itemize}
\item{a smooth compact manifold $M$ with a free $S^1$-action}
\item{an $S^1$-invariant smooth map $\pi_M : M \to B$}
\item{an $S^1$-equivariant isomorphism of real line bundles $\psi_M : det(TM \oplus \pi^*_M(TB)) \cong \pi^*_M(U)$.}
\end{itemize}
Then $M/S^1$ is a smooth compact manifold and $\pi_M$ descends to $\pi_{M/S^1} : M/S^1 \to B$. Let $\mathcal{L}_M \to M$ be the complex line bundle associated to the principal $S^1$-bundle $M \to M/S^1$ and define the $m$-th Seiberg-Witten invariant of $(M, \pi_M , \psi_M)$ to be:
\[
SW(M,\pi,\psi) = (\pi_{M/S^1})_*( c_1(\mathcal{L}_M)^m ) \in H^{2m+dim(B)-dim(M)-1}(B ; \mathbb{Z}_{w_1})
\]
where $w_1 = w_1(U)$ and where we use $\psi_M$ to identify the relative orientation class of $\pi_{M/S^1} : M/S^1 \to B$ with $\pi^*(U)$. 

Now let $f : \mathbb{V} \to \mathbb{W}$ be a families monopole map satisfying (M1)-(M7) and let $[\eta] \in \mathcal{CH}(f)$ be a chamber represented by a map $\eta : B \to \mathbb{W}_{\mathbb{R}} \setminus l_\mathbb{R}(\mathbb{V}_\mathbb{R})$ such that $\eta$ is transverse to $f$. Set $U = det(Ind(l_\mathbb{R}))$. Then we obtain a triple $( M , \pi_M , \psi_M)$, where $M = \widetilde{\mathcal{M}}_\eta = f^{-1}(\eta)$ equipped with the natural $S^1$-action, $\pi_M$ is the natural projection to $B$ and $\psi_M$ is the isomorphism in Lemma \ref{lem:orientation}. Clearly the families Seiberg-Witten invariants of $f$, $SW_m(f,\eta)$ coincide with the Seiberg-Witten invariants of the triple $(M , \pi , \psi)$.

We say that two triples $(M_1, \pi_{M_1} , \psi_{M_1})$, $(M_2, \pi_{M_2} , \psi_{M_2})$ are cobordant if there exists a smooth compact manifold $M$ with boundary the disjoint union of $M_1$ and $M_2$ such that the $S^1$-actions on $M_1, M_2$ extends to a free $S^1$-action on $M$, the maps $\pi_{M_1}, \pi_{M_2}$ are the restrictions to $\partial M$ of a $S^1$-invariant map $\pi_M : M \to B$ and there is an $S^1$-equivariant isomorphism $\psi_M : \det(TM \oplus \pi^*(TB)) \to \pi_M^*(U)$ such that $\psi|_{M_2} = \psi_{M_2}$, $\psi|_{M_1} = -\psi_{M_1}$ (where we use the outward normal first convention to identify $det(TM)|_{M_j}$ with $det(TM_j)$ for $j=1,2$). Repeating the argument used in the proof of Lemma \ref{lem:cobordisminvariance}, we find that the Seiberg-Witten invariants of a triple $(M , \pi , \psi)$ are cobordism invariant. Thus to obtain the families Seiberg-Witten invariants $SW_m(f,\eta)$ of $f$ from a finite dimensional approximation, it suffices to obtain the cobordism class of associated triple $(M , \pi , \psi)$. This will be our goal in the following subsections.

\subsection{Finite dimensional monopole maps and their Seiberg-Witten invariants}\label{sec:fsw}

In this subsection we introduce the notion of a finite dimensional families monopole map and define the Seiberg-Witten invariants of such maps. In the following subsections we will see how to obtain a finite dimensional monopole map as an approximation of an infinite dimensional one and show that the Seiberg-Witten invariants are preserved under this process.

Let $B$ be a compact smooth manifold. Suppose that $V,V' \to B$ are complex vector bundles of ranks $a,a'$ and $U,U' \to B$ are real vector bundles of ranks $b,b'$. We make $V,V'$ into $S^1$-equivariant vector bundles by letting $S^1$ act by scalar multiplication on the fibres. Similarly we make $U,U'$ into $S^1$-equivariant vector bundles by giving them the trivial action.

For any vector bundle $W$, let $S_W$ be the unit sphere bundle in $\mathbb{R} \oplus W$. Thus $S_W$ is obtained from $W$ by one-point compactifying each fibre of $W$. Let $B_W \subset S_W$ denote the section at infinity. If $V$ is a complex vector bundle and $U$ a real vector bundle, then the $S^1$-action on $V \oplus U$ extends to $S_{V \oplus U}$ preserving the section at infinity. We often write $S_{V,U}$ instead of $S_{V \oplus U}$ to remind ourselves that the $S^1$-action on $S_{V,U}$ depends on the pair $(V,U)$.

Suppose that we have an $S^1$-equivariant map $f : S_{V,U} \to S_{V',U'}$ covering the identity on $B$ and which sends infinity to infinity. Our interest in such maps is that they arise as finite dimensional approximations of the Seiberg-Witten monopole map of spin$^c$ families of $4$-manifolds. More precisely, suppose that $X$ is a compact, oriented, smooth $4$-manifold with $b_1(X) = 0$ and suppose that $E \to B$ is a spin$^c$ family with fibres diffeomorphic to $X$. To such data we constructed the families Seiberg-Witten monopole map (Example \ref{ex:b1=0}). As will be shown in Subsection \ref{sec:fda}, we can take a finite dimensional approximation of the Seiberg-Witten monopole map to obtain an $S^1$-equivariant map $f : S_{V,U} \to S_{V', U'}$ where $V,V',U,U'$ satisfy:
\[
V - V' = D \in K^0(B), \quad U' - U = H^+ \in KO^0(B),
\]
where $D \in K^0(B)$ denotes the families index of the families spin$^c$ Dirac operator of $(E , \mathfrak{s}_{E})$ and $H^+$ is the vector bundle over $B$ whose fibres are the space of harmonic self-dual $2$-forms on the corresponding fibres of the family $E \to B$.

\begin{remark}
A similar construction holds in the case that $b_1(X) > 0$. Consider a spin$^c$ family $(E , \mathfrak{s}_E)$ over $B$ which admits a section $x : B \to E$. To such data we constructed a families Seiberg-Witten monopole map (Example \ref{ex:b1>0}). Taking a finite dimensional approximation of this map one again obtains an equivariant map $f : S_{V \oplus U} \to S_{V' \oplus U'}$, except that $V,V',U,U'$ are now vector bundles over the space $J$, where $\pi_J : J \to B$ is the Jacobian bundle of the family $E \to B$ (see Example \ref{ex:b1>0}). Further, we have that $V,V',U,U'$ satisfy:
\[
V - V' = D_J \in K^0(J), \quad U' - U = \pi^*_J(H^+) \in KO^0(J),
\]
where $H^+$ is defined as in the $b_1(X)=0$ case and $D_J$ is defined as follows. Let $D_A$ be the spin$^c$ Dirac operator associated to the reference connection $A$. Each point in $J$ defines a flat line bundle on $X$. Coupling this line bundle to $D_A$ we get a family of spin$^c$ Dirac operators parametrised by $J$. Then $D_J$ is defined as the families index of this family.
\end{remark}

Consider again an $S^1$-equivariant map $f : S_{V , U} \to S_{V' , U'}$, where $V,V',U,U'$ are vector bundles on $B$ and $f$ covers the identity. This will be the setup used henceforth. By analogy with with the case of the Seiberg-Witten monopole map, we let $D$ denote the virtual bundle $D = V - V' \in K^0(B)$ and similarly let $H^+$ denote $H^+ = U' - U \in KO^0(B)$.

\begin{remark}\label{rem:inclusion}
As we will see in Subsection \ref{sec:fda}, if $f$ arises as a finite dimensional approximation of the Seiberg-Witten monopole map, then $U$ can be identified with a subbundle of $U'$ and $f|_U : U \to U'$ with the inclusion $\iota : U \to U'$.
\end{remark}

In light of Remark \ref{rem:inclusion}, we will assume henceforth that:\\

\noindent {\bf Assumption 1}: $U$ can be identified with a subbundle of $U'$ and that $f|_U : U \to U'$ is the inclusion $\iota: U \to U'$.\\

Under Assumption 1 we can assume that $U'$ splits into a direct sum $U' = U \oplus H^+$ and that $H^+$ is a genuine vector bundle rather than a virtual vector bundle. Let us also set:
\[
d = rank_{\mathbb{C}}(D) = a-a', \quad b^+ = rank_{\mathbb{R}}(H^+) = b'-b.
\]
In the case that $f$ is the finite dimensional approximation of the Seiberg-Witten monopole map of a spin$^c$ family $(E_X , \mathfrak{s}_{E})$, we of course have
\[
d = \frac{ c(\mathfrak{s}_X)^2 - \sigma(X) }{8}, \quad b^+ = b^+(X),
\]
where $\mathfrak{s}_X = \mathfrak{s}_{E}|_X$ is the spin$^c$ structure on $X$ induced by restriction of $\mathfrak{s}_{E}$ to a fibre $X$ of $E$.

\begin{definition}
Let $V,V'$ be complex vector bundles and $U,U'$ be real vector bundles over a compact smooth base manifold $B$. An $S^1$-equivariant map $f : S_{V,U} \to S_{V',U'}$ preserving sections at infinity and satisfying Assumption 1 will be called a {\em finite dimensional (families) monopole map}.
\end{definition}

Let $X,Y \to B$ be fibre bundles over $B$ equipped with sections $x_\infty : B \to X$, $y_\infty : B \to Y$. Then we can form the fibrewise smash product $X \wedge_B Y$ by taking the fibre product $X \times_B Y$ and for each $b \in B$ collapsing $x_\infty(b) \times Y_b \cup X_b \times y_\infty(b)$ to a point. If $X,Y$ have $S^1$-actions and $x_\infty,y_\infty$ are $S^1$-invariant then $X \wedge_B Y$ inherits an $S^1$-action. In particular if $V_1,V_2$ are complex vector bundles and $U_1,U_2$ real vector bundles, then one finds
\begin{equation}\label{equ:fibresmash}
S_{V_1,U_1} \wedge_B S_{V_2,U_2} = S_{V_1 \oplus V_2 , U_1 \oplus U_2}.
\end{equation}

In what follows, we will be interested primarily in the stable equivariant homotopy class of $f$. By the stable equivariant homotopy class of $f$, we mean that we consider the homotopy class of $f$ up to stabilisation, where by stabilisation of $f$ we mean taking the fibrewise smash product of $f$ with $id : S_{\mathbb{C}^m , \mathbb{R}^n} \to S_{\mathbb{C}^m , \mathbb{R}^n}$ for any $m,n \ge 0$. Homotopies are taken to be homotopies through finite dimensional monopole maps, that is, through equivariant maps $f_t$ covering the identity on $B$, preserving the section at infinity and such that $f_t|_U : U \to U'$ is a linear inclusion for each $t$. By (\ref{equ:fibresmash}), stabilisation by $\mathbb{C}$ and $\mathbb{R}$ change $V,V',U,U'$ as follows:
\begin{equation*}
\begin{aligned}
V & \mapsto V \oplus \mathbb{C}, \quad V' \mapsto V' \oplus \mathbb{C} \quad (\text{stabilisation by } \mathbb{C}), \\
U & \mapsto U \oplus \mathbb{R}, \quad U' \mapsto U' \oplus \mathbb{R} \quad (\text{stabilisation by } \mathbb{R}).
\end{aligned}
\end{equation*}
Of course, such stabilisations do not change the $K$-theory classes $D = V - V'$ and $H^+ = U' - U$.

More generally, we can consider stabilisation by an arbitrary complex vector bundle $A$ or a real vector bundle $B$:
\begin{equation*}
\begin{aligned}
V & \mapsto V \oplus A, \quad V' \mapsto V' \oplus A \quad (\text{stabilisation by } A), \\
U & \mapsto U \oplus B, \quad U' \mapsto U' \oplus B \quad (\text{stabilisation by } B).
\end{aligned}
\end{equation*}
This corresponds to taking the fibrewise smash product of $f$ with $id : S_{A,0} \to S_{A,0}$ or with $id : S_{0,B} \to S_{0,B}$. Note that for any such vector bundles $A,B$ we can find complementary vector bundles $\hat{A}, \hat{B}$ such that $A \oplus \hat{A} = \mathbb{C}^m$, $B \oplus \hat{B} = \mathbb{R}^n$. 
Hence the more general notion of stabilisation by vector bundles retains the underlying stable homotopy class of $f$.

After stabilising by a suitable choice of vector bundles, we can assume that $V' \cong \mathbb{C}^{a'} \times B$ and $U' \cong \mathbb{R}^{b'} \times B$ are trivial vector bundles over $B$. In this case $f$ can be viewed as a stable equivariant homotopy class $f : S_{V,U} \to ( \mathbb{C}^{a'} \oplus \mathbb{R}^{b'})^+$ from the total space of $S_{V,U}$ into the sphere $( \mathbb{C}^{a'} \oplus \mathbb{R}^{b'})^+$. In this way $f$ defines a stable equivariant cohomotopy class, which we may refer to as the {\em Bauer-Furuta invariant} of $f$, following \cite{bafu}. On the other hand, we will see that it is sometimes more convenient to stabilise such that $V \cong \mathbb{C}^a \times B$ and $U \cong \mathbb{R}^b \times B$ are trivial. We will make use of both types of stabilisations of $f$.

Let $f : S_{V , U} \to S_{V',U'}$ be a finite dimensional monopole map. Thus we can assume that $U' = U \oplus H^+$ and that $f|_U : U \to U'$ is the inclusion map $\iota : U \to U'$.

\begin{definition}\label{def:chamber}
A {\em chamber} for $f$ is a homotopy class of section $\phi : B \to U' - U$. Equivalently, this is the same as a homotopy class of section $B \to S(H^+)$, where $S(H^+)$ is the unit sphere bundle of $H^+$. Denote by $\mathcal{CH}(f)$ the set of chambers.
\end{definition}

\begin{remark}
In general, the set of chambers for $f$ may be empty. Indeed this happens precisely when $H^+$ does not admit a non-vanishing section. An obvious necessary condition for this is that $b^+ > 0$. A sufficient condition for the existence of a chamber is $b^+ > dim(B)$. Using the fact that $S(H^+) \to B$ is a sphere bundle with fibres of dimension $b^+-1$, we see that if $b^+ > 1 + dim(B)$ then there exists a unique chamber.
\end{remark}

\begin{lemma}\label{lem:orientation2}
Let $X,Y$ be smooth finite dimensional vector bundles over $B$. Assume that $X,Y$ are equipped with $S^1$-actions and that the projections $\pi_X : X \to B$, $\pi_Y : Y \to B$ are $S^1$-invariant. Let $f : X \to Y$ be a smooth $S^1$-equivariant map covering the identity on $B$ and suppose that $\eta : B \to Y$ is an $S^1$-invariant smooth map which is transverse to $f$. Then $M = f^{-1}(\eta)$ is a smooth manifold with an $S^1$-action. Let $\pi_M : M \to B$ be the projection to $B$. Then there is a naturally defined $S^1$-equivariant isomorphism
\[
\psi_M : det( TM \oplus \pi_M^*(TB)) \cong \pi^*( U ).
\]
where $U = det(Y \oplus X)$.
\end{lemma}
\begin{proof}
$TX|_M$ can be decomposed in two ways: (1) using the vertical and horizontal subspaces with respect to $\pi_X : X \to B$ and (2) using the tangent and normal bundles of $M$. Equating these gives a canonical isomorphism:
\[
TM \oplus NM \cong \pi^*_M(V' \oplus TB),
\]
where $NM$ is the normal bundle of $M$ in $X$. But since $M = f^{-1}(\eta)$ and $\eta$ is transverse to $f$, we get a canonical isomorphism $NM \cong f^*( \pi_Y^*(W'))|_M \cong \pi^*_M(W')$. Hence
\[
TM \oplus \pi^*_M(W') \cong \pi^*_M(V' \oplus TB),
\]
which taking determinants gives $det(TM \oplus \pi_M^*(TB)) \cong \pi_M^*(U)$, where $U = det( Y \oplus X)$.
\end{proof}

Now we turn to the construction of Seiberg-Witten invariants of a finite dimensional monopole map $f : S_{V,U} \to S_{V',U'}$. Let $B_{V,U}$ denote the section at infinity of $S_{V,U}$ and similarly define $B_{V',U'}$. Our assumption that $f$ preserves sections at infinity means that we can regard $f$ as a map of pairs:
\[
f : ( S_{V\oplus U} , B_{V\oplus U} ) \to (S_{V'\oplus U'} , B_{V' \oplus U'})
\]
covering the identity on $B$. This point of view will quite useful in subsequent calculations.

Let $[\phi] \in \mathcal{CH}(f)$ be a chamber and let $\phi : B \to U'$ be a representative section (so in particular the image of $\phi$ is disjoint from $U$). By \cite[Theorem 7.1]{field} there exists an equivariant homotopy of $f$ to a map $f'$ such that $f'$ is $S^1$-transversal to $\phi(B)$ (for the definition of $G$-transversality for a compact Lie group $G$, see \cite{bier} or \cite{field}. The definitions of $G$-transversality given in these papers was shown to be equivalent in \cite{field2}). Note that $f'$ can be chosen arbitrarily close to $f$ in the $\mathcal{C}^0$-topology. Thus we can assume that ${f'}^{-1}( \phi(B) )$ is disjoint from $U$, so $S^1$ acts freely on ${f'}^{-1}(\phi(B))$. In such a case one easily verifies that the definition of $S^1$-equivariant transversality reduces to transversality in the ordinary sense. So $f'$ meets $\phi(B)$ transversally and $\widetilde{\mathcal{M}}_\phi = {f'}^{-1}( \phi(B) )$ is a compact smooth submanifold of $V \oplus U$ of dimension $2d-b^+ + dim(B)$. Moreover $S^1$ acts freely on $\widetilde{\mathcal{M}}_\phi$, so the quotient $\mathcal{M}_\phi = \widetilde{\mathcal{M}}_\phi/S^1$ is a compact smooth manifold of dimension $2d-b^+-1 + dim(B)$. The quotient map $\widetilde{\mathcal{M}}_\phi \to \mathcal{M}_\phi$ is a principal $S^1$-bundle. Let $\mathcal{L} = \widetilde{\mathcal{M}}_\phi \times_{S^1} \mathbb{C}$ be the associated complex line bundle over $\mathcal{M}_\phi$ and let $\pi_{\mathcal{M}_\phi} : \mathcal{M}_\phi \to B$ denote the natural projection map to $B$. By Lemma \ref{lem:orientation2}, we have a naturally defined isomorphism $\psi : det(T\mathcal{M}_\phi \oplus \pi_{\mathcal{M}_\phi}^*(TB)) \cong \pi^*_{\mathcal{M}_\phi}( det(H^+) )$.

\begin{definition}
Let $[\phi] \in \mathcal{CH}(f)$ be represented by $\phi$. Let $m \ge 0$ be a non-negative integer. The $m$-th Seiberg-Witten invariant of $(f , \phi)$ is the cohomology class
\[
SW_m( f , \phi ) = ({\pi_{\mathcal{M}_\phi}})_*( c_1(\mathcal{L})^m ) \in H^{2m-(2d-b^+-1)}(B ; \mathbb{Z}_{w_1})
\]
where $w_1 = w_1(H^+) \in H^1(B ; \mathbb{Z}_2)$ and $\mathbb{Z}_{w_1}$ is the associated local system with coefficient group $\mathbb{Z}$.
\end{definition}

We have defined $SW_m(f , \phi)$ as an invariant of the pair $(f,\phi)$. Later we will see that $SW_m(f,\phi)$ only depends on the stable homotopy class of $f$ and on the chamber to which $\phi$ belongs. However, before we do this we should check that $SW_m(f,\phi)$ does not depend on the choice of homotopy $f'$ of $f$ chosen so that $f'$ is $S^1$-transversal to $\phi(B)$. In fact this will follow easily once we show that $SW_m(f,\phi)$ can be expressed in purely cohomological terms and so we defer the proof until Subsection \ref{sec:swcohomology}.

\begin{proposition}
The invariants $SW_m(f,\phi)$ depend only on the stable homotopy class $[f]$ of $f$ and the chamber $[\phi]$.
\end{proposition}
This will also follow easily once we have obtained a purely cohomological definition of $SW_m(f,\phi)$, so we also defer the proof of this until Subsection \ref{sec:swcohomology}. We will also prove a wall crossing formula, expressing the dependence of $SW_m(f,\phi)$ on the choice of chamber.

\subsection{Finite dimensional approximation}\label{sec:fda}

Let $f : \mathbb{V} \to \mathbb{W}$ be a monopole map satisfying (M1)-(M7). In this subsection we will recall how a finite dimensional approximation of $f$ is constructed, as in \cite{bafu}. 

Recall that $l_{\mathbb{R}}$ is injective and we have set $W_{0,\mathbb{R}} = Ker(l^*_{\mathbb{R}})$, a finite rank subbundle of $\mathbb{W}_{\mathbb{R}}$. Next, consider $l_{\mathbb{C}} : \mathbb{V}_{\mathbb{C}} \to \mathbb{W}_{\mathbb{C}}$. From \cite[Proposition A5]{at}, there exists a finite rank complex subbundle $W_{0,\mathbb{C}} \subset \mathbb{W}_{\mathbb{C}}$ such that $W_{0,\mathbb{C}}$ surjects to the cokernel of $l_\mathbb{C}$ for any $b \in B$. Set
\[
W_0 = W_{0,\mathbb{C}} \oplus W_{0 , \mathbb{R}} \subset \mathbb{W}.
\]
Then $W_0$ is a finite rank subbundle with the property that $W_0$ surjects to the cokernel of $l$ for any $b \in B$. Let
\[
D_1(\mathbb{W}) = \{ w \in \mathbb{W} \; | \; ||w|| \le 1 \}
\]
be the unit disc bundle in $\mathbb{W}$. By (M3), there exists an $R > 0$ such that
\[
f^{-1}(D_1(\mathbb{W})) \subseteq \{ v \in \mathbb{V} \; | \; ||v|| < R \}.
\]
In particular this implies:
\begin{equation}\label{equ:1}
{\rm if } \; ||v || \ge R, \text{ then } ||f(v)|| > 1.
\end{equation}
Now let $C \subset \mathbb{W}$ be the closure of $c( D_{4R}(\mathbb{V}))$. By compactness of $c$, the set $C$ is compact. Let $C'$ be the image of $C$ under the projection $\mathbb{W} \to \mathbb{W}/W_0 \cong W_0^\perp$. Then $C'$ is also compact. Fix a trivialisation $W^\perp_0 \cong H \times B$ for some Hilbert space $H$ (note that since $W_0$ has finite rank, $W_0^\perp$ has infinite rank, so can be trivialised. Moreover the real and complex parts of $W_0^\perp$ both have infinite rank. Trivialising them separately, we see that the trivialisation $W_0^\perp \cong H \times B$ can be chosen $S^1$-equivariantly for some action of $S^1$ on $H$). Under this trivialisation $C'$ is a compact subset of $H \times B$, hence, for any fixed $\epsilon > 0$, can be covered by finitely many open sets of the form
\[
\{ x \in H \; | \; ||x-h_i||< \epsilon \} \times U_i
\]
for some $h_1, \dots , h_n \in H$ and open subsets $U_1, \dots , U_n$ of $B$. Let $H_0$ be the span of $h_1, \dots , h_n, ih_1, \dots, ih_n$ (where $ih_j$ means $h_j$ acted upon by $i \in S^1$). Then $H_0$ is a finite rank $S^1$-equivariant subspace of $H$. Let $W_1$ be the preimage of $H_0 \times B$ under the trivialisation $W_0^\perp \cong H \times B$. Then $W_1$ is a finite rank subbundle of $\mathbb{W}$, orthogonal to $W_0$. Set
\[
W' = W_0 \oplus W_1.
\]
Then $W'$ is a finite rank $S^1$-equivariant subbundle of $\mathbb{W}$. By construction $W'$ satisfies the following two properties:
\begin{itemize}
\item{$W'$ surjects to the cokernel of $l$ for any $b \in B$.}
\item{Every element of $C$ is a distance less than $\epsilon$ from $W'$.}
\end{itemize}
Let $p : \mathbb{W} \to W'$ denote the orthogonal projection to $W'$ and $p^\perp = 1-p$ the orthogonal projection to $W'' = (W')^\perp$. The second property of $W'$ listed above implies the following important estimate:
\begin{equation}\label{equ:2}
{\rm if } \; ||v|| \le 4R, \; {\rm then} \; ||p^\perp c(v) || < \epsilon. 
\end{equation}
We emphasise here that the choice of $W'$ is dependent on the choice of $\epsilon$. We will say that a finite rank subbundle $W' \subset \mathbb{W}$ which contains $W_0$ and which satisfies (\ref{equ:2}) is a {\em subbundle of class $\epsilon$} (this notion depends on the choice of $R>0$ and choice of subbundle $W_0$, but it is the $\epsilon$ dependency that we wish to emphasise, hence the choice of terminology).

Given a subbundle $W' \subset \mathbb{W}$ of class $\epsilon$, let $V' = l^{-1}(W') \subseteq \mathbb{V}$. Then $V'$ is a finite rank subbundle of $\mathbb{V}$ (because of the fact that $W'$ surjects to the cokernel of $l$ for each $b \in B$). We also let $V'' = (V')^\perp$ be the orthogonal complement. Then we have
\[
\mathbb{V} = V' \oplus V'' \quad \mathbb{W} = W' \oplus W''
\]
and $l$ decomposes into $l = l_1 + l_2 + l_3$, where $l_1 = l|_{V'} : V' \to W'$, $l_2 = p^\perp \circ l|_{V''} : V'' \to W''$, $l_3 = p \circ l|_{V''} : V'' \to W'$.

Now we are almost ready to give the construction of a finite dimensional approximation of $f$. Let $H$ be a Hilbert space. By definition the $1$-point completed Hilbert sphere is defined as $S_H = S( \mathbb{R} \oplus H)$, the unit sphere in $\mathbb{R} \oplus H$. Similarly for a Hilbert bundle $\mathbb{V} \to B$ we define $S_{\mathbb{V}}$ to be the unit sphere bundle of $\mathbb{R} \oplus \mathbb{V}$. Thus $S_{\mathbb{V}}$ is obtained from $\mathbb{V}$ by adding a point at infinity to every fibre. The bounded property (M3) ensures that $f : \mathbb{V} \to \mathbb{W}$ extends continuously to a map $f : S_{\mathbb{V}} \to S_{\mathbb{W}}$. Consider the restriction
\[
f|_{S_{V'}} : S_{V'} \to S_{\mathbb{W}}.
\]
\begin{lemma}\label{lem:avoid}
Assume that $\epsilon < 1$. Then the image of $f|_{S_{V'}}$ is disjoint from the image of $S(W'')$ in $S_{\mathbb{W}}$.
\end{lemma}
\begin{proof}
Since $f$ sends the point at infinity in $S_{V'}$ to the point at infinity in $S_\mathbb{W}$, we just have to show that the image of $f|_{V'}$ is disjoint from $S(W'') \subset \mathbb{W}$. Suppose on the contrary that there exists some $v \in V'$ with $f(v) \in S(W'')$. Thus $pf(v) = 0$ and $|| p^\perp f(v) || = ||f(v)|| = 1$. By (\ref{equ:1}), we see that $||v|| < R$ and thus by (\ref{equ:2}) it follows that $|| p^\perp c(v)|| \le \epsilon$. But if $v \in V'$, then $p^\perp l(v) = 0$ and thus 
\[
1 = || p^\perp f(v)|| = || p^\perp c(v)|| \le \epsilon < 1,
\]
a contradiction.
\end{proof}
Recall from \cite{bafu} that there is a deformation retraction
\[
\rho : S_{\mathbb{W}} \setminus S(W'') \to S_{W'}.
\]
We will recall the construction of $\rho$ in a moment. Then by Lemma \ref{lem:avoid}, we obtain a map
\[
\hat{f} = \rho \circ f|_{S_{V'}} : S_{V'} \to S_{W'}.
\]
We take $\hat{f}$ as our desired finite dimensional approximation of $f$.

We recall the definition of $\rho$. First of all we need to identify $\mathbb{W}$ as a subset of $S_{\mathbb{W}} = S(\mathbb{R} \oplus \mathbb{W})$. This is done by the stereographic map
\[
\mathbb{W} \ni w \mapsto \frac{1}{1 + ||w||^2} \left( ||w||^2-1 , 2w \right)
\]
and the point at infinity is $\infty = (1,0)$. Now under the decomposition $\mathbb{W} = W' \oplus W''$, we have
\[
S_{\mathbb{W}} = S(\mathbb{R} \oplus W' \oplus W'') = \{ (a,b,c) \in \mathbb{R} \oplus W' \oplus W'' \; | \; a^2 + ||b||^2 + ||c||^2 = 1\}.
\]
Under the stereographic map, $S(W'')$ corresponds to $(a,b,c) \in S_{\mathbb{W}}$ such that $||c|| \neq 1$ or equivalently, such that $a^2 + ||b||^2 \neq 0$. Then we may define $\rho$ by
\[
\rho(a,b,c) = \frac{1}{\sqrt{a^2 + ||b||^2}} (a,b,0).
\]
One can check that $\rho(\infty) = \infty$ and if $w \in \mathbb{W}$, then
\begin{equation}\label{equ:rho}
\rho(w) = \begin{cases} 0, & \text{if } pw=0 \text{ and } ||p^\perp w || < 1, \\ \infty, & \text{if } pw=0 \text{ and } ||p^\perp w || > 1, \\ \lambda(w)pw, & \text{otherwise}, \end{cases}
\end{equation}
where $\lambda(w)$ is a positive real number depending smoothly on $w \in \mathbb{W} \setminus S(W'')$ (the precise form of $\lambda(w)$ is not important. All we will need to know about $\lambda(w)$ is that it is a positive real number). Note that $\rho$ is a retraction to $S_{W'}$, meaning that $\rho(w) = w$ for any $w \in S_{W'}$.

\begin{proposition}
The map $\hat{f} : S_{V'} \to S_{W'}$ is a finite dimensional monopole map.
\end{proposition}
\begin{proof}
We just have to check that $\hat{f}$ preserves the section at infinity and that it satisfies Assumption 1. But $f$ and $\rho$ both clearly preserve the sections at infinity hence so does $\hat{f}$. For Assumption 1, note that if $v \in (V')^{S^1}$, then $c(v)=0$ by (M5), so $f(v) = l_1(v) \in W'$. But if $w \in W'$ then $\rho(w) = w$ because $\rho$ is a retraction to $S_{W'}$. Hence $\hat{f}(v) = \rho(l_1(v)) = l_1(v)$. So the restriction of $\hat{f}$ to $(V')^{S^1}$ is just the restriction of $l_\mathbb{R}$ to $(V')^{S^1}$ and is therefore a linear injection.
\end{proof}

In order to identify the Seiberg-Witten invariants of $f$ and its finite dimensional approximation $\hat{f}$, we must first be able to identify the chambers. Let $(W')^{S^1}$, $(V')^{S^1}$ denote the $S^1$-invariant parts of $W',V'$. Then in the notation of Subsection \ref{sec:fsw} we have correspondences $(W')^{S^1} \leftrightarrow U'$, $(V')^{S^1} \leftrightarrow U$ and $W_{0,\mathbb{R}} \leftrightarrow H^+$. In particular a chamber for $\hat{f} : S_{V'} \to S_{W'}$ is a homotopy class of section $\hat{\eta} : B \to (W')^{S^1} \setminus l_1( (V')^{S^1})$.

The following proposition is easily checked using the definitions of $V',W'$.
\begin{proposition}
Let $\eta : B \to \mathbb{W}_{\mathbb{R}} \setminus l_{\mathbb{R}}( \mathbb{V}_{\mathbb{R}})$ represent a chamber for $f$. Then $p \circ \eta : B \to W'$ is valued in $(W')^{S^1} \setminus l_1( (V')^{S^1})$ and hence represents a chamber for $\hat{f}$. The map $\eta \mapsto p \circ \eta$ induces a bijection between chambers of $f$ and $\hat{f}$.
\end{proposition}

\begin{lemma}
Let $\hat{\eta} : B \to (W')^{S^1} \setminus l_1( (V')^{S^1})$ be a chamber for $\hat{f}$. Then $\hat{f}^{-1}( \hat{\eta} )$ contains no fixed points.
\end{lemma}
\begin{proof}
This is trivial, since the image of $\hat{f}|_{(V')^{S^1}}$ is precisely $l_1( (V')^{S^1})$.
\end{proof}

\begin{lemma}\label{lem:finiteperturb}
Any chamber of $\hat{f}$ may be represented by a section $\hat{\eta} : B \to (W')^{S^1} \setminus l_1( (V')^{S^1})$ with $||\hat{\eta}|| < \delta$. Moreover there exists an equivariant homotopy of $\hat{f}$ to a map $\hat{f}' : S_{V'} \to S_{W'}$ such that $\hat{f}'$ is transverse to $\hat{\eta}$. The map $\hat{f}'$ can additionally be chosen arbitrarily close to $\hat{f}$ in the $\mathcal{C}^0$-topology.
\end{lemma}
\begin{proof}
This follows by the same type of argument used in Subsection \ref{sec:fsw}.
\end{proof}

\subsection{Equality of Seiberg-Witten invariants}\label{sec:equalitysw}
Let $f : \mathbb{V} \to \mathbb{W}$ be a monopole map satisfying (M1)-(M7). Let $R>0$ and $W_0$ be chosen as in Subsection \ref{sec:fda}. Let $\epsilon > 0$ and let $W' \subset \mathbb{W}$ be a finite rank subbundle of class $\epsilon$. Let $W'',V',V''$ be defined as in Subsection \ref{sec:fda} and recall that $l = l_1+l_2+l_3$, where $l_1 : V' \to W'$, $l_2 : V'' \to W''$, $l_3 : V'' \to W'$.

Assume that $\epsilon < 1$ so that the finite dimensional approximation $\hat{f} : S_{V'} \to S_{W'}$ is defined and is a finite dimensional monopole map.

Let $[\eta] \in \mathcal{CH}(f)$ be a chamber represented by $\eta : B \to \mathbb{W}_\mathbb{R} \setminus l_\mathbb{R}(\mathbb{V}_\mathbb{R})$ and assume that $\eta$ is transverse to $f$. Let $\hat{\eta} = p\eta : B \to (W')^{S^1} \setminus l_1( (V')^{S^1})$. Then $\hat{\eta}$ represents the corresponding chamber of $\hat{f}$. Under these conditions we have defined the families Seiberg-Witten invariants of $(f,\eta)$ and also the families Seiberg-Witten invariants of $(\hat{f} , \hat{\eta})$. In this subsection we will establish their equality.

\begin{theorem}\label{thm:swequality}
Let $R$ and $W_0$ be fixed as above. For all sufficiently small $\epsilon > 0$, if $W'$ is a finite rank subbundle of class $\epsilon$ and $\hat{f} : S_{V'} \to S_{W'}$ the correponding finite dimensional approximation, then
\[
SW_m(f,\eta) = SW_m(\hat{f},\hat{\eta})
\]
for all $m \ge 0$.
\end{theorem}

The rest of this subsection is concerned with the proof of Theorem \ref{thm:swequality}. The overall strategy of the proof is as follows. Recall that in Subsection \ref{sec:familiesmonopole} we defined the Seiberg-Witten invariants of a triple $(M , \pi_M , \psi_M)$. We also defined the notion of a corbordism of triple and showed that the Seiberg-Witten invariants of a triple are cobordism independent. Associated to the pair $(f,\eta)$ is a triple $(M , \pi_M , \psi_M)$ where $M = f^{-1}(\eta)$, $\pi_M$ is the natural projection and $\psi_M$ is the isomorphism given by Lemma \ref{lem:orientation}. The Seiberg-Witten invariants of $(f,\eta)$ are the Seiberg-Witten invariants of the triple $(M , \pi_M, \psi_M)$.

Similarly, let $\hat{f}'$ and $\hat{\eta}$ be as in Lemma \ref{lem:finiteperturb}. Then $\hat{M} = \widetilde{\mathcal{M}}_{\hat{\eta}} = (\hat{f}')^{-1}(\hat{\eta})$ is a compact smooth manifold with a free $S^1$-action and an $S^1$-invariant map to $B$. By Lemma \ref{lem:orientation2} we obtain an $S^1$-equivariant isomorphism $\psi_{\hat{M}} : det( T\widetilde{\mathcal{M}}_{\hat{\eta}} \oplus \pi^*_{\mathcal{M}_\eta}(TB) ) \cong \pi^*_{\mathcal{M}_\eta}( det( Ind(l_\mathbb{R}) ))$ and thus a triple $(\hat{M} , \pi_{\hat{M}} , \psi_{\hat{M}})$. Then the Seiberg-Witten invariants of $(\hat{f} , \hat{\eta})$ are, essentially by definition, the Seiberg-Witten invariants of the triple $(\hat{M}, \pi_{\hat{M}} , \psi_{\hat{M}})$.

Thus to prove Theorem \ref{thm:swequality}, it will be enough to show that the triples $(M , \pi_M , \psi_M)$ and $(\hat{M} , \pi_{\hat{M}} , \psi_{\hat{M}})$ are cobordant, provided $\epsilon$ is sufficiently small. We will construct such a cobordism through a sequence of steps, but first we need some preliminary results.

\begin{lemma}
For any $\epsilon > 0$ and any $W'$ of class $\epsilon$ we have:
\begin{itemize}
\item{$l_2$ is invertible and $|| l_2^{-1} || \le K$ for some constant $K$ which is independent of $W'$ and $\epsilon$.}
\item{$||l_3|| \le T$ for some constant $T$ which is independent of $W'$ and $\epsilon$.}
\end{itemize}
\end{lemma}
\begin{proof}
We first show $l_2$ is an isomorphism. It suffices to show $l_2$ is injective and surjective. 

For injectivity, suppose $l_2(v'') = 0$, where $v'' \in V''$. Then $l(v'') = l_2(v'') + l_3(v'') = l_3(v'') \in W'$. But $V' = l^{-1}(W')$, so $l(v'') = l(v')$ for some $v' \in V'$. Thus $l(v''-v') = 0$ and $v'' - v' \in Ker(l)$. But $V' = l^{-1}(W')$ implies that $Ker(l) \subseteq V'$, and it follows that $v'' = 0$.

For surjectivity, let $w'' \in W''$. Then since $W'$ surjects to $coker(l)$, there exists $w' \in W'$ such that $w'$ and $-w''$ map to the same element of the cokernel. Hence $w'+w'' = l(v)$ for some $v \in \mathbb{V}$. Writing $v$ as $v = v' + v''$ where $v' \in V', v'' \in V''$, we see that $w' + w'' = [ l_1(v') + l_3(v'')] + l_2(v'')$. Equating $W''$ components gives $l_2(v'') = w''$.

Next, let us define $V_0 = l^{-1}(W_0)$. Then $V_0$ is a constant rank subbundle of $V'$ and hence $V' = V_0 \oplus V_1$ where $V_1$ is the orthogonal complement of $V_0$ in $V'$. Similarly define $W_1$ as the orthogonal complement of $W_0$ in $W'$. Then we have orthogonal decompositions:
\[
\mathbb{V} = V_0 \oplus V_1 \oplus V'' \quad \mathbb{W} = W_0 \oplus W_1 \oplus W''.
\]
Let $p_0,p_1,p_2$ denote the orthogonal projections from $\mathbb{W}$ to $W_0,W_1,W''$ respectively. In particular, $p = p_0+p_1$ and $p^\perp = p_2$. Recall that $l_2 : V'' \to W''$ is defined as $l_2 = p_2 \circ l|_{V''}$. Similarly define $\tilde{l}_2 : (V_1 \oplus V'') \to (W_1 \oplus W'')$ by $\tilde{l}_2 = (p_1+p_2) \circ l|_{V_1 \oplus V''}$. Note that $(V_1 \oplus V'') = V_0^\perp$, $(W_1 \oplus W'') = W_0^\perp$ and that $\tilde{l}_2$ depends only on the choice of $W_0$ (and not on the choice of $W'$ or $\epsilon$). Arguing in the same way that we did for $l_2$, we see that $\tilde{l_2}$ is invertible. Let $w \in W'$. Let $v = \tilde{l}_2^{-1}(w)$ so that $\tilde{l}_2(v) = w$. Now decompose $v$ into $v = v_1+v''$, $v_1 \in V_1$, $v'' \in V''$. Then
\[
w = (p_1 + p_2) l(v) = (p_1+p_2)l(v_1) + (p_1+p_2)l(v'').
\]
Extracting $W''$ components gives
\[
w = p_2 l(v_1) + p_2 l(v'').
\]
But $v_1 \in V_1 \subseteq V' = l^{-1}(W')$, so $l(v_1) \in W'$ and $p_2 l(v_1) = 0$. Thus $w = p_2 l(v'') = l_2(v'')$ and $v'' = l_2^{-1}(w)$. Thus
\[
|| l_2^{-1}(w) || = ||v''|| \le ||v|| = ||\tilde{l}_2^{-1}(w)|| \le K ||w||
\]
where $K = \sup_{b \in B} ||\tilde{l}_2^{-1}||$ does not depend on $W'$ or $\epsilon$.

Lastly, we note that if $v'' \in V''$ then $l_3(v'') = p l(v'')$ and hence $|| l_3(v'')|| = ||p l(v'')|| \le ||l(v'')|| \le T ||v''||$, where $T = \sup_{b \in B} || l ||$ does not depend on $W'$ or $\epsilon$.
\end{proof}

\begin{lemma}\label{lem:notzero}
Let $t \in [0,1]$ be a real number and $v \in V'$ with $||v|| = R$. Then:
\begin{itemize}
\item{$pf(v) \neq 0$, hence $\rho f(v) \in S_{W'}$ does not equal $\infty$ and we can regard $\rho f(v)$ as an element of $W'$.}
\item{$|| (1-t)pf(v) + t \rho(f(v)) || > 0$.}
\end{itemize}
\end{lemma}
\begin{proof}
Let $v \in V'$ with $||v|| = R$ and suppose on the contrary that $pf(v) = 0$. By (\ref{equ:1}), we have $||f(v)|| \ge 1$ and by (\ref{equ:2}), we have $|| p^\perp c(v)|| \le \epsilon < 1$. Also $p^\perp f(v) = p^\perp c(v)$, since $v \in V'$. So
\[
1 \le ||f(v)|| = ||p^\perp f(v)|| = || p^\perp c(v)|| < 1
\]
a contradiction.

Let $v \in V'$ with $||v|| = R$. By the above, we can assume that $\rho f(v) = u pf(v)$, where $u = \lambda(f(v))$ is a positive real number and $pf(v) \neq 0$. Thus
\[
(1-t)pf(v) + t\rho(f(v)) = ( 1-t + tu)pf(v)
\]
is non-zero, unless $1-t+tu = 0$. But this would imply that $(1-t) = -tu$. Clearly $(1-t) \ge 0$ and $-tu \le 0$ since $u$ is positive. So the only way we can get equality is if $1-t = 0 = -tu$. But $1-t = 0$ implies $t=1$ and $-tu=0$ implies $t=0$, so this is impossible.
\end{proof}

Fix once and for all a real number $\epsilon > 0$ such that
\[
\epsilon < \inf \left\{ \frac{1}{4} , \frac{R}{6K}   , \frac{1}{16K(T+3QR)} \right\}
\]
where $Q = sup_{b \in B} ||q||$ (recall that $c(v) = q(v,v)$). Fix a corresponding choice of finite rank subbundles $W',V'$, where $W'$ is of class $\epsilon$ and let $\hat{f} : S_{V'} \to S_{W'}$ be the finite dimensional approximation. As previously explained, we can assume that $\eta$ is chosen with $||\eta|| < \delta$ for any given $\delta > 0$. We will assume that $\delta$ is chosen with $\delta < \epsilon$, but the precise value of $\delta$ will not be fixed until later.

The cobordism from $M = f^{-1}(\eta)$ to $\hat{M} = (\hat{f}')^{-1}(\hat{\eta})$ will be carried out in a sequence of cobordisms $M = M_1 \sim M_2 \sim M_4 = \hat{M}$. For this purpose we introduce the following notation: let $f_1 : \mathbb{V} \to \mathbb{W}$ be given by $f_1 = f$, let $\eta_1 = \eta : B \to \mathbb{W}_{\mathbb{R}} \setminus l_{\mathbb{R}}( \mathbb{V}_{\mathbb{R}})$ and recall that we may assume that $|| \eta_1|| < \delta$ and that $\eta_1$ transverse to $f_1$. Write $\eta_1 = \eta_1' + \eta_1''$, where $\eta'_1 = p\eta_1$, $\eta_1'' = p^\perp \eta_1$.

\begin{lemma}\label{lem:pperpq}
For all $v,w \in \mathbb{V}$, we have
\begin{equation*}
|| p^\perp q(v,w) || \le \frac{\epsilon}{2R^2} ||v|| ||w||.
\end{equation*}
\end{lemma}
\begin{proof}
Equation (\ref{equ:polarisation}) together with (\ref{equ:2}) implies that:
\[
|| p^\perp q(v,w) || \le 2 \epsilon
\]
for all $v,w \in \mathbb{V}$ with $||v||,||w|| = 2R$. But since $q$ is $\mathbb{R}$-bilinear, a simple rescaling argument gives the result.
\end{proof}

\begin{lemma}\label{lem:contraction}
Fix a point $b \in B$. Let $v' \in \mathbb{V}_b$ satisfy $||v'|| \le 3R/2$. Define a map $\phi : V''_b \to V''_b$ by
\[
\phi(v'') = l_2^{-1}(\eta_1'') - l_2^{-1}( p^\perp c(v'+v'')).
\]
Then:
\begin{itemize}
\item[(1)]{$\phi$ sends $D_R(V''_b)$ to itself.}
\item[(2)]{$\phi : D_R(V''_b) \to D_R(V''_b)$ is a contraction. Hence there exists a unique $\theta(v') \in D_R(V''_b)$ such that $\theta(v') = \phi(\theta(v'))$.}
\item[(3)]{We have that $|| \theta(v')|| < \inf \{ R , \frac{1}{4(T+3RQ)} \}$.}
\item[(4)]{Letting $b$ and $v'$ vary, $\theta$ defines a smooth map $\theta : D_R(V') \to D_R(V'')$.}
\item[(5)]{The map $\theta$ is $S^1$-equivariant.}
\end{itemize}
\end{lemma}
\begin{proof}
Let $x,y \in D_R(V''_b)$. Then
\begin{equation*}
\begin{aligned}
|| \phi(x) - \phi(y) || &= || l_2^{-1}( p^\perp c(v'+x) - p^\perp c(v'+y) ) || \\
&\le K || p^\perp c(v'+x) - p^\perp c(v'+y) ||.
\end{aligned}
\end{equation*}
But
\begin{equation*}
\begin{aligned}
p^\perp c(v'+x) - p^\perp c(v'+y) &= p^\perp q(v'+x,v'+x) - p^\perp q(v'+y,v'+y) \\
&= 2p^\perp q(v',x-y) + p^\perp q(x+y,x-y) \\
&= p^\perp q(2v'+x+y, x-y).
\end{aligned}
\end{equation*}
Then using Lemma \ref{lem:pperpq} we find:
\begin{equation}\label{equ:contraction}
\begin{aligned}
|| \phi(x) - \phi(y) || & \le K || p^\perp c(v'+x) - p^\perp c(v'+y) || \\
& \le K || p^\perp q(2v'+x+y,x-y) || \\
& \le \frac{\epsilon K}{2R^2} ||2v' + x+y|| ||x-y|| \\
& \le \frac{\epsilon K}{2R^2} ( 2||v'|| + ||x|| + ||y|| ) ||x-y|| \\
& \le \frac{5 \epsilon K}{2R} ||x-y|| \\
& < \frac{1}{2} ||x-y||
\end{aligned}
\end{equation}
where the last line follows from $\epsilon < \frac{R}{6K}$. Then it will follow that $\phi$ acts as a contraction on $D_R(V''_b)$, provided we can show that $\phi$ sends $D_R(V''_b)$ to itself. Let $x \in D_R(V''_b)$. Then
\begin{equation*}
\begin{aligned}
|| \phi(x) || & \le ||\phi(x) - \phi(0) || + ||\phi(0) || \\
& \le \frac{1}{2} ||x|| + ||\phi(0)|| \\
& \le \frac{1}{2}R + ||\phi(0)||.
\end{aligned}
\end{equation*}
But $\phi(0) = l_2^{-1}(\eta''_1)$, so
\begin{equation}\label{equ:phi0}
|| \phi(0) || = ||l_2^{-1}(\eta''_1) || \le K ||\eta''_1|| \le K ||\eta_1|| \le K\delta \le K\epsilon.
\end{equation}
Then since $\epsilon < \frac{R}{6K}$ we get that $||\phi(0)|| \le R/6$ and thus $||\phi(x)|| \le R/2 + R/6 < R$. This proves that $\phi$ sends $D_R(V''_b)$ to itself and is a contraction. Moreover, setting $x = \theta(v')$, we have
\[
||\theta(v')|| \le ||\phi(\theta(v'))||  < R.
\]
To obtain the estimate $||\theta(v')|| < \frac{1}{4(T+3RQ)}$, we note that the second to last line of (\ref{equ:contraction}) together with (\ref{equ:phi0}) gives:
\begin{equation*}
\begin{aligned}
|| \theta(v')|| = ||\phi(\theta(v'))|| &\le ||\phi(\theta(v')) - \phi(0)|| + ||\phi(0)|| \\
&\le \frac{5 \epsilon K}{2R}||\theta(v')|| + K \epsilon \\
&\le \frac{5\epsilon K}{2} + K\epsilon \\
&\le 4K\epsilon \\
&\le \frac{1}{4(T+3RQ)}
\end{aligned}
\end{equation*}
where we also used that $||\theta(v')|| < R$. The smoothness of $\theta : D_R(V') \to D_R(V'')$ follows easily from the implicit function theorem (note that $\theta$ extends smoothly over the boundary of $D_R(V')$ because $\theta(v')$ is defined for $||v'|| \le 3R/2$).

To prove $S^1$-equivariance of $\theta$, let $z \in S^1$. Recall that $\phi(v'') = l_2^{-1}(\eta_1'') - l_2^{-1}( p^\perp c(v'+v''))$. Since this depends on $v'$ and we wish to examine the dependence, we will write $\phi_{v'}$ instead of $\phi$, so
\[
\phi_{v'}(v'') = l_2^{-1}(\eta_1'') - l_2^{-1}( p^\perp c(v'+v'')).
\]
Now using that $c$ is $S^1$-equivariant, it follows that $\phi_{z \cdot v'}(z \cdot v'') = z \cdot \phi_{v'}(v'')$. It follows that both $\theta(z \cdot v')$ and $z \cdot \theta(v')$ are fixed points of $\phi_{z \cdot v'}$, hence $\theta(z \cdot v') = z \cdot \theta(v')$ by uniqueness of the fixed point.
\end{proof}

We now set $\eta_2 = \eta_1' = p\eta_1$ and let $f_2 : V' \to W'$ be defined by
\[
f_2(v') = l_1(v') + l_3(\theta(v')) + pc(v' + \theta(v')).
\]
\begin{lemma}\label{lem:m1m2}
Let $M_1 = f_1^{-1}(\eta_1) \subseteq \mathbb{V}$ and let 
\[
M_2 = \{ v' \in V' \; | f_2(v') = \eta_2, \; ||v'|| < R \} \subseteq V'.
\]
Then:
\begin{itemize}
\item{We have that $M_1 \subseteq D_R(\mathbb{V})$ and $M_2 \subseteq D_R(V')$.}
\item{The map $\mathbb{V} \to V'$, $v \mapsto v' = pv$ restricts to a bijection $p|_{M_1} : M_1 \to M_2$.}
\item{The map $D_R(V') \to \mathbb{V}$, $v' \mapsto v = v' + \theta(v')$ restricts to a bijection $M_2 \to M_1$ inverse to $p|_{M_1}$.}
\item{The map $p|_{M_1} : M_1 \to M_2$ is an $S^1$-equivariant diffeomorphism.}
\item{$\eta_2$ is transverse to $f_2$.}
\end{itemize}
\end{lemma}
\begin{proof}
Suppose that $f_1(v) = \eta_1$. Then since $||\eta_1|| \le \delta < \epsilon < 1$, it follows from (\ref{equ:1}) that $||v|| < R$. Hence $M_1 \subseteq D_R(\mathbb{V})$. Also $M_2 \subseteq D_R(V')$ by its definition. Consider again the equation $f_1(v) = \eta_1$. Write $v=  v' + v''$ with $v' = pv \in V'$, $v'' = p^\perp v \in V''$. Then as $f_1 = f = l + c$, the $W'$ and $W''$ components of $f_1(v) = \eta_1$ are:
\begin{equation*}
\begin{aligned}
l_1(v') + l_3(v'') + pc(v'+v'') &= \eta'_1 \\
l_2(v'') + p^\perp c(v'+v'') &= \eta''_1.
\end{aligned}
\end{equation*}
A solution of this pair of equations must have $||v|| < R$, hence $||v'|| < R$, $||v''|| < R$. Since $l_2$ is invertible, the second equation can be rewitten as:
\[
v'' = l_2^{-1}(\eta''_1) - l_2^{-1} p^\perp c(v'+v'').
\]
Then since $||v'|| < R$, Lemma \ref{lem:contraction} implies that this equation has a unique solution $v'' = \theta(v')$. Substituting, we get that
\[
M_1 = \{ v \in \mathbb{V} \; | \; f_1(v) = \eta_1 \} = \{ v'  \in V' \; | \; ||v'||<R, \; l_1(v') + l_3( \theta(v') ) + pc(v'+\theta(v')) = \eta'_1\},
\]
where the bijection is given by $v \mapsto v' = pv$. However we recognise the set on the right as $M_2$. So $p : \mathbb{V} \to \mathbb{V'}$ restricts to a bijection $p|_{M_1} : M_1 \to M_2$. The inverse map is $v' \mapsto v = v' + \theta(v')$. Since both of these maps are smooth and $S^1$-equivariant, we have that $p|_{M_1}$ is an $S^1$-equivariant diffeomorphism. The fact that $f_2$ is transverse to $\eta_2$ just follows from the fact that $f_1$ is transverse to $\eta_1$.
\end{proof}

Note that $M_1$ can be promoted to a triple $(M_1 , \pi_1 , \psi_1)$ in the sense defined in Subsection \ref{sec:familiesmonopole}, where $\pi_1$ is the natural projection to $B$ and $\psi_1$ is the isomorphism of Lemma \ref{lem:orientation}. Similarly $M_2$ comes equipped with a natural map $\pi_2 : M_2 \to B$ and the diffeomorphism $M_1 \cong M_2$ of Lemma \ref{lem:m1m2} sends $\pi_1$ to $\pi_2$. Then we can extend $(M_2,\pi_2)$ to a triple $(M_2,\pi_2,\psi_2)$ where $\psi_2$ is obtained from Lemma \ref{lem:orientation2}. Comparing the proofs of Lemma \ref{lem:orientation} and \ref{lem:orientation2} one easily sees that the diffeomorphism $M_1 \to M_2$ that we have constructed sends $\psi_1$ to $\psi_2$. Hence the triples $(M_1,\pi_1,\psi_1), (M_2,\pi_2,\psi_2)$ are cobordant (indeed they are isomorphic).

For $t \in [0,1]$, define $f_t : D_R(V') \to W'$ by
\[
f_t(v) = l_1(v) + (1-t)l_3(\theta(v)) + pc(v+ (1-t)\theta(v)).
\]
Note that $f_t$ extends smoothly over the boundary of $D_R(V')$ because $\theta$ does. Then $f_t$ defines an equivariant homotopy from $f_2$ to $f_3$, where $f_3 : D_R(V') \to W'$ is defined by
\[
f_3(v) = l_1(v) + pc(v) = p( l(v) + c(v)) = p(f(v)).
\]

\begin{lemma}
Let $v' \in D_R(V')$. Then $\rho(f(v)) \neq \infty$.
\end{lemma}
\begin{proof}
Suppose on the contrary that $v \in D_R(V')$ and $\rho(f(v)) = \infty$. By (\ref{equ:rho}) this means that $pf(v) = 0$ and $|| p^\perp f(v) || > 1$. But $v \in V'$ means that $p^\perp f(v) = p^\perp c(v)$. Also $||v|| \le R$ implies that $||p^\perp c(v) || \le \epsilon$, by (\ref{equ:2}). Thus
\[
1 < || p^\perp f(v) || = ||p^\perp c(v) || \le \epsilon < 1,
\]
a contradiction.
\end{proof}
By this lemma, if $v \in D_R(V')$, then we can regard $\rho(f(v))$ as an element of $W'$ and in fact $\rho(f(v)) = \lambda(f(v)) pf(v)$ by (\ref{equ:rho}). For $t \in [0,1]$, we define $h_t : D_R(V') \to W'$ by
\[
h_t(v) = (1-t)f_3(v) + tf_4(v)
\]
where $f_4 : D_R(V') \to W'$ is defined by:
\[
f_4(v) = \rho(f(v)).
\]
In other words $f_4 = \hat{f}|_{D_R(V')}$, where $\hat{f}$ is the finite dimensional approximation of $f$. Clearly $h_t$ is an equivariant homotopy from $f_3$ to $f_4$.

\begin{lemma}\label{lem:notsmall}
Let $S_R(V') = \{ v \in V' \; | \; ||v|| = R\}$. Then for any $t \in [0,1]$ and $v \in S_R(V')$ we have that $||f_t(v)|| > \epsilon$ and $|| h_t(v) || > 0$.
\end{lemma}
\begin{proof}
The inequality $||h_t(v)|| > 0$ was shown in Lemma \ref{lem:notzero}. Now suppose there is some $v \in V'$ with $||v|| = R$ and $||f_t(v)|| \le \epsilon$. By (\ref{equ:1}) we have that $||f(v)|| > 1$ and by (\ref{equ:2}) we have $|| p^\perp c(v)|| < \epsilon$. So
\begin{equation*}
\begin{aligned}
1 < ||f(v)|| &\le ||f(v) - f_t(v)|| + ||f_t(v)|| \\
&\le \epsilon + ||f(v) - f_t(v)|| \\
&\le \epsilon + ||pc(v+(1-t)\theta(v)) -pc(v)|| + ||l_3(\theta(v))|| + ||p^\perp c(v)|| \\
&\le 2\epsilon + ||c(v+(1-t)\theta(v)) -c(v)|| + T||\theta(v)|| \\
&\le 2\epsilon + T||\theta(v)|| + ||q(v+(1-t)\theta(v) , v+(1-t)\theta(v) ) - q(v,v)|| \\
&\le 2\epsilon + T||\theta(v)|| + ||2q(v , \theta(v))|| + ||q(\theta(v) , \theta(v))|| \\
&\le 2\epsilon + (T+ 2Q||v|| + Q||\theta(v)||)||\theta(v)|| \\
&\le 2\epsilon + (T + 3QR)||\theta(v)|| \\
&\le 2\epsilon + \frac{1}{4} < 1
\end{aligned}
\end{equation*}
where in the last line we used Lemma \ref{lem:contraction} and $\epsilon < 1/4$. This is a contradiction, hence no such $v$ exists.
\end{proof}

\begin{lemma}\label{lem:notsmall2}
Let $S^+_{V'}$ denote the complement in $S_{V'}$ of $\{ v \in V' \; | \; ||v|| < R \}$. If $v \in S_{V'}^+$ then $\hat{f}(v) \neq 0$.
\end{lemma}
\begin{proof}
Suppose that $v \in S_{V'}^+$ and $\hat{f}(v) = 0$. Recall that $\hat{f}(v) = \rho(f(v))$. Since $\hat{f}(\infty) = \infty$, we can assume $v \in V'$ and since $v \in S_{V'}^+$ we can further assume that $||v|| \ge R$. From (\ref{equ:rho}) we see that $\rho(f(v)) = 0$ if and only if $pf(v) = 0$ and $||p^\perp f(v)|| < 1$, hence $||f(v)|| < 1$. But from (\ref{equ:1}), we have that $||f(v)|| \ge 1$, so no such $v$ can exist.
\end{proof}

Let us define
\[
\nu = \inf_{v \in V', ||v|| = R} || h_t(v) ||.
\]
By Lemma \ref{lem:notsmall} and compactness of $\{ v \in V' \; | \; ||v|| = R\}$ (since $V'$ is finite dimensional), we have that $\nu > 0$. Similarly by Lemma \ref{lem:notsmall2} and compactness of $S^+_{V'}$ there exists some $\nu' > 0$ such that $\hat{f}$ sends $S^+_{V'}$ to $S_{W'} \setminus D_{\nu'}(W')$. Now at last we fix a choice of $\delta > 0$ such that $\delta < \inf \{ \epsilon/2 , \nu/2 , \nu' \}$.
\begin{lemma}
Let $t \in [0,1]$ and $v \in S_R(V')$. Then $|| f_t(v) - \eta_2 || > \delta$ and $||h_t(v) - \eta_2 || > \delta$.
\end{lemma}
\begin{proof}
From Lemma \ref{lem:notsmall} we have that $|| f_t(v) || > \epsilon$, hence 
\[
||f_t(v) - \eta_2|| \ge ||f_t(v)|| - ||\eta_2|| > \epsilon - \delta > \epsilon/2 > \delta.
\]
Similarly, $||h_t(v)|| \ge \nu$ by the definition of $\nu$ and hence 
\[
||h_t(v) - \eta_2|| \ge ||h_t(v)|| - ||\eta_2|| \ge \nu - \delta > \nu/2 > \delta.
\]
\end{proof}

Let $B_\infty \subseteq S_{W'}$ denote the section at infinity. There exists a fibrewise deformation retraction $S_{W'} \setminus D_\delta(W') \to B_\infty$, moreover we can choose the retracting homotopy to be $S^1$-equivariant and smooth. It follows that if $\varphi : D_R(V') \to W'$ is any smooth $S^1$-equivariant map such that $\varphi|_{S_R(V')}$ takes values in $W' \setminus D_\delta(W')$, then $\varphi$ admits a smooth $S^1$-equivariant extension $\tilde{\varphi} : S_{V'} \to S_{W'}$ with the property that if $v \in S_{V'} \setminus D_R(V')$, then $\tilde{\varphi}(v) \in S_{W'} \setminus D_\delta(W')$. Therefore if $\eta : B \to W'$ is any section with $||\eta || < \delta$, it follows that $\tilde{\varphi}^{-1}(\eta) = \varphi^{-1}(\eta)$. We apply this construction to $f_2,f_3,f_4$ and also the homotopies $f_t$, $h_t$ to obtain $\tilde{f}_2, \tilde{f}_3, \tilde{f}_4$ and homotopies $\tilde{f}_t$ from $\tilde{f}_2$ to $\tilde{f}_3$ and $\tilde{h}_t$ from $\tilde{f}_3$ to $\tilde{f}_4$. 

Recall that $\eta_2$ is transverse to $f_2$ and that $M_2 = f_2^{-1}(\eta_2)$ is equivariantly diffeomorphic to the original moduli space $M_1$. From the above remarks we have that $\tilde{f}_2$ is also transverse to $\eta_2$ and that $M_2 = \tilde{f}_2^{-1}(\eta_2)$.

\begin{lemma}\label{lem:cobord1}
There exists an $S^1$-equivariant map $\tilde{f}_4' : S_{V'} \to S_{W'}$ equivariantly homotopic to $\tilde{f}_4$ such that $\tilde{f}_4'$ is transverse to $\eta_2$. Moreover $\tilde{f}_4'$ can be chosen to be arbitrarily close to $\tilde{f}_4$ and such that the preimage $M_4 = (\tilde{f}_4' )^{-1}(\eta_2)$ contains no fixed points of the $S^1$-action. In addition, $M_4$ is equivariantly cobordant to $M_2$ (and hence also to $M_1$).
\end{lemma}
\begin{proof}
Let $\tau : [0,1] \to [0,1]$ be a smooth function such that $\tau(t) = 0$ for $0 \le t \le 1/3$ and $\tau(t) = 1$ for $2/3 \le t \le 1$. Then the homotopies $\tilde{f}_{\tau(t)}$ and $\tilde{h}_{\tau(t)}$ can be joined together smoothly to give a smooth homotopy $k_t$ from $\tilde{f}_2$ to $\tilde{f}_4$ with the property that if $k_t(v) \in D_\delta(W')$, then $||v|| < R$. We can equivariantly deform $\tilde{f}_4$ to $\tilde{f}'_4$ and equivariantly deform $k_t$ to a homotopy $k'_t$ from $\tilde{f}_2$ to $\tilde{f}'_4$ such that $\tilde{f}_2, \tilde{f}'_4$ are transverse to $\eta_2$ and $k' : [0,1] \times S_{V'} \to S_{W'}$ is transverse to the constant path $\eta_2 : [0,1] \times B \to W'$. If we take our deformations to be sufficiently close to the original maps, then $M_4 = (\tilde{f}_4')^{-1}(\eta_2)$ and $(k')^{-1}(\eta_2)$ will be fixed point free. Moreover, $(k')^{-1}(\eta_2)$ defines an equivariant cobordism from $M_2$ to $M_4$.
\end{proof}

\begin{lemma}\label{lem:cobord2}
There exists an equivariant homotopy of $\hat{f}$ to a map $\hat{f}'$ such that $\eta_2$ is transverse to $\hat{f}'$. Also $\hat{f}'$ can be chosen arbitrarily close to $\hat{f}$ and such that there is an equivariant diffeomorphism $M_4 = (\hat{f}')^{-1}(\eta_2)$.
\end{lemma}
\begin{proof}
Let $f_4' = \tilde{f}_4' |_{D_R(V')}$. Then $f'_4$ is a deformation of $f_4 = \hat{f}|_{D_R(V')}$. Let $H : [0,1] \times D_R(V') \to S_{W'}$ be the homotopy from $\hat{f}_{D_R(V')}$ to $f'_4$. Recall that the image of $\tilde{f}_4 |_{S_R(V')} = \hat{f}|_{S_R(V')}$ is disjoint from $D_\delta(W')$. We can therefore choose the deformation $\tilde{f}'_4$ and homotopy $H$ such that the image of $H |_{[0,1] \times S_R(V')}$ is disjoint from $D_\delta(W')$. It follows that there exists some $\epsilon' > 0$ such that $H(t,v) \notin D_\delta(W')$ for all $t \in [0,1]$ and all $v \in V'$ with $R - \epsilon' \le ||v|| \le R$. Now define $K : [0,1] \times S_{V'} \to S_{W'}$ as follows. As in the previous lemma, let $\tau : [0,1] \to [0,1]$ be a smooth function such that $\tau(t) = 0$ for $0 \le t \le 1/3$ and $\tau(t) = 1$ for $2/3 \le t \le 1$. Recall that we defined $S_{V'}^+ = S_{V'} \setminus \{ v \in V' \; | \; ||v|| < R \}$. Then we define
\[
K(t,v) = \begin{cases} \hat{f}(v) & \text{if } v \in S_{V'}^+ \\ H( t \tau( \frac{R - ||v||}{\epsilon'} ) , v ) & \text{if } R-\epsilon' \le ||v|| \le R \\ H(t,v) & \text{if } ||v|| \le R-\epsilon'. \end{cases}
\]
Then $K$ is a smooth equivariant homotopy from $\hat{f}$ to $\hat{f}'$, where we set $\hat{f}'(v) = K(1,v)$. Since we can choose $\tilde{f}_4'$ can be chosen arbitrarily close to $\tilde{f}_4$ and such that the homotopy $H$ also remains arbitrarily close to $\tilde{f}_4$ at all times, it follows that $K(t,v)$ can be made arbitrarily close to $\hat{f}$ for all times, in particular $\hat{f}'$ can be made arbitrarily close to $\hat{f}$. Next we claim that $\hat{f}'$ is transverse to $\eta_2$. Recall that $||\eta_2|| < \delta$. From Lemma \ref{lem:notsmall2}, we see that if $v \in S_{V'}^+$ then $\hat{f}(v) = \hat{f}'(v) \notin D_{\delta}(W')$, so there are no solutions to $\hat{f}'(v) = \eta_2$ with $v \in S_{V'}^+$. Similarly, since $H(1,v) \notin D_\delta(W')$ for all $v \in V'$ with $R - \epsilon' \le ||v|| \le R$, we see that there are no solutions to $\hat{f}'(v) = \eta_2$ with $R - \epsilon' \le ||v|| \le R$ as well. Next we note that $\hat{f}' |_{D_{R-\epsilon'}(V')} = f_4' |_{D_{R-\epsilon'}(V')}$ and that there are no solutions to $f'_4(v) = \eta_2$ with $R - \epsilon' \le ||v|| \le R$, since $H(1,v) = f'_4$. Thus $\hat{f}'$ is transverse to $\eta_2$ because $f'_4$ is and $(\hat{f}')^{-1}(\eta_2) = (f'_4)^{-1}(\eta_2) = M_4$.
\end{proof}

\noindent {\em Completion of proof of Theorem \ref{thm:swequality}}: we have seen that the Seiberg-Witten invariants of the original monopole map are the Seiberg-Witten invariants of the triple $(M_1 , \pi_1 , \psi_1)$ and similarly the Seiberg-Witten invariants of $\hat{f}$ are those of the triple $(M_4 , \pi_4 , \psi_4)$, where $M_4 = \hat{M} = (\hat{f}')^{-1}(\eta_2)$, $\pi_4$ is the  natural projection to $B$ and $\psi_4$ is the isomorphism given as in Lemma \ref{lem:orientation2}. We have previously defined the triple $(M_2,\pi_2,\psi_2)$ and shown an isomorphism $(M_1 , \pi_1 , \psi_1) \cong (M_2, \pi_2 , \psi_2)$. Lemmas \ref{lem:cobord1} and \ref{lem:cobord2} show that there is a cobordism of pairs $(M_2 , \pi_2) \sim (M_4,\pi_4)$. It just remains to check that that $\psi_2, \psi_4$ extend over this cobordism. Recall that $\psi_2$ is the isomorphism obtained by applying Lemma \ref{lem:orientation2} to $f_2$ and similarly $\psi_4$ is the isomorphism obtained by applying Lemma \ref{lem:orientation2} to $f'_4$. Applying Lemma \ref{lem:orientation2} to the homotopy $k'_t|_{D_R(V')}$ (where $k'_t$ is as in Lemma \ref{lem:cobord1}) joining $f_2$ to $f'_4$, we get the desired extension of $\psi_2, \psi_4$ over the cobordism. So at last we have shown that $(M_1, \pi_1 , \psi_1)$ and $(M_4, \pi_4, \psi_4)$ are cobordant and thus have the same Seiberg-Witten invariants.

\section{Cohomological formulation of the families Seiberg-Witten invariants}\label{sec:cohom}

In this section we will show how the families Seiberg-Witten invariants of a finite dimensional monopole map can be recovered from purely cohomological operations. This reformulation of the Seiberg-Witten invariants makes it easy to establish various properties of them, such as the wall crossing formula or the computation of their Steenrod powers.

\subsection{Equivariant cohomology computations}\label{sec:cohomological}

Let $B$ be a compact smooth manifold as before. Let $V \to B$ be a complex vector bundle of rank $a$ and $U \to B$ a real vector bundle of rank $b$. We make $V$ into an $S^1$-equivariant vector bundle where $S^1$ acts by scalar multiplication on the fibres of $V$. We make $U$ into an $S^1$-equivariant vector bundle with the trivial $S^1$-action. Let $S_{V,U}$ denote the unit sphere bundle $S_{V,U} = S(\mathbb{R} \oplus V \oplus U)$. Let $S_U$ denote $S_{0 , U}$, let $q : S_U \to B$ be the projection to $B$. Let $\mathbb{P}(V)$ be the projective bundle associated to $V$ and note that $q^*(\mathbb{P}(V)) = \mathbb{P}(q^*V)$.

We will work with the $S^1$-equivariant cohomology of various spaces. Let $\mathbb{C}$ be equipped with the standard $S^1$-action by scalar multiplication. This defines an equivariant line bundle over a point and thus a class $x = c_1(\mathbb{C}) \in H^2_{S^1}(pt ; \mathbb{Z})$. Then $H^*_{S^1}(pt ; \mathbb{Z})$ is isomorphic to $\mathbb{Z}[x]$, the ring of polynomials in $x$ with integer coefficients. Equip $B$ with the trivial $S^1$-action. Since we will be working with spaces that fibre (equivariantly) over $B$, the cohomology groups of interest will be modules over the equivariant cohomology of $B$. We introduce the following notation:
\[
\mathcal{H}^* = H^*_{S^1}(B ; \mathbb{Z}) = H^* \otimes H^*(B ; \mathbb{Z}) = H^*(B ; \mathbb{Z})[x].
\]
We will also need to work with local systems and with coefficient groups other than $\mathbb{Z}$. Let $A$ be a local system of abelian groups on $B$ (equipped with trivial $S^1$-action) and define:
\[
\mathcal{H}^*(A) = H^*_{S^1}(B ; A) = H^* \otimes H^*(B ; A) = H^*(B ; A)[x].
\]
We are mostly interested in local systems which arise as follows. Suppose $w \in H^1(B , \mathbb{Z}_2)$. Then $w$ corresponds to a principal $\mathbb{Z}_2$ covering $B_w \to B$ and we let $\mathbb{Z}_w$ denote the local system $B_w \times_{\mathbb{Z}_2} \mathbb{Z}$, where $\mathbb{Z}_2$ acts on $\mathbb{Z}$ as multiplication by $\pm 1$.

Recall that $S_{V,U} = S(\mathbb{R} \oplus V \oplus U)$. There is a global section $B \to S_{V,U}$ given by $(1,0,0)$. Let $B_{V,U}$ denote the image of this section.

In the computations that follow we will first assume that $U$ is oriented. The results and proofs can easily be adapted to the case that $U$ is non-orientable, provided we use local coefficients. 

\begin{proposition}\label{prop:equivthom}
Suppose that $U$ is orientable. Then $H^*_{S^1}( S_{V,U} , B_{V,U} ; \mathbb{Z})$ is a free rank $1$ module over $\mathcal{H}^*$ with generator $\tau_{V,U} \in H^{2a+b}_{S^1}( S_{V,U} , B_{V,U} ; \mathbb{Z})$.
\end{proposition}
\begin{proof}
This is the Thom isomorphism in equivariant cohomology. We briefly recall the proof. Since $G = S^1$ acts on $V\oplus U$, we obtain a vector bundle $(V \oplus U)_G$ over $B \times BG$ given by $(V \oplus U) \times_G EG \to B \times_G EG = B \times BG$. Now we just apply the usual Thom isomorphism to $(V \oplus U)_G$. In particular we obtain an {\em equivariant Thom class} $\tau_{V,U}$. This point of view also makes it clear that the choice of generator $\tau_{V,U}$ corresponds to a choice of orientation of $U$.
\end{proof}

In the case that $U$ is not orientable we instead have an equivariant Thom class $\tau_{V,U} \in H^{2a+b}_{S^1}( S_{V,U} , B_{V,U} ; \mathbb{Z}_{w_1(U)})$ and the Thom isomorphism continues to hold as long as we use local coefficients.

Recall that the {\em equivariant Euler class} $e_{V,U} \in \mathcal{H}^{2a+b}(\mathbb{Z}_{w_1(U)})$ of $V\oplus U$ is the pullback of the equivariant Thom class $\tau_{V,U}$ under the zero section $B \to (V \oplus U)$ (viewed as a map of pairs $(B , \emptyset ) \to ( S_{V,U} , B_{V,U})$). Clearly this factors as $e_{V,U} = e_V e_U$, where $e_U$ is just the ordinary Euler class of $U$ and $e_V$ is the equivariant Euler class of $V$. Using the splitting principle, one finds that $e_V$ is given by:
\[
e_V = x^a + c_1(V)x^{a-1} + c_2(V)x^{a-2} + \cdots + c_a(V).
\]

Let $b \ge 1$ and let $U$ be a real oriented orthogonal vector bundle. Suppose that we are given an oriented sub-bundle $i_W : W \to U$ of positive codimension. Assume that there exists a section $\phi : B \to U$ which is disjoint from $W$. Equivalently, assume that the orthogonal complement $W^\perp$ of $W$ in $U$ admits a non-vanishing section. Recall from \cite{atbo} that pushforward maps can be constructed in equivariant cohomology. In particular we have that $\phi$ induces a pushforward map
\[
\phi_* : H^j_{S^1}(B ; \mathbb{Z}) \to H^{j+2a+b}_{S^1}( S_{V,U} , S_W ; \mathbb{Z} ).
\]
The image of $\phi_*(1)$ under the natural map $H^{2a+b}_{S^1}(S_{V,U} , S_W ; \mathbb{Z}) \to H^{2a+b}_{S^1}(S_{V,U} , B_{V,U} ; \mathbb{Z})$ is precisely the equivariant Thom class $\tau_{V,U}$. Thus $\phi_*(1)$ is a lift of $\tau_{V,U}$ and hence we sometimes denote it as $\phi_*(1) = \widetilde{\tau}^\phi_{V,U}$.

\begin{proposition}\label{prop:equivthom2}
Let $U$ be a real oriented vector bundle of rank $b \ge 1$ and suppose that $W \subset U$ is a proper subbundle of rank $w$ which is also oriented. Suppose that $U$ admits a section $\phi : B \to U$ whose image is disjoint from $W$. Let $V$ be a complex vector bundle of rank $a$. Then $H^*_{S^1}( S_{V,U} , S_{W} ; \mathbb{Z})$ is a free $\mathcal{H}^*$-module of rank $2$, generated by $\widetilde{\tau}^\phi_{V,U}$ and $\delta \tau_{0,W}$, where $\widetilde{\tau}^\phi_{V,U} = \phi_*(1)$ and $\delta \tau_{0 , W}$ is the image of $\tau_{0,W}$ under the coboundary map $\delta : H^{w}_{S^1}( S_{W} , B_{W} ; \mathbb{Z}) \to H^{w+1}_{S^1}( S_{V,U} , S_{W} ; \mathbb{Z})$ of the triple $(S_{V,U} , S_{W} , B_{W})$.
\end{proposition}
\begin{proof}
Consider the long exact sequence in equivariant cohomology of the triple $(S_{V,U} , S_{W} , B_{V,U})$:
\[
\cdots \to H^i_{S^1}( S_{V,U} , S_{W} ; \mathbb{Z}) \buildrel i_{W}^* \over \longrightarrow H^i_{S^1}(S_{V,U} , B_{V,U} ; \mathbb{Z}) \buildrel \iota^* \over \longrightarrow H^{i}_{S^1}( S_{W} , B_{V,U} ; \mathbb{Z} ) \buildrel \delta \over \longrightarrow \cdots
\]
by our discussion above we have that $i_W^*( \widetilde{\tau}^\phi_{V,U}) = \tau_{V,U}$. It follows by the Thom isomorphism that $i_W^*$ is surjective and hence $\iota^*=0$. Therefore the above long exact sequence splits into short exact sequences:
\[
0 \to H^{i-1}_{S^1}( S_{W} , B_{V,U} ; \mathbb{Z} ) \buildrel \delta \over \longrightarrow H^i_{S^1}( S_{V,U} , S_{W} ; \mathbb{Z}) \buildrel i_{W}^* \over \longrightarrow H^i_{S^1}(S_{V,U} , B_{V,U} ; \mathbb{Z}) \to 0.
\]
But $H^i_{S^1}(S_{V,U} , B_{V,U} ; \mathbb{Z})$ is a free $\mathcal{H}^*$-module generated by $\tau_{V,U}$. So the choice of lift $\widetilde{\tau}^\phi_{V,U}$ of $\tau_{V,U}$ determines a splitting of the above short exact sequences.
\end{proof}

Using local coefficients we can easily extend Proposition \ref{prop:equivthom2} to the case that $U$ and $W$ are not necessarily oriented.

Next we will be interested in the relative equivariant cohomology of the pair $( S_{V,U} , S_U )$. In this case it is useful to replace $S_U$ with an equivariant tubular neighbourhood. Recall that we let $q : S_U \to B$ denote the projection to $B$. Then one sees that the normal bundle to $S_U$ in $S_{V,U}$ is given by $q^*V$, equipped with the action of $S^1$ by fibrewise scalar multiplication. Let $\widetilde{U}$ be the unit open disc bundle in $q^*(V)$, which we can identify with an equivariant tubular neighbourhood of $S_U$. Let $\widetilde{Y} = S_{V,U} \setminus \widetilde{U}$. We observe that $\widetilde{Y}$ is a compact manifold with boundary $\partial \widetilde{Y} = S( q^*V )$, the unit sphere bundle of $q^*V$. We further observe that $S^1$ acts freely on $\widetilde{Y}$. Let $Y = \widetilde{Y}/S^1$ be the quotient, which is a compact manifold with boundary $\partial Y = (\partial \widetilde{Y})/S^1 = S(q^* V)/S^1 = \mathbb{P}(q^*V) = q^*( \mathbb{P}(V))$. Note that $\widetilde{Y} \to Y$ is a principal circle bundle whose Chern class is just the image of $x \in \mathcal{H}^2$ in $H^2(Y ; \mathbb{Z})$ under the natural map $\mathcal{H}^* \to H^*_{S^1}(\widetilde{Y} ; \mathbb{Z}) \cong H^*(Y ; \mathbb{Z})$. Note further that the restriction of $x$ to $\partial Y$ is precisely the Chern class of the bundle $\mathcal{O}_{q^*V}(1) \to \mathbb{P}(q^*(V))$ (the dual of the tautological line bundle over $\mathbb{P}(q^*(V))$).

\begin{proposition}\label{prop:equivthom3}
Let $U$ be an orientable real vector bundle. Then:
\begin{itemize}
\item{There is a naturally defined isomorphism $H^*_{S^1}( S_{V,U} , S_U ; \mathbb{Z}) \cong H^*( Y , \partial Y ; \mathbb{Z})$ of $\mathcal{H}^*$-modules.}
\item{$H^*_{S^1}( S_{V,U} , S_U ; \mathbb{Z})$ is a free module over $H^*(B ; \mathbb{Z})$ with basis $\{ (\delta \tau_{0,U}) x^i\}_{i=0}^{a-1}$, where $\delta : H^{*-1}_{S^1}( S_U , B_{V,U} ; \mathbb{Z}) \to H^*_{S^1}( S_{V,U} , S_U ; \mathbb{Z})$ is the coboundary map in the long exact sequence of the triple $( S_{V,U} , S_U , B_{V,U})$.}
\end{itemize}
\end{proposition}
\begin{proof}
Using excision, we have isomorphisms:
\[
H^*_{S^1}( S_{V,U} , S_U ; \mathbb{Z}) \cong H^*_{S^1}( S_{V,U} , \tilde{U} ; \mathbb{Z}) \cong H^*_{S^1}( \widetilde{Y} , \partial \widetilde{Y} ; \mathbb{Z}) \cong H^*(Y , \partial Y ; \mathbb{Z}).
\]
Next, we observe that there is an isomorphism 
\[
H^*_{S^1}( S_{V,U} , S_U ; \mathbb{Z}) \cong H^*_{S^1}( D(V\oplus U) , S(V\oplus U) \smallsmile D(U) ; \mathbb{Z}).
\]
The long exact sequence of the triple $(D(V \oplus U) , S(V \oplus U) \smallsmile D(U) , S(V \oplus U))$ has the form:
\begin{equation*}
\begin{aligned}
\cdots && \buildrel \delta \over \longrightarrow H^*_{S^1}( D(V\oplus U) , S(V\oplus U) \smallsmile D(U) ; \mathbb{Z}) \to H^*_{S^1}(D(V\oplus U) , S(V\oplus U) ; \mathbb{Z}) \to \\
&& \to H^*_{S^1}( S(V\oplus U) \smallsmile D(U) , S(V\oplus U) ; \mathbb{Z}) \buildrel \delta \over \longrightarrow \cdots
\end{aligned}
\end{equation*}
But $H^*_{S^1}( S(V\oplus U) \smallsmile D(U) , S(V\oplus U) ; \mathbb{Z}) \cong H^*_{S^1}( S_U , B_U ; \mathbb{Z}) \cong \mathcal{H}^* \tau_{0,U}$ and the map 
\[
\mathcal{H}^* \tau_{V,U} \cong H^*_{S^1}(D(V\oplus U) , S(V\oplus U) ; \mathbb{Z}) \to H^*_{S^1}( S_U , B_U ; \mathbb{Z}) \cong H^*_{S^1}(D(V) ; \mathbb{Z}) \cong \mathcal{H}^*,
\]
which corresponds to pullback under the zero section of $V$, is given by $\tau_{V,U} \mapsto e_V \tau_{0,U}$. In particular, this map is injective and the image is $\mathcal{H}^* e_V \tau_{0,U}$. Hence the coboundary map induces isomorphisms
\[
\delta : \mathcal{H}^{i-1}/\langle e_V \rangle \to H^{i+b}_{S^1}( D(V) , S(V) ; \mathbb{Z}), \quad y \mapsto y \cdot \delta \tau_{0,U}.
\]
Recall that the equivariant Euler class $e_V$ has the form
\[
e_V = x^a + c_1(V)x^{a-1} + \cdots + c_a(V)
\]
and it follows easily that $\mathcal{H}^*/\langle e_V \rangle$ is a free $H^*(B; \mathbb{Z})$ module generated by $\{ x^i \}_{i=0}^{a-1}$. Hence $H^*_{S^1}( D(V) , S(V) ; \mathbb{Z})$ is a free $H^*(B ; \mathbb{Z})$ module generated by $\{ \delta \tau_{0,U} x^i \}_{i=0}^{a-1}$.
\end{proof}

As with other results in this subsection, Proposition \ref{prop:equivthom3} extends to the case that $U$ is not orientable using local coefficients.

Let $\pi_Y : Y \to B$ denote the projection map from $Y$ to $B$. One sees that $Y$ has the structure of a locally trivial fibre bundle over $B$ whose fibres are manifolds with boundary. An orientation of $U$ determines a relative orientation of $Y \to B$, that is, an orientation of $TY \oplus \pi_Y^*(TB)$. From this we obtain a pushforward map
\[
(\pi_Y)_* : H^i(Y , \partial Y ; \mathbb{Z}) \to H^{i-l}(B ; \mathbb{Z}),
\]
where $l = 2a+b-1$ is the fibre dimension. This map is defined by the following commutative diagram:
\begin{equation*}\xymatrix{
H^i(Y , \partial Y ; \mathbb{Z}) \ar[rr]^-{(\pi_Y)_*} \ar[d]^-{\cong} & & H^{i-l}(B ; \mathbb{Z}) \ar[d]^-{\cong} \\
H_{dim(B)+l-i}( Y ; \pi_Y^*\mathbb{Z}_B) \ar[rr]^-{(\pi_Y)_*} & & H_{dim(B)+l-i}(B ; \mathbb{Z}_B)
}
\end{equation*}
where the vertical arrows are given by Poincar\'e-Lefschetz duality and $\mathbb{Z}_B$ denotes the orientation local system on $B$.

\begin{lemma}\label{lem:pushforward}
Let $Y$ be a compact manifold with boundary and let $A$ be a local system of abelian groups on $Y$. Let $\iota : \partial Y \to Y$ denote the inclusion of the boundary. Let $\mathbb{Z}_Y$ denote the orientation local system on $Y$. Then $\delta = \iota_*$, where $\delta$ is the coboundary map $\delta : H^{i-1}(\partial Y ; \iota^*A) \to H^i(Y , \partial Y ; A)$ in the long exact sequence of the pair $(Y , \partial Y)$ and $\iota_* : H^{i-1}(\partial Y ; \iota^*A) \to H^i(Y , \partial Y ; A)$ is the pushforward map defined by the commutative diagram:
\begin{equation*}\xymatrix{
H^{i-1}(\partial Y ; \iota^* A) \ar[r]^-{\iota_*} \ar[d]^-{\cong} & H^{i}(Y , \partial Y ; A) \ar[d]^-{\cong} \\
H_{dim(Y)-i}( \partial Y ; \iota^* (A \otimes \mathbb{Z}_Y) ) \ar[r]^-{\iota_*} & H_{dim(Y)-i}(Y ; A \otimes \mathbb{Z}_Y)
}
\end{equation*}
where the vertical arrows are given by Poincar\'e-Lefschetz duality.
\end{lemma}
\begin{proof}
For a pair of spaces $(U,V)$ and local system $A$, let $C_j( U , V ; A)$ and $C^j(U , V ; A)$ denote the singular chain/cochain complexes. Let $\omega \in C^j( \partial Y ; \iota^* A)$ represent a class $[\omega] \in H^j(\partial Y ; \iota^* A)$. Choose an extension $\widetilde{\omega} \in C^j(Y ; A)$ of $\omega$ to $Y$. Then $\delta [\omega] = [ \delta \widetilde{\omega}]$.

Next let $[Y] \in H_{n}( Y ; \partial Y ; \mathbb{Z}_Y )$ denote the fundamental class of $Y$, where $n = dim(Y)$. Taking cap product with $[Y]$ gives maps:
\begin{equation*}
\begin{aligned}
\! \, [Y] \cap \; &: H^j(Y ; A) \to H_{n-j}( Y , \partial X ; A \otimes \mathbb{Z}_Y), \\
\! \, [Y] \cap \; &: H^j(Y , \partial Y ; A) \to H_{n-j}( Y ; A \otimes \mathbb{Z}_Y),
\end{aligned}
\end{equation*}
which are the isomorphisms defining Poincar\'e-Lefschetz duality. Moreover, the fundamental classes of $Y$ and $\partial Y$ are related by
\[
\partial [Y] = [\partial Y],
\]
where $\partial : H_n(Y , \partial Y ; \mathbb{Z}_Y) \to H_{n-1}(\partial Y ; \iota^* \mathbb{Z}_Y)$ is the boundary map in the long exact sequence of the pair $(Y , \partial Y)$. Let $[\widetilde{Y}], [\widetilde{\partial Y}]$ be representatives of $[Y], [\partial Y]$ at the chain level. Then we have:
\[
\partial [\widetilde{Y}] = [\widetilde{\partial Y}] + \text{exact term}.
\]
The pushforward map $\iota_*$ is characterised at the chain level by
\[
[\widetilde{Y}] \cap \iota_* \omega = \iota_*( [\widetilde{\partial Y}] \cap \omega ) + \text{exact term}.
\]
However, we have:
\begin{equation*}
\begin{aligned}
\iota_*( [\widetilde{\partial Y}] \cap \omega ) &= \iota_*( [\widetilde{\partial Y}] \cap \iota^* \widetilde{\omega} ) \text{ (since $\iota^* \widetilde{\omega} = \omega$)} \\
&= \iota_*[\widetilde{\partial Y}] \cap \widetilde{\omega} \\
&= \partial [\widetilde{Y}] \cap \widetilde{\omega} + \text{exact term} \text{ (since $\partial [Y] = [\partial Y]$)} \\
&= [\widetilde{Y}] \cap \delta \widetilde{\omega} + \text{exact term} \text{ (using the Leibniz rule for cap products)}.
\end{aligned}
\end{equation*}
So at the level of (co)homology classes $[Y] \cap \iota_* \omega = [Y] \cap \delta[\omega]$ and thus $\iota_* \omega = \delta \omega$ as claimed.
\end{proof}

\begin{proposition}\label{prop:pushforwardsegre}
Suppose $U$ is orientable and fix an orientation. The fibre integration map
\[
(\pi_Y)_* : H^i(Y , \partial Y ; \mathbb{Z}) \to H^{i-l}(B ; \mathbb{Z})
\]
is a morphism of $H^*(B ; \mathbb{Z})$-modules and is given by:
\[
(\pi_Y)_*( \delta \tau_{0 , U} x^j) = \begin{cases} 0 & \text{if } j < a-1, \\ s_{j-(a-1)}(V) & \text{if } j \ge a-1, \end{cases}
\]
where $s_j(V) \in H^{2j}(B ; \mathbb{Z})$ is the {\em $j$-th Segre class} of $V$ \cite[\textsection 3.2]{fult}, which is characterised by the property that if we let 
\[
s(V) = s_0(V) + s_1(V) + s_2(V) + \cdots \in H^{ev}(B ; \mathbb{Z})
\]
denote the total Segre class of $V$ and similarly $c(V) = 1 + c_1(V) + c_2(V) + \cdots $ the total Chern class, then
\[
c(V)s(V) = 1.
\]
In the case that $U$ is not orientable, the analogous result holds using local coefficients.
\end{proposition}
\begin{proof}
Recall that $\widetilde{U}$ is an equivariant tubular neighbourhood of $S_U$ in $S_{V,U}$. Hence the inclusion of triples $(S_{V,U} , S_U , B_{V,U} ) \to (S_{V,U} , \widetilde{U} , B_{V,U})$ induces a commutative diagram of long exact sequences where the vertical maps are isomorphisms
\[
\xymatrix{
\cdots \ar[r] & H^{*-1}_{S^1}( S_U , B_{V,U} ; \mathbb{Z} ) \ar[r]^-\delta  & H^*_{S^1}(S_{V,U} , S_U ; \mathbb{Z} ) \ar[r] & H^*_{S^1}( S_{V,U} , B_{V,U} ; \mathbb{Z} ) \ar[r] & \cdots \\
\cdots \ar[r] & H^{*-1}_{S^1}( \widetilde{U} , B_{V,U} ; \mathbb{Z} ) \ar[r]^-\delta \ar[u] & H^*_{S^1}(S_{V,U} ,  \widetilde{U} ; \mathbb{Z} ) \ar[r] \ar[u] & H^*_{S^1}( S_{V,U} , B_{V,U} ; \mathbb{Z} ) \ar[r] \ar[u] & \cdots
}
\]
Similarly the inclusion $( \widetilde{Y} , \partial \widetilde{Y} , \emptyset ) \to (S_{V,U} , \widetilde{U} , B_{V,U})$ induces a commutative diagram of long exact sequences
\[
\xymatrix{
\cdots \ar[r] & H^{*-1}_{S^1}( \widetilde{U} , B_{V,U} ; \mathbb{Z} ) \ar[r]^-\delta \ar[d] & H^*_{S^1}(S_{V,U} ,  \widetilde{U} ; \mathbb{Z} ) \ar[r] \ar[d] & H^*_{S^1}( S_{V,U} , B_{V,U} ; \mathbb{Z} ) \ar[r] \ar[d] & \cdots \\
\cdots \ar[r] & H^{*-1}_{S^1}( \partial \widetilde{Y} ; \mathbb{Z} )  \ar[r]^-\delta & H^*_{S^1}( \widetilde{Y} ,  \partial \widetilde{Y} ; \mathbb{Z} ) \ar[r] & H^*_{S^1}( \widetilde{Y} ; \mathbb{Z} ) \ar[r] & \cdots
}
\]
In particular, combining these sequences we have a commutative diagram
\[
\xymatrix{
H^{*-1}_{S^1}( S_U , B_{V,U} ; \mathbb{Z} ) \ar[r]^-\delta \ar[d]^-{i^*} & H^*_{S^1}(S_{V,U} , S_U ; \mathbb{Z} ) \ar[d]^-{i^*} \\
H^{*-1}_{S^1}( \partial \widetilde{Y} ; \mathbb{Z} ) \ar[r]^-\delta & H^*_{S^1}( \widetilde{Y} ,  \partial \widetilde{Y} ; \mathbb{Z} )
}.
\]
Since $S^1$ acts freely on $\widetilde{Y}$ and $\partial \widetilde{Y}$, we may replace the equivariant cohomology groups on the bottow row of this diagram with corresponding non-equivariant cohomology groups of their quotients, giving a commutative diagram
\[
\xymatrix{
H^{*-1}_{S^1}( S_U , B_{V,U} ; \mathbb{Z} ) \ar[r]^-\delta \ar[d]^-{i^*} & H^*_{S^1}(S_{V,U} , S_U ; \mathbb{Z} ) \ar[d]^-{i^*} \\
H^{*-1}( \partial Y ; \mathbb{Z} ) \ar[r]^-\delta & H^*( Y ,  \partial Y ; \mathbb{Z} )
}.
\]
Moreover, the vertical arrow on the right is an isomorphism (it is the isomorphism constructed in the proof of Proposition \ref{prop:equivthom3}). By Proposition \ref{prop:equivthom3}, the classes $i^*( (\delta \tau_{0,U})x^i )$, $i = 0, \dots , a-1$ form a basis for $H^*(Y , \partial Y ; \mathbb{Z} )$ as a $H^*(B ; \mathbb{Z})$-module.

Clearly $(\pi_Y)_*$ is a morphisms of $H^*(B ; \mathbb{Z})$-modules. Let $\iota : \partial Y \to Y$ denote the inclusion and let $\pi_{\partial Y} : \partial Y \to B$ be the projection to $B$. Note that $\pi_{\partial Y} = \pi_Y \circ \iota$. It follows that $(\pi_{\partial Y})_* = (\pi_Y)_* \circ \iota_*$ and from Lemma \ref{lem:pushforward}, we have that $\iota_* = \delta$, where $\delta : H^{i-1}(\partial Y ; \mathbb{Z}) \to H^i(Y , \partial Y ; \mathbb{Z})$ is the coboundary map of the long exact sequence of the pair $(Y , \partial Y)$. Therefore
\[
(\pi_Y)_*( i^*\delta \tau_{0,U} x^j ) = (\pi_Y)_* \circ \delta ( i^*(\tau_{0,U}) x^j ) = (\pi_{\partial Y})_*( i^*(\tau_{0,U}) x^j ).
\]
Next, we recall that $\partial Y = q^*( \mathbb{P}(V))$, where $q$ is the projection $q : S_U \to B$. Note that $i^*(\tau_{0,U})$ is precisely the Thom class of the bundle $U \to B$ and it follows that
\[
(\pi_{\partial Y})_*( i^*(\tau_{0,U}) x^j) = (\pi_{\mathbb{P}(V)})_*( x^j ),
\]
where $\pi_{\mathbb{P}(V)}$ is the projection $\pi_{\mathbb{P}(V)} : \mathbb{P}(V) \to B$. If $j < a-1$ then the degree of $x^j$ is less than the fibre dimension of $\mathbb{P}(V)$, so we obviously have $(\pi_{\mathbb{P}(V)})_*( x^j ) = 0$ in this case. So it remains to prove the identity
\begin{equation}\label{equ:integration}
(\pi_{\mathbb{P}(V)})_*(x^{j+(a-1)}) = s_j(V), \quad \text{for all $j \ge 0$.}
\end{equation}
We prove this by induction on $j$. The case $j=0$ is obvious since $s_0(V) = 1$ and $x^{a-1}$ restricts to a generator of $H^{2a-2}(\mathbb{CP}^{a-1} ; \mathbb{Z})$ on each fibre of $\mathbb{P}(V)$. Now recall that on $\mathbb{P}(V)$ the class $x$ satisfies the equation
\[
x^a + c_1(V)x^{a-1} + \cdots + c_a(V) = 0
\]
(this is the Grothendieck approach to defining the Chern classes of a complex vector bundle $V$). Now let $j \ge 0$ and suppose Equation (\ref{equ:integration}) has been proven for all $j' \le j$. Then (letting $s_k(V) = 0$ if $k<0$) we have:
\begin{equation*}
\begin{aligned}
(\pi_{\mathbb{P}(V)})_*( x^{(j+1)+(a-1)} ) &= (\pi_{\mathbb{P}(V)})_*( x^{j+a} ) \\
&= (\pi_{\mathbb{P}(V)})_*( x^j( -c_1(V)x^{a-1} -c_2(V)_x^{a-2} - \cdots - c_a(V) ) \\
&= -( c_1(V)s_j(V) - c_2(V)s_{j-1}(V) - \cdots - c_a(V)s_{j+1-a} ) \\
&= s_{j+1}(V),
\end{aligned}
\end{equation*}
where we used the relation
\[
s_{j+1}(V)c_0(V) + s_j(V)c_1(V) + \cdots + s_{j+1-a}(V) c_a(V) = 0 \quad \text{for all $j \ge 0$},
\]
which follows by taking the degree $2j+2$ part of the identity $c(V)s(V) = 1$.
\end{proof}

\subsection{Cohomological formula for the families Seiberg-Witten invariants}\label{sec:swcohomology}

Let $f : S_{V , U} \to S_{V' , U'}$ be a finite dimensional monopole map. Then as in Subsection \ref{sec:fsw} we may assume that $U' = U \oplus H^+$ and that $f|_U$ is the inclusion $U \to U'$. For convenience, we will assume throughout this section that $U,U'$ are oriented. As usual, the results easily extend to the general case provided we use local coefficients. As before we assume $f$ preserves sections at infinity:
\[
f : (S_{V,U} , B_{V,U}) \to (S_{V',U'} , B_{V',U'} ).
\]
Let $[\phi] \in \mathcal{CH}(f)$ be a chamber and let $\phi : B \to U' - U$ be a representative section. Since $\phi$ avoids $U$, we obtain a pushforward map
\[
\phi_* : H^j_{S^1}(B ; \mathbb{Z}) \to H^{j+2a'+b'}_{S^1}( S_{V',U'} , S_U ; \mathbb{Z} )
\]
and we also have a pullback map
\[
f^* : H^j_{S^1}( S_{V',U'} , S_U ; \mathbb{Z} ) \to H^j_{S^1}( S_{V,U} , S_U ; \mathbb{Z} ) \cong H^j( Y , \partial Y ; \mathbb{Z})
\]

\begin{theorem}\label{thm:pdclass}
For each $m\ge 0$ we have
\[
SW_m(f,\phi) = (\pi_Y)_*( x^m \smallsmile f^*\phi_*(1) ).
\]
\end{theorem}
The proof will be given later in this subsection. But first we make some observations about this formula:
\begin{itemize}
\item{This formula shows that $SW_m(f,\phi)$ depends only of $f$ and $\phi$ and not on a choice of perturbation $f'$ of $f$ meeting $\phi$ transversally.}
\item{This formula also shows that $SW_m(f , \phi)$ depends only on the homotopy classes of $f$ and $\phi$, in particular $SW_m(f,\phi)$ depends only on $\phi$ only through the chamber $[\phi]$ in which $\phi$ lies.}
\item{We will show in Proposition \ref{prop:swsuspension} that $SW_m(f,\phi)$ depends only on the stable homotopy class of $f$. More precisely $SW_m(f,\phi)$ is invariant under stabilisation of $f$ by real or complex vector bundles.}
\end{itemize}

Let $\widetilde{\tau}_{V',U'}^\phi = \phi_*(1) \in H^{2a'+b'}( S_{V',U'} , S_U ; \mathbb{Z})$. Observe that the image of $\widetilde{\tau}_{V',U'}^\phi$ under the map $H^{2a'+b'}(S_{V',U'} , S_U ; \mathbb{Z} ) \to H^{2a'+b'}( S_{V',U'} , B_{V',U'} ; \mathbb{Z})$ is $\tau_{V',U'}$, the equivariant Thom class. This is simply because $\phi$ is a section of $V' \oplus U'$, so the induced pushforward map $\phi_* : H^j_{S^1}( B ; \mathbb{Z}) \to H^{j+2a'+b'}_{S^1}( S_{V',U'} , B_{V',U'} ; \mathbb{Z})$ sends $1$ to the equivariant Thom class.

Next, we note that $f^*(\widetilde{\tau}_{V',U'}^\phi) \in H^{2a'+b'}_{S^1}( S_{V,U} , S_U ; \mathbb{Z})$ can be written uniquely as
\begin{equation}\label{equ:muclasses}
f^*(\widetilde{\tau}_{V',U'}^\phi) = \sum_{j=0}^{a-1} \mu_{j}(f,\phi) \delta \tau_{0,U} x^{(a-1)-j},
\end{equation}
where $\mu_{j}(f,\phi) \in H^{2j-(2d-b^+-1)}(B ; \mathbb{Z})$. Note that since $H^k(B ; \mathbb{Z}) = 0$ for $k>dim(B)$, we have that $\mu_j = 0$ whenever $2j - (2d-b^+-1) > dim(B)$, that is, $\mu_j = 0$ whenever $2j \ge 2d -(b^+ - dim(B))$.

Let $f'$ be a deformation of $f$ such that $f'$ meets $\phi$ transversally. Then $\widetilde{\mathcal{M}}_\phi = {f'}^{-1}( \phi(B))$ is a compact embedded submanifold of $S_{V,U}$ without boundary. Recall that we define $\widetilde{Y}$ to be the complement in $S_{V,U}$ of an equivariant tubular neighbourhood of $S_U$ and that $\widetilde{Y}$ is a compact manifold with boundary. Choosing the neighbourhood sufficiently small, we may assume that $\widetilde{\mathcal{M}}_\phi$ is an embedded submanifold of $\widetilde{Y}$ and that $\widetilde{\mathcal{M}}_\phi \cap \partial \widetilde{Y} = \emptyset$. Noting that $S^1$ acts freely on $\widetilde{\mathcal{M}}_\phi$, we let $\mathcal{M}_\phi = \widetilde{\mathcal{M}}_\phi/S^1$, which is an embedded submanifold of $Y = \widetilde{Y}/S^1$. Let $\mathcal{L} \to \mathcal{M}_\phi$ be the line bundle associated to the principal circle bundle $\widetilde{\mathcal{M}}_\phi \to \mathcal{M}_\phi$. Then clearly $c_1(\mathcal{L}) = x|_{{\mathcal{M}}_\phi}$. Note that the dimension of $\mathcal{M}_\phi$ is given by
\[
dim(\mathcal{M}_\phi) = (2a+b) - (2a'+b') - 1 = 2d-b^+-1.
\]
Recall that a choice of orientation of $U$ defines a relative orientation of $S_{V,U} \to B$, hence also a relative orientation of $\pi_Y : Y \to B$. Moreover a choice of orientation on $U'$ defines an orientation of the normal bundle of $\mathcal{M}_\phi \subseteq Y$. Let $\pi_{{\mathcal{M}}_\phi} : \mathcal{M}_\phi \to B$ be the restriction of $\pi_Y$ to $\mathcal{M}_\phi$. We see that an orientation of $U \oplus U'$ (equivalently, an orientation of $H^+$) determines a relative orientation of $\pi_{{\mathcal{M}}_\phi} : \mathcal{M}_\phi \to B$ and hence a push-forward map $(\pi_{{\mathcal{M}}_\phi})_* : H^j( \mathcal{M}_\phi ; \mathbb{Z}) \to H^{j-(2d-b^+-1)}(B ; \mathbb{Z})$. Recall that for each $m \ge 0$, we have defined the Seiberg-Witten invariants of $(f,\phi)$ to be given by:
\[
SW_m(f,\phi) = (\pi_{{\mathcal{M}}_\phi})_*( c_1(\mathcal{L})^m ) \in H^{2m - (2d-b^+-1)}( B ; \mathbb{Z}).
\]
Just as we saw for the $\mu_j(f,\phi)$, since $H^k(B ; \mathbb{Z}) = 0$ for $k>dim(B)$, we have that $SW_m(f,\phi) = 0$ whenever $2m \ge 2d - (b^+-dim(B))$.

\begin{proposition}
The $SW_m(f,\phi)$ are related to the classes $\mu_j(f,\phi) \in H^{2j - (2d-b^+-1)}(B ; \mathbb{Z})$ by:
\begin{equation}\label{equ:mutosw}
\sum_{m \ge 0} SW_m(f,\phi) t^m = \left( \sum_{j = 0}^{a-1} \mu_j(f,\phi) t^j \right) \left( \sum_{n \ge 0} s_n(V) t^n \right)
\end{equation}
where $t$ is a formal variable and $s_n(V)$ is the $n-$th Segre class of $V$. This relation can be inverted to express the $\mu_j(f,\phi)$ invariants in terms of the $SW_m(f,\phi)$ invariants:
\begin{equation}\label{equ:swtomu}
\sum_{j=0}^{a-1} \mu_j(f,\phi) t^j = \left( \sum_{m \ge 0} SW_m(f,\phi) t^m \right) \left( \sum_{n \ge 0} c_n(V) t^n \right),
\end{equation}
where $c_n(V)$ is the $n$-th Chern class of $V$.
\end{proposition}
\begin{proof}
By Theorem \ref{thm:pdclass} we have
\[
SW_m(f,\phi) = (\pi_Y)_*( x^m f^*\phi_*(1) ).
\]
Recall that we set $\widetilde{\tau}_{V',U'}^\phi = \phi_*(1)$. Then combined with Equation (\ref{equ:muclasses}) and using Proposition \ref{prop:pushforwardsegre}, we get:
\begin{equation*}
\begin{aligned}
SW_m(f,\phi) &= (\pi_Y)_* \left( \sum_{j=0}^{a-1} \mu_{a-1-j}(f,\phi) \delta \tau_{0,U} x^{j+m} \right) \\
& = \sum_{j=0}^{a-1} \mu_{a-1-j}(f,\phi) (\pi_Y)_*( \delta \tau_{0,U} x^{j+m} ) \\
&= \sum_{j=0}^{a-1} \mu_{a-1-j}(f,\phi) s_{j+m-(a-1)}(V) \\
&= \sum_{j+n = m} \mu_j(f,\phi) s_n(V), \\
\end{aligned}
\end{equation*}
which proves Equation (\ref{equ:mutosw}). Equation (\ref{equ:swtomu}) follows from Equation (\ref{equ:mutosw}) and the identity
\[
(c_0(V) + tc_1(V) + \cdots )(s_0(V) + ts_1(V) + ts_2(V) + \cdots ) = 1.
\]
\end{proof}

\begin{proof}[Proof of Theorem \ref{thm:pdclass}]
Let $\mathbb{Z}_B$ denote the orientation local system on $B$. Fix choices of orientations on $U$ and $U'$. This gives an isomorphism of the orientation local system of $Y$ with $(\pi_Y)^* \mathbb{Z}_B$ and an isomorphism of the orientation local system of $\mathcal{M}_\phi$ with $(\pi_{{\mathcal{M}}_\phi})^* \mathbb{Z}_B$. The submanifold $\mathcal{M}_\phi \subset Y$ has a fundamental class $[\mathcal{M}_\phi] \in H_{2d-b^+-1}(Y ; (\pi_Y)^*\mathbb{Z}_B )$. Let $\eta_{{\mathcal{M}}_\phi}$ denote the Poincar\'e dual class $\eta_{{\mathcal{M}}_\phi} \in H^{dim(X) - (2d-b^+-1)}( Y , \partial Y ; \mathbb{Z})$. Let $\iota : \mathcal{M}_\phi \to Y$ denote the inclusion. We obtain a pushforward map $\iota_* : H^j(\mathcal{M}_\phi ; \mathbb{Z}) \to H^{j + (2d-b^+-1)}(X , \partial X ; \mathbb{Z})$. From general properties of the pushforward map we have that if $\theta \in H^j( Y ; \mathbb{Z} )$, then $\iota_*( \iota^* \theta) = \theta \smallsmile \eta_{{\mathcal{M}}_\phi}$. Therefore
\begin{equation}\label{equ:swm}
SW_m(f,\phi) = (\pi_{{\mathcal{M}}_\phi})_* ( c_1(\mathcal{L})^m ) = (\pi_Y)_* \circ \iota_*( \iota^* (x^m)) = (\pi_Y)_*( x^m \smallsmile \eta_{{\mathcal{M}}_\phi} ).
\end{equation}
To complete the proof, we just need to show that $\eta_{{\mathcal{M}}_\phi} = f^*(\phi_*(1))$. We begin by observing that the normal bundle $N_{{\widetilde{\mathcal{M}}}_\phi}$ to $\widetilde{\mathcal{M}}_\phi$ in $\widetilde{Y}$ is equivariantly identified via the derivative of $f'$ with the normal bundle of $B_\phi = \phi(B) \subseteq S_{V',U'}$. But since $\phi$ is a section, the normal bundle of $B_\phi$ can be identified with the restriction of the vertical tangent bundle of $S_{V',U'} \to B$ to $B_\phi$, which is exactly $V' \oplus U'$. Thus
\[
N_{{\widetilde{\mathcal{M}}}_\phi} = {f'}^*(V' \oplus U') \cong {f}^*(V' \oplus U').
\]
Since $S^1$ acts freely on $\mathcal{M}_\phi$ with quotient $\mathcal{M}_\phi$, we have that $N_{{\widetilde{\mathcal{M}}}_\phi}$ descends to a vector bundle $N_{\mathcal{M}_\phi}$ on $\mathcal{M}_\phi$, which is precisely the normal bundle of $\mathcal{M}_\phi$ in $Y$. Let $U$ denote a tubular neighbourhood of $\mathcal{M}_\phi$ in $Y$, which we identify with the open unit disc bundle in $N_{{\mathcal{M}}_\phi}$. The inclusion $\iota : \mathcal{M}_\phi \to Y$ factors as:
\[
\mathcal{M}_\phi \buildrel \zeta \over \longrightarrow U \buildrel j_U \over \longrightarrow Y,
\]
where $\zeta$ is the zero section of $N_{{\mathcal{M}}_\phi}$ and $j_U$ is the inclusion map, which is an open immersion. Passing to cohomology, we obtain a factorisation of $\iota_*$ as:
\[
H^j(\mathcal{M}_\phi ; \mathbb{Z}) \buildrel \zeta_* \over \longrightarrow H^{j +(2d-b^+-1)}_{c}(U ; \mathbb{Z}) \buildrel (j_U)_* \over \longrightarrow H^{j+(2d-b^+-1)}(Y , \partial Y ; \mathbb{Z}),
\]
where $H^{j +(2d-b^+-1)}_{c}(U ; \mathbb{Z})$ denotes compactly supported cohomology of $U$ of degree $j+(2d-b^+-1)$. The pushforward $\zeta_* :  H^j(\mathcal{M}_\phi ; \mathbb{Z}) \to H^{j +(2d-b^+-1)}_{c}(U ; \mathbb{Z})$ is the Thom isomorphism for $N_{{\mathcal{M}}_\phi} \cong U$. Next we note that 
\[
\eta_{{\mathcal{M}}_\phi} = \iota_*(1) = (j_U)_* \circ \zeta_*(1) = (j_U)_*( \tau_{N_{{\mathcal{M}}_\phi}} ),
\]
where $\tau_{N_{{\mathcal{M}}_\phi}} = \zeta_*(1) \in H^{2a'+b'}_c( U ; \mathbb{Z}) \cong H^{2a'+b'}( U^+ , \infty ; \mathbb{Z})$ is the Thom class of $N_{{\mathcal{M}}_\phi}$. Instead of computing $(j_U)_*( \tau_{N_{{\mathcal{M}}_\phi}} )$ directly we will work $S^1$-equivariantly on $\widetilde{Y}$. The inclusion $\widetilde{\iota} : \widetilde{M} \to \widetilde{Y}$ factors as:
\[
\widetilde{\mathcal{M}} \buildrel \widetilde{\zeta} \over \longrightarrow \widetilde{U} \buildrel \widetilde{j_U} \over \longrightarrow \widetilde{Y},
\]
where $\widetilde{U}$ is the open disc bundle in $N_{\widetilde{{\mathcal{M}_\phi}}}$.

Note that $H^{2a'+b'}_{S^1}( \widetilde{U}^+ , \infty ; \mathbb{Z})$ is isomorphic to $H^{2a'+b'}(U^+ , \infty ; \mathbb{Z})$. To see this, let $O$ be an open neighbourhood of $\infty$ in $U^+$ which strongly deformation retracts to $\infty$ and let $\widetilde{O}$ be the pre-image of $O$ in $\widetilde{U}^+$. Note that $S^1$ acts freely on $\widetilde{O} \setminus \infty$. Then we have a sequence of isomorphisms:
\begin{equation*}
\begin{aligned}
H^*_{S^1}( \widetilde{U}^+ , \infty ; \mathbb{Z}) & \cong H^*_{S^1}( \widetilde{U}^+ , \widetilde{O} ; \mathbb{Z}) \text{ (by deformation retraction)} \\
& \cong H^*_{S^1}( \widetilde{U} , \widetilde{O} \setminus \infty ; \mathbb{Z} ) \text{ (excision)} \\
& \cong H^*( U , O \setminus \infty ; \mathbb{Z}) \text{ ($S^1$ acts freely)} \\
& \cong H^*( U^+ , O ; \mathbb{Z}) \text{ (excision)} \\
& \cong H^*( U^+ , \infty ; \mathbb{Z}) \text{ (by deformation retraction)}.
\end{aligned}
\end{equation*}
Henceforth we use this sequence of isomorphisms to identity $H^{2a'+b'}_{S^1}( \widetilde{U}^+ , \infty ; \mathbb{Z})$ with $H^{2a'+b'}(U^+ , \infty ; \mathbb{Z})$. Let $\tau_{N_{\widetilde{{\mathcal{M}_\phi}}}} \in H^{2a'+b'}_{S^1}( \widetilde{U}^+ , \infty ; \mathbb{Z})$ denote the equivariant Thom class of $N_{\widetilde{{\mathcal{M}_\phi}}}$. Then $\tau_{N_{\widetilde{{\mathcal{M}_\phi}}}}$ agrees with $\tau_{N_{{\mathcal{M}_\phi}}}$ under the isomorphism $H^{2a'+b'}_{S^1}( \widetilde{U}^+ , \infty ; \mathbb{Z}) \cong H^{2a'+b'}( U^+ , \infty ; \mathbb{Z})$ constructed above. Let $(\widetilde{j_U})_* : H^*_{S^1}( \widetilde{U}^+ , \infty ; \mathbb{Z}) \to H^*_{S^1}(\widetilde{Y} , \partial \widetilde{Y} ; \mathbb{Z})$ be the pushforward with respect to $\widetilde{j_U}$. Then we have a commutative diagram:
\begin{equation*}\xymatrix{
H^*_{S^1}(\widetilde{U}^+ , \infty ; \mathbb{Z}) \ar[d]^-\cong \ar[r]^-{(\widetilde{j_U})_*} & H^*_{S^1}(\widetilde{Y} , \partial \widetilde{Y} ; \mathbb{Z}) \\
H^*(U^+ , \infty ; \mathbb{Z}) \ar[d]^-\cong & H^*(Y , \partial Y ; \mathbb{Z}) \ar[u]^-\cong \\
H^*_c(U ; \mathbb{Z}) \ar[ur]^-{(j_U)_*} &
}
\end{equation*}
Under the isomorphism $H^*_c(U ; \mathbb{Z})  \cong H^*(U^+ , \infty ; \mathbb{Z})$, the pushforward map $(j_U)_* : H^*(U^+ , \infty ; \mathbb{Z}) \to H^*(Y , \partial Y ; \mathbb{Z})$ may be constructed as follows: let $q: Y \to U^+$ be the map given by $q(x) = x$ if $x\in U$ and $q(x) = \infty$ otherwise. Noting that $q$ sends $\partial Y$ to $\infty$ we have that $q$ gives a map of pairs $q : (Y , \partial Y) \to (U^+ , \infty)$ and we have $(j_U)_* = q^*$. Similarly define $\widetilde{q} : \widetilde{Y} \to \widetilde{U}^+$ by $\widetilde{q}(x) = x$ if $x \in \widetilde{U}$ and $\widetilde{q}(x) = \infty$ otherwise. Tracing through the sequence of isomorphisms defining $(\widetilde{j_U})_*$, one may verify that we likewise have $(\widetilde{j_U})_* = \widetilde{q}^*$.

Our goal now will be to compute $\widetilde{q}^*(\tau_{N_{\widetilde{{\mathcal{M}_\phi}}}}) = (\widetilde{j_U})_*( \tau_{N_{\widetilde{{\mathcal{M}_\phi}}}} ) \in H^{2a'+b'}_{S^1}( \widetilde{Y} , \partial \widetilde{Y} ; \mathbb{Z})$, which of course coincides with $(j_U)_*(\tau_{N_\mathcal{M}}) = \eta_{\mathcal{M}}$ under the isomorphism \linebreak $H^{2a'+b'}_{S^1}( \widetilde{Y} , \partial \widetilde{Y} ; \mathbb{Z}) \cong H^{2a'+b'}( Y , \partial Y ; \mathbb{Z})$. Recall that $f'$ is transverse to $B_\phi$, and the points of $B_\phi$ are fixed points of the $S^1$-action and disjoint from $S_U$. Since $f'$ is transverse to $B_\phi$, we can find an $S^1$-equivariant open tubular neighbourhood $W_\phi \subseteq S_{U',V'}$ of $B_\phi$ such that each point of $W_\phi$ is a regular value of $f'$. Note also that $W_\phi$ can be identified with the open unit disc bundle in the normal bundle of $B_\phi$ in $S_{U',V'}$, which is isomorphic $U'\oplus V'$. Then (possibly after shrinking $W_\phi$ and $\widetilde{U}$), we can assume that $\widetilde{U} = {f'}^{-1}(W_\phi)$ and that $f'|_{\widetilde{U}} \to W_\phi$ is a surjective submersion. By choosing $W_\phi$ small enough, we may assume that $W_\phi$ is disjoint from $S_U$. Let $j_{W_\phi} : W_\phi \to S_{U',V'}$ denote the inclusion map. Consider the following commutative diagram of $S^1$-equivariant maps:
\begin{equation*}\xymatrix{
\widetilde{Y} \ar[rr]^-{f'|_{\widetilde{Y}}} & & S_{U',V'} \\
\widetilde{U} \ar[u]^-{\widetilde{j_U}} \ar[rr]^-{f'|_{\widetilde{U}}} & & W_\phi \ar[u]^-{j_{W\phi}}
}
\end{equation*}
Let $\tau_{U',V'}^\phi \in H^{2a'+b'}_{S^1}( W_\phi^+ , \infty ; \mathbb{Z})$ be the equivariant Thom class. Then $\tau_{N_{\widetilde{{\mathcal{M}_\phi}}}} = (f'|_{\widetilde{U}})^*(\tau_{U',V'}^\phi)$. So under the isomorphism $H^{2a'+b'}( Y  , \partial Y ; \mathbb{Z}) \cong H^{2a'+b'}_{S^1}(\widetilde{Y} , \partial \widetilde{Y} ; \mathbb{Z})$, $\eta_{{\mathcal{M}_\phi}}$ is given by: 
\begin{equation}\label{equ:eta0}
\eta_{{\mathcal{M}_\phi}} = \widetilde{q}^* (f'|_{\widetilde{U}})^*(\tau_{U',V'}^\phi).
\end{equation}

Let $q_{W_\phi} : S_{U',V'} \to W_\phi^+$ be the map given by $q_{W_\phi}(x) = x$ if $x \in W_\phi$ and $q_{W_\phi}(x) = \infty$ otherwise. Then we have a commutative diagram of pairs:
\begin{equation*}\xymatrix{
( \widetilde{Y} , \partial \widetilde{Y} ) \ar[r]^-{f'|_{\widetilde{Y}}} \ar[d]^-{\widetilde{q}} & (S_{U',V'} , S_U) \ar[d]^-{q_{W_\phi}} \\
(\widetilde{U}^+ , \infty ) \ar[r]^-{f'|_{\widetilde{U}}} & (W_\phi^+ , \infty )
}
\end{equation*}
From commutativity of the diagram and Equation (\ref{equ:eta0}), we get:
\[
\eta_{{\mathcal{M}_\phi}} = (f'|_{\widetilde{Y}})^* q_{W_\phi}^*(\tau_{U',V'}^\phi).
\]
But note that $q_W^*(\tau_{U',V'}^\phi)$ is exactly the class $\widetilde{\tau}_{V',U'}^\phi = \phi_*(1)$. Thus
\[
\eta_{{\mathcal{M}_\phi}} = {f'}^*( \phi_*(1)) = f^*(\phi_*(1)).
\]
Substituting into Equation (\ref{equ:swm}), we get the desired equality:
\[
SW_m(f,\phi) = (\pi_Y)_*( x^m \smallsmile \eta_{{\mathcal{M}_\phi}} ) = (\pi_Y)_*( x^m \smallsmile f^*\phi_*(1) ).
\]
\end{proof}

\begin{proposition}\label{prop:swsuspension}
The Seiberg-Witten invariants $SW_m(f,\phi)$ are independent of stabilisation of $f$ by real or complex vector bundles.
\end{proposition}
\begin{proof}
We will focus on the case of stabilisation by a complex vector bundle, the real case being considerably simpler. Let $f : S_{V , U} \to S_{V' , U'}$ and a chamber $\phi : B \to (U' \setminus U)$ be given. Let $A$ be a complex vector bundle of rank $\alpha$, set $V_A = V \oplus A$, $V' = V' \oplus A$ and let $f_A : V_A \oplus U \to V'_A \oplus U$ denote the suspension of $f$. Observe that $S_{V'_A , U'}$ is the fibrewise smash product $S_{V'_A , U'} = S_{V',U'} \wedge_B S_{A,0}$ (equivariantly) and hence we have an external cup product
\[
\smallsmile \; : H^i_{S^1}( S_{V',U'} , S_U ; \mathbb{Z} ) \times H^j_{S^1}( S_{A,0} , B_{A,0} ; \mathbb{Z}) \to H^{i+j}_{S^1}( S_{V'_A,U'} , S_U ; \mathbb{Z}).
\]
Using this cup product we see that $\widetilde{\tau}^\phi_{V'_A,U'} = \widetilde{\tau}^\phi_{V',U'} \smallsmile \tau_{A,0}$ and hence 
\begin{equation}\label{equ:taususpend}
f_A^*(\widetilde{\tau}^\phi_{V'_A,U'}) = f^*(\widetilde{\tau}^\phi_{V',U'}) \smallsmile \tau_{A,0}.
\end{equation}
Recall from Proposition \ref{prop:equivthom3} that $H^*_{S^1}( S_{V,U} , S_U ; \mathbb{Z}) \cong H^*( Y , \partial Y ; \mathbb{Z})$ is a free module over $H^*(B ; \mathbb{Z})$ with basis $\{ (\delta \tau_{0,U}) x^i\}_{i=0}^{a-1}$. Moreover, the proof of Proposition \ref{prop:equivthom3} shows that as $H^*_{S^1}(B ; \mathbb{Z})$-modules we have an isomorphism
\[
H^*_{S^1}( S_{V,U} , S_U ; \mathbb{Z}) \cong \left( \frac{ \mathbb{Z}[x] }{\langle e_V \rangle }\right) \delta \tau_{0,U}
\]
where $e_V \in H^a_{S^1}(B ; \mathbb{Z})$ is the equivariant Euler class of $V$. Similarly we have an isomorphism
\[
H^*_{S^1}( S_{V_A,U} , S_U ; \mathbb{Z}) \cong \left( \frac{ \mathbb{Z}[x] }{\langle e_V e_A \rangle }\right) \delta \tau_{0,U}
\]
where $e_A \in H^a_{S^1}(B ; \mathbb{Z})$ is the equivariant Euler class of $A$. Noting that $S_{V_A,U} = S_{V,U} \wedge S_{A,0}$ we see that the external cup product with $\tau_{A,0}$ defines a morphism of $H^*_{S^1}(B ; \mathbb{Z})$-modules:
\[
\smallsmile \tau_{A,0} : H^*_{S^1}( S_{V,U} , S_U ; \mathbb{Z} ) \to H^{*+\alpha}_{S^1}( S_{V_A,U} , S_U ; \mathbb{Z}).
\]
Tracing through the isomorphisms used in the proof of Proposition \ref{prop:equivthom3}, we see that $\smallsmile \tau_{A,0}$ corresponds to the morphism
\[
\left( \frac{ \mathbb{Z}[x] }{\langle e_V \rangle }\right) \delta \tau_{0,U} \to \left( \frac{ \mathbb{Z}[x] }{\langle e_V e_A \rangle }\right) \delta \tau_{0,U}, \quad u \mapsto u \cdot e_A
\]
given by multiplication by $e_A$. Let $\eta = f^*(\widetilde{\tau}^\phi_{V',U'}) \in H^{2a'+b'}_{S^1}( S_{V,U} , S_U ; \mathbb{Z} )$ so that $SW_m(f,\phi) = (\pi_Y)_*( x^m \smallsmile \eta)$, where $Y$ is defined as in Subsection \ref{sec:cohomological}. Similarly define $\eta_A = f^*_A(\widetilde{\tau}^\phi_{V'_A,U'}) \in H^{2a'+b'}_{S^1}(S_{V_A,U} , S_U ; \mathbb{Z})$ so that $SW_m(f_A , \phi) = (\pi_{Y_A})_*( x^m \smallsmile \eta_A)$, where $Y_A$ is defined in the same way as $Y$, but with $V_A$ replacing $V$. Then from Equation (\ref{equ:taususpend}), we have that:
\[
SW_m(f_A , \phi) = (\pi_{Y_A})_*( x^m \smallsmile \eta \smallsmile e_A).
\]
Therefore, to show that $SW_m(f_A , \phi) = SW_m(f,\phi)$ it suffices to show that:
\begin{equation}\label{equ:pushforwardidentity} 
(\pi_{Y_A})_*( u \smallsmile e_A) = (\pi_Y)_*( u )
\end{equation}
for any $u \in H^{*}_{S^1}( S_{V,U} , S_U ; \mathbb{Z} )$. But since $H^{*}_{S^1}( S_{V,U} , S_U ; \mathbb{Z} ) \cong \left( \frac{ \mathbb{Z}[x] }{\langle e_V \rangle }\right) \delta \tau_{0,U}$, it suffices to check Equation (\ref{equ:pushforwardidentity}) for $u = x^i \delta \tau_{0,U}$, $i = 0,1, \dots , a-1$. From Proposition \ref{prop:pushforwardsegre}, we see that $(\pi_Y)_*( x^i \delta \tau_{0,U} )$ is zero for $0 \le i \le a-2$ and is $1$ for $i = a-1$. On the other hand, since the equivariant Euler class of $A$ has the form
\[
e_A = x^\alpha + x^{\alpha-1}c_1(A) + \cdots + c_\alpha(A),
\]
we see again from Proposition \ref{prop:pushforwardsegre} that
\[
(\pi_{Y_A})_*( x^i e_A \delta \tau_{0,U} ) = (\pi_{Y_A})_*( x^{i+\alpha} + x^{i+\alpha-1}c_1(A) + \cdots + x^i c_\alpha(A) )
\]
is zero for $0 \le i \le a-2$ and is $1$ for $i = a-1$. This completes the proof.
\end{proof}

\section{Steenrod operations on the families Seiberg-Witten invariants}\label{sec:steenrod}

The Steenrod operations are stable cohomology operations. The stability of these operations makes them convenient to use in the study of families Seiberg-Witten invariants, which are extracted from a stable homotopy class. For simplicity, let us consider the Steenrod squares. The Steenrod reduced power operations for odd primes can also be considered, however the algebra turns out to be considerably more difficult and so we will not purse the problem of calculating them here. 

\subsection{Steenrod squares computation}

Our calculations will be carried out using $S^1$-equivariant cohomology with $\mathbb{Z}_2$-coefficients. The Borel model for equivariant cohomology allows us to extend the Steenrod operations to the equivariant setting. We have that $H^*_{S^1}(pt ; \mathbb{Z}_2) \cong \mathbb{Z}_2[x]$, where $x$ has degree $2$. It follows easily that $Sq^i(x^m)=0$ whenever $i$ is odd and $Sq^{2j}(x^m) = \binom{m}{j}x^{m+j}$.

Let $f : (S_{V,U} , B_{V,U} ) \to (S_{V',U'} , B_{V',U'})$ be a finite dimensional monopole map. It turns out that for the purposes of computing Steenrod operations it is most convenient to stabilise $f$ such that the bundles $V,U$ are trivial, say $V = \mathbb{C}^{a}$, $U = \mathbb{R}^{b}$. Recall that $U' = U \oplus H^+$, hence in this case we have $U' = \mathbb{R}^{b} \oplus H^+$. In particular, this gives an equality of Stiefel-Whitney classes:
\[
w_j(U') = w_j(H^+), \text{ for all } j \ge 0.
\]
Next recall that $V - V' = D$, hence $V' = V - D = \mathbb{C}^a - D$. From this we obtain:
\begin{lemma}\label{lem:equivariantchern}
Let $c_{S^1 , j}(V') \in \mathcal{H}^{2j}$ denote the $j$-th equivariant Chern class of $V'$. Then
\begin{equation}\label{equ:equivariantchern}
c_{S^1,j}(V') = \sum_{l = 0}^{j} x^{j-l} \binom{a'-l}{j-l} s_l(D).
\end{equation}
\end{lemma}
\begin{remark}\label{rem:binomnonneg}
Note that since $V' = \mathbb{C}^a - D$, it follows that the Chern classes of $V'$ are the Segre classes of the virtual bundle $D$. So $s_l(D) = 0$ whenever $l > a'$, since $V'$ has rank $a'$. Hence the non-zero terms in Equation (\ref{equ:equivariantchern}) only involve binomial coefficients whose upper index is non-negative.
\end{remark}
\begin{proof}
We use the splitting principal to compute $c_{S^1,j}(V')$. Suppose that (non-equivariantly) the Chern roots of $V'$ are $y_1, \dots , y_{a'}$. Then the equivariant Chern roots of $V'$ are $(x+y_1) , \dots , (x+y_{a'})$ and hence
\begin{equation*}
\begin{aligned}
c_{S^1,j}(V') &= \sum_{i_1 < \cdots < i_{j}} (x+y_{i_1}) \cdots (x+y_{i_j}) \\
& = \sum_{ I \subseteq \{1, \dots , a'\} \atop |I| = j} \sum_{l=0}^j x^{j-l} \sum_{J \subseteq I \atop |J| = l } y_J \quad (\text{where $y_J = y_{j_1} \cdots y_{j_l}$ for $J = \{j_1 , \dots , j_l\}$}) \\
& = \sum_{l=0}^j x^{j-l} \sum_{ J \subseteq \{1, \dots , a' \} \atop |J| = l } y_J \sum_{ I \supseteq J \atop |I| = j } 1 \\
& = \sum_{l=0}^{min(j,a')} x^{j-l} s_l(D) \binom{a'-l}{j-l} \\
& = \sum_{l=0}^{j} x^{j-l} s_l(D) \binom{a'-l}{j-l} \quad (\text{by Remark }\ref{rem:binomnonneg}).
\end{aligned}
\end{equation*}

\end{proof}

Recall that the families Seiberg-Witten invariants $SW_m(f,\phi)$ of $f$ are related to the $\mu_j(f,\phi)$ classes by Equations (\ref{equ:mutosw}) and (\ref{equ:swtomu}). But we have assumed that $V$ is a trivial bundle, hence in this case we have:
\[
SW_m(f,\phi) = \mu_m(f,\phi).
\]

Recall that the $\mu_j(f,\phi)$ invariants are defined by:
\[
f^*(\widetilde{\tau}_{V',U'}^\phi) = \sum_{j=0}^{a-1} \mu_{j}(f,\phi) \delta \tau_{0,U} x^{(a-1)-j}.
\]
Let us define $\eta \in H^*( \mathbb{P}(V) ; \mathbb{Z}_2)$ by
\begin{equation}\label{equ:eta}
\eta = \sum_{j=0}^{a-1} \mu_{j}(f,\phi) x^{(a-1)-j}
\end{equation}
so that $\eta \cdot \delta \tau_{0,U}$ is Poincar\'e dual to the Seiberg-Witten moduli space. It follows that
\[
SW_m(f,\phi) = \pi_*( x^m \eta),
\]
where $\pi$ is the projection $\pi : \mathbb{P}(V) \to B$.

In the calculations that follow we will compute Steenrod squares. All cohomology classes are taken with $\mathbb{Z}_2$ coefficients. In particular any integral cohomology classes which appear in the following calculations should be understood as their mod $2$ reductions.

\begin{lemma}\label{lem:steenrodeta}
Let $\eta \in H^*( \mathbb{P}(V) ; \mathbb{Z}_2)$ be defined as in Equation (\ref{equ:eta}). Then for all $j \ge 0$, we have:
\begin{equation*}
\begin{aligned}
Sq^{2j}(\eta) &= \sum_{l=0}^j c_{S^1,l}(V') w_{2j-2l}(H^+) \eta, \\
Sq^{2j+1}(\eta) &= \sum_{l=0}^j c_{S^1,l}(V') w_{2j-2l+1}(H^+) \eta.
\end{aligned}
\end{equation*}
\end{lemma}
\begin{proof}
Recall that if $\tau_E$ is the Thom class of a vector bundle $E$, then $Sq^j(\tau_E) = w_j(E) \tau_E$. Let $t$ be a formal variable and set:
\[
Sq_t = Sq^0 + t Sq^1 + \cdots \quad w_t(E) = w_0(E) + tw_1(E) + \cdots
\]
so that $Sq_t(\tau_E) = w_t(E) \tau_E$. Applying this to our setting (and using the fact that Steenrod operations commute with the coboundary map in the long exact sequence of a pair), we obtain:
\[
Sq_t( \delta \tau_{0,U}) = w_t(U) \delta \tau_{0,U}, \quad Sq_t( \tau_{V',U'}) = w_{S^1,t}(V' \oplus U') \tau_{V',U'}
\]
where $w_{S^1,t} = w_{S^1,0} + t w_{S^1,1} + \cdots $ and $w_{S^1,j}$ is the $j$-th equivariant Stiefel-Whitney class. 

Recall that $\widetilde{\tau}_{V',U'}^\phi = \phi_*(1)$. To understand the action of $Sq_t$ on $\widetilde{\tau}_{V',U'}^\phi$ we factor $\phi_*$ as follows. Let $B_\phi = \phi(B)$ denote the image of $B$ in $S_{V',U'}$. Let $W_\phi$ denote an equivariant tubular neighbourhood of $B_\phi$, which we can identify with the normal bundle of $B_\phi \subseteq S_{V',U'}$. Let $q : (S_{V',U'} , S_U) \to (W_\phi^+ , \infty)$ be the collapsing map which acts as the identity on $W_\phi$ and collapses the complement of $W_\phi$ to $\infty$. Let $\Phi : H^*_{S^1}(B ; \mathbb{Z}) \to H^{*+2a'+b'}_{S^1}( W^+_\phi , \infty ; \mathbb{Z} )$ be the Thom isomorphism. Then $\phi_* = q^* \circ \Phi$, essentially by definition of the pushforward map of a submanifold. It follows that $\widetilde{\tau}_{V',U'}^\phi = \phi_*(1) = q^*( \Phi(1) ) = q^*( \tau_{W_\phi})$, where $\tau_{W_\phi}$ is the equivariant Thom class of the normal bundle of $B_\phi$. But since $\phi$ is a section, the normal bundle of $B_\phi$ can be identified with the restriction of the vertical tangent bundle of $S_{V',U'} \to B$ to $B_\phi$, which is exactly $V' \oplus U'$. Thus $Sq_t( \tau_{W_\phi} ) = w_{S^1, t}(V' \oplus U') \tau_{W_\phi}$ and using $\widetilde{\tau}_{V',U'}^\phi = q^*( \tau_{W_\phi} )$, we get:
\[
Sq_t( \widetilde{\tau}_{V',U'}^\phi) = w_{S^1 , t}(V' \oplus U')\widetilde{\tau}_{V',U'}^\phi.
\]
Now by the definition of $\eta$, we have $\eta \cdot \delta \tau_{0,U} = f^*( \widetilde{\tau}_{V',U'}^\phi )$ and applying $Sq_t$ we have:
\begin{equation*}
\begin{aligned}
Sq_t(\eta) w_t(U) \cdot \delta \tau_{0,U} &= Sq_t(\eta) Sq_t( \delta \tau_{0,U} ) \\
&= Sq_t( \eta \cdot \delta \tau_{0,U} ) \\
&= Sq_t( f^*( \widetilde{\tau}_{V',U'}^\phi )) \\
&= f^*( Sq_t( \widetilde{\tau}_{V',U'}^\phi ) ) \\
& = w_{S^1 , t}(V' \oplus U') f^*(\widetilde{\tau}_{V',U'}^\phi) \\
& = w_{S^1,t}(V' \oplus U') \eta \cdot \delta \tau_{0,U}.
\end{aligned}
\end{equation*}
Recalling that $U$ is taken to be trivial, we have $w_t(U) = 1$ and hence
\[
Sq_t(\eta) = w_{S^1,t}(V' \oplus U') \eta.
\]
In particular for any $i \ge 0$, we have:
\[
Sq^i(\eta) = w_{S^1,i}(V'\oplus U') \eta.
\]
The proof of the lemma is completed by noting that:
\begin{equation*}
\begin{aligned}
w_{S^1,2j}(V' \oplus U') &= \sum_{l=0}^j c_{S^1,l}(V')w_{2j-2l}(H^+), \\
w_{S^1,2j+1}(V' \oplus U') &= \sum_{l=0}^j c_{S^1,l}(V')w_{2j-2l+1}(H^+).
\end{aligned}
\end{equation*}

\end{proof}
Combining Lemmas \ref{lem:equivariantchern} and \ref{lem:steenrodeta}, we obtain:
\begin{equation*}
\begin{aligned}
Sq^{2j}(\eta) &= \sum_{l=0}^j c_{S^1,l}(V') w_{2j-2l}(H^+) \eta \\
&= \sum_{l=0}^j \sum_{k=0}^l x^{l-k} \binom{a'-k}{l-k} s_k(D) w_{2j-2l}(H^+) \eta \\
&= \sum_{l=0}^j x^l \sum_{k=0}^{j-l} \binom{a'-k}{l} s_k(D) w_{2j-2l-2k}(H^+) \eta.
\end{aligned}
\end{equation*}
Similarly,
\[
Sq^{2j+1}(\eta) = \sum_{l=0}^j x^l \sum_{k=0}^{j-l} \binom{a'-k}{l} s_k(D) w_{2j-2l-2k+1}(H^+) \eta.
\]
Then using Equation (\ref{equ:eta}) to substitute for $\eta$, we find:
\begin{equation*}
\begin{aligned}
Sq^{2j}(\eta) &= \sum_{m=0}^{a-1} \sum_{l=0}^j \sum_{k=0}^{j-l} \binom{a'-k}{l} s_k(D) w_{2j-2k-2l}(H^+) SW_m(f,\phi) x^{a-1-m+l} \\
&= \sum_{l=0}^j \sum_{m=-l}^{a-1-l} \sum_{k=0}^{j-l} \binom{a'-k}{l}s_k(D)w_{2j-2k-2l}(H^+) SW_{m+l}(f,\phi) x^{a-1-m} \\
&= \sum_{l=0}^{j} \sum_{m=0}^{a-1} \sum_{k=0}^{j-l} \binom{a'-k}{l}s_k(D)w_{2j-2k-2l}(H^+) SW_{m+l}(f,\phi) x^{a-1-m},
\end{aligned}
\end{equation*}
where the last line is obtained by using the fact that $V$ is a trivial bundle of rank $a$ to deduce that $x^i = 0$ if $i \ge a$ and also the fact that $SW_i(f,\phi) = 0$ if $i \ge a$, so that the sum over $m$ in the range $-l \le m \le a-1-l$ gives the same result as summing over $m$ in the range $0 \le m \le a-1$. We can repeat the same calculation for the odd Steenrod squares $Sq^{2j+1}$. In summary, we have obtained:
\begin{equation}\label{equ:sqeta1}
\begin{aligned}
Sq^{2j}(\eta) &= \sum_{m=0}^{a-1} \sum_{l=0}^{j} \sum_{k=0}^{j-l} \binom{a'-k}{l}s_k(D)w_{2j-2k-2l}(H^+) SW_{m+l}(f,\phi) x^{a-1-m}, \\
Sq^{2j+1}(\eta) &= \sum_{m=0}^{a-1} \sum_{l=0}^{j} \sum_{k=0}^{j-l} \binom{a'-k}{l}s_k(D)w_{2j-2k-2l+1}(H^+) SW_{m+l}(f,\phi) x^{a-1-m}.
\end{aligned}
\end{equation}

On the other hand applying $Sq^{2j}$ or $Sq^{2j+1}$ directly to Equation (\ref{equ:eta}) and using the Cartan formula, we find:
\begin{equation}\label{equ:sqeta2}
\begin{aligned}
Sq^{2j}(\eta) &= \sum_{m=0}^{a-1} \sum_{l=0}^j Sq^{2j-2l}(SW_{m+l}(f,\phi)) \binom{a-1-m-l}{l} x^{a-1-m}, \\
Sq^{2j+1}(\eta) &= \sum_{m=0}^{a-1} \sum_{l=0}^j Sq^{2j-2l+1}(SW_{m+l}(f,\phi)) \binom{a-1-m-l}{l} x^{a-1-m}.
\end{aligned}
\end{equation}

Equating powers of $x$ in Equations (\ref{equ:sqeta1}) and (\ref{equ:sqeta2}), we immediately find:

\begin{lemma}\label{lem:steenrodrecursive}
The Steenrod squares of the (mod $2$ reductions of the) Seiberg-Witten classes $SW_m(f,\phi)$ satisfy the following recursive equations:
\begin{equation*}
\begin{aligned}
\sum_{l=0}^j Sq^{2j-2l}(SW_{m+l}(f,\phi)) \binom{a-1-m-l}{l} &= \sum_{l=0}^j \sum_{k=0}^{j-l} \binom{a'-k}{l} s_k(D) w_{2j-2k-2l}(H^+) SW_{m+l}(f,\phi), \\
\sum_{l=0}^j Sq^{2j-2l+1}(SW_{m+l}(f,\phi)) \binom{a-1-m-l}{l} &= \sum_{l=0}^j \sum_{k=0}^{j-l} \binom{a'-k}{l} s_k(D) w_{2j-2k-2l+1}(H^+) SW_{m+l}(f,\phi).
\end{aligned}
\end{equation*}

\end{lemma}
\begin{remark}
Before we proceed to solve this recursion, let us make some observations:
\begin{itemize}
\item{
For $j=0$, the above equations reduce to the trivial identites:
\[
Sq^0( SW_m(f,\phi) ) = SW_m(f,\phi), \quad Sq^1( SW_m(f,\phi) ) = w_1(H^+) SW_m(f,\phi),
\]
the latter being a simple consequence of the fact that $Sq^1$ is the Bockstein homomorphism and that the classes $SW_m(f,\phi)$ lift to integral cohomology classes but with values in the local coefficient system $\mathbb{Z}_{w_1(H^+)}$.}
\item{Recall that $SW_i(f,\phi) = 0$ if $i \ge a$ and $s_k(D) = 0$ if $k > a'$. Thus all non-zero terms in the above formula involve only binomial coefficients whose upper index is non-negative.}
\end{itemize}
\end{remark}

\begin{lemma}\label{lem:recur0}
Suppose there are classes $\theta^j_m \in \mathcal{H}^{2m+2j-(2d-b^+-1)}(\mathbb{Z}_2)$ for all $j,m \ge 0$ satisfying the recursive relation
\begin{equation}\label{equ:recur1}
\sum_{l=0}^j \theta^{j-l}_{m+l}\binom{a-1-m-l}{l} = \sum_{l=0}^j \sum_{k=0}^{j-l} \binom{a'-k}{l} s_k(D) w_{2j-2k-2l}(H^+) SW_{m+l}(f,\phi) \quad \text{for all } j,m \ge 0.
\end{equation}
Then $Sq^{2j}(SW_m(f,\phi)) = \theta^j_m$.

Similarly, if there are classes $\pi^j_m \in \mathcal{H}^{2m+2j+1-(2d-b^+-1)}(\mathbb{Z}_2)$ for all $j,m \ge 0$ satisfying
\begin{equation}\label{equ:recur2}
\sum_{l=0}^j \pi^{j-l}_{m+l}\binom{a-1-m-l}{l} = \sum_{l=0}^j \sum_{k=0}^{j-l} \binom{a'-k}{l} s_k(D) w_{2j-2k-2l+1}(H^+) SW_{m+l}(f,\phi), \quad \text{for all } j,m \ge 0.
\end{equation}
Then $Sq^{2j+1}(SW_m(f,\phi)) = \pi^j_m$.
\end{lemma}
\begin{proof}
Follows easily from Lemma \ref{lem:steenrodrecursive} by induction on $j$.
\end{proof}

We now look for solutions of the recursive equations (\ref{equ:recur1}), (\ref{equ:recur2}). Experimenting with small values of $j$, one is naturally lead to look for solutions of the form:
\begin{equation}\label{equ:thetapi}
\begin{aligned}
\theta^j_m = \sum_{l=0}^j \sum_{k=0}^{j-l} f^{m}_{k,l} s_k(D) w_{2j-2l-2k}(H^+) SW_{m+l}(f,\phi), \\
\pi^j_m = \sum_{l=0}^j \sum_{k=0}^{j-l} f^{m}_{k,l} s_k(D) w_{2j-2l-2k+1}(H^+) SW_{m+l}(f,\phi)
\end{aligned}
\end{equation}
for some coefficients $f^{m}_{k,l} \in \mathbb{Z}_2$, where $k,l,m \ge 0$.

\begin{lemma}\label{lem:recurf}
Let $\theta^j_m$, $\pi^j_m$ be defined as in Equation (\ref{equ:thetapi}) for some coefficients $f^{m}_{k,l}$. Suppose that $f^{m}_{k,l}$ satisfies the following recursion:
\begin{equation}\label{equ:recur3}
\sum_{l' = 0}^{l} f^{m+l'}_{k,l-l'} \binom{a'+d-1-m-l'}{l'} = \binom{a'-k}{l} \; ({\rm mod} \; 2)
\end{equation}
for all $k,l,m \ge 0$ and all $a' \ge max( k , -d+1+m+l)$. Then $\theta^j_m, \pi^j_m$ satisfy Equations (\ref{equ:recur1}) and (\ref{equ:recur2}) for all $j,m \ge 0$.
\end{lemma}
\begin{proof}
Follows by direct substitution of (\ref{equ:thetapi}) into (\ref{equ:recur1}) and (\ref{equ:recur2}).
\end{proof}

\begin{lemma}\label{lem:vzero}
For all integers $u$ and all $j \ge 0$, we have:
\[
\sum_{l=0}^j \binom{u+l}{j-l} \binom{2l-1}{l} = \binom{u+1}{j} \; ({\rm mod} \; 2).
\]
\end{lemma}
\begin{proof}
The identity can be easily verified for $j=0,1,2$, so we will assume that $j \ge 3$. If $j \ge 3$, then we need to show that
\[
\sum_{l=2}^j \binom{u+l}{j-l} \binom{2l-1}{l} = \binom{u+1}{j} + \binom{u}{j} + \binom{u+1}{j-1} \; ({\rm mod} \; 2) \; \; \text{for all $j \ge 3$}.
\]
Using Pascal's formula the right hand side simplifies mod $2$ to $\binom{u}{j-2}$. So we need to show that:
\begin{equation}\label{equ:ident1}
\sum_{l=2}^j \binom{u+l}{j-l} \binom{2l-1}{l} = \binom{u}{j-2} \; ({\rm mod} \; 2) \; \; \text{for all $j \ge 3$}.
\end{equation}
Assume $l \ge 2$. Using Pascal's formula, we find
\[
\binom{2l}{l} = \binom{2l-1}{l} + \binom{2l-1}{l-1} = 2 \binom{2l-1}{l}.
\]
Hence $\binom{2l-1}{l}$ is even unless $\binom{2l}{l} = 2 \; ({\rm mod} \; 4)$. However it is known that the number of factors of $2$ in $\binom{2l}{l}$ is equal to the number of $1$s in the binary expansion of $l$. We give a proof of this fact. For a positive integer $n$, let $\nu_2(n)$ denote the number of times $2$ divides $n$. Then $\nu_2(n!) = \sum_{k \ge 1} \lfloor n/2^k \rfloor$ \cite[Chapter 2, Exercise 6]{ir}. Let $n = \sum_{j=0}^r a_j 2^j$, $a_j \in \{0,1\}$ be the binary expansion of $n$. Then $\lfloor n/2^k \rfloor = \sum_{j=k}^{r} a_j 2^{j-k}$ and it follows that
\[
\lfloor 2n/2^k \rfloor - 2 \lfloor n/2^k \rfloor = a_{k-1}
\]
for any $k \ge 1$. Hence
\[
\nu_2 \left( \binom{2n}{n} \right) = \nu_2( (2n)!) - 2\nu_2( n! ) = \sum_{k \ge 1} \left( \lfloor 2n/2^k \rfloor - 2 \lfloor n/2^k \rfloor \right) = \sum_{k \ge 0} a_k
\]
is the number of $1$s in the binary expansion of $n$.

Thus $\binom{2l-1}{l}$ is odd if and only if $l$ is a power of $2$, say $l = 2^k$. So (\ref{equ:ident1}) is equivalent to showing that:
\begin{equation}\label{equ:binom1}
\sum_{k \ge 1 \atop 2^k \le j} \binom{u+2^k}{j-2^k} = \binom{u}{j-2} \; ({\rm mod} \; 2) \; \; \text{for all $j \ge 3$}.
\end{equation}
Recall that if positive integers $a,b$ have binary expansions $a = a_0 + 2a_1 + \cdots + 2^u a_u$, $b = b_0 + 2b_1 + \cdots + 2^u b_u$, then Lucas's theorem \cite{fine} says that:
\begin{equation}\label{equ:binom2}
\binom{a}{b} = \prod_{i=0}^u \binom{a_i}{b_i} \; ({\rm mod} \; 2).
\end{equation}
Now suppose that $j = j_0 + 2j_1 + \cdots + 2^u j_u$, where $j_u = 1$. Then the sum on the left hand side of (\ref{equ:binom1}) runs from $k = 1, \dots , u$. The Vandermonde identity gives:
\[
\binom{u+2^k}{j-2^k} = \sum_{r,s \ge 0 \atop r+s = j-2^k} \binom{u}{r} \binom{2^k}{s}.
\]
But using (\ref{equ:binom2}), we see that $\binom{2^k}{s}$ is even, unless $s = 0$ or $2^k$. Hence
\[
\binom{u+2^k}{j-2^k} = \begin{cases} \binom{u}{j-2^k} + \binom{u}{j-2^{k+1}} \; ({\rm mod} \; 2) & \text{for }1 \le k \le u-1, \\ \binom{u}{j-2^k} \; ({\rm mod} \; 2) & \text{for } k=u. \end{cases}
\]
Substituting this into the left hand side of (\ref{equ:binom1}) gives:
\[
\sum_{k=1}^u \binom{u}{j-2^k} + \sum_{k=1}^{u-1} \binom{u}{j-2^{k+1}}.
\]
The terms in these sums cancel in pairs (mod $2$), except for the $k=1$ term in the first sum. Hence the above equals $\binom{u}{j-2}$ mod $2$. This proves (\ref{equ:binom1}) and the proof of the lemma is complete.
\end{proof}

\begin{lemma}\label{lem:recurf2}
For all $k,l,m \ge 0$, let
\[
f^{m}_{k,l} = \binom{d-1-m+l+k}{l}.
\]
Then $f^m_{k,l}$ satisfies Equation (\ref{equ:recur3}) for all $k,l,m,a' \ge 0$.
\end{lemma}
\begin{proof}
Note that $f^m_{k,l}$ is a binomial coefficient whose upper index is an integer which may be negative. Set $v = a'+(d-1-m)$, $u = k + (d-1-m)$. Thus it suffices to show that:
\begin{equation}\label{equ:recur4}
\sum_{l=0}^j \binom{u+j-2l}{j-l} \binom{v-l}{l} = \binom{v-u}{j} \; ({\rm mod} \; 2)
\end{equation}
for all $u,v,j$ with $j \ge 0$ and $u,v$ any integers. Note that
\begin{equation}\label{equ:binomidentity}
\begin{aligned}
\binom{a}{b} &= \frac{a(a-1) \cdots (a-b+1)}{b!} \\
&= \pm \frac{(b-a-1)(b-a-2) \cdots a}{b!} \\
& = \pm \binom{b-a-1}{b} = \binom{b-a-1}{b} \; ({\rm mod} \; 2).
\end{aligned}
\end{equation}
Hence (\ref{equ:recur4}) can be re-written as:
\[
\sum_{l=0}^j \binom{-u+l-1}{j-l} \binom{v-l}{l} = \binom{v-u}{j} \; ({\rm mod} \; 2),
\]
or replacing $u$ with $-(u+1)$, this becomes
\begin{equation}\label{equ:recur5}
\sum_{l=0}^j \binom{u+l}{j-l} \binom{v-l}{l} = \binom{v+u+1}{j} \; ({\rm mod} \; 2)
\end{equation}
for all $j \ge 0$ and all integers $u,v$. We will prove that (\ref{equ:recur5}) holds by an inductive argument. For any integers $u,v$ and any $j \ge 0$, let us set
\[
\delta(u,v,j) = \sum_{l=0}^j \binom{u+l}{j-l} \binom{v-l}{l} + \binom{v+u+1}{j} \in \mathbb{Z}_2.
\]
So we aim to show that $\delta(u,v,j)$ for all $u,v$ and all $j \ge 0$. For any $u,v$ and any $j \ge 0$, let $P(u,v,j)$ be the proposition that Equation (\ref{equ:recur5}) holds, or equivalently, that $\delta(u,v,j) = 0$. Recall that Pascal's formula $\binom{a}{b} = \binom{a-1}{b} + \binom{a-1}{b-1}$ holds for any non-negative integer $b$ and any number $a$. Hence we find:
\begin{equation*}
\begin{aligned}
&\delta(u,v,j) = \sum_{l=0}^j \binom{u+l}{j-l} \binom{v-l}{l} + \binom{v+u+1}{j} \\
& \; \; \; \; = \sum_{l=0}^j \binom{u-1+l}{j-l}\binom{v-l}{l} + \sum_{l=0}^{j-1} \binom{u-1+l}{j-1-l}\binom{v-l}{l} + \binom{v+u}{j} + \binom{v+u}{j-1} \\
& \; \; \; \; = \delta(u-1,v,j) + \delta(u-1,v,j-1).
\end{aligned}
\end{equation*}
Hence if any two of $P(u,v,j), P(u-1,v,j)$ and $P(u-1,v,j-1)$ hold, then so does the third. Applying Pascal's formula instead to the second binomial factor in $\delta(u,v,j)$ gives:
\begin{equation*}
\begin{aligned}
&\delta(u,v,j) = \sum_{l=0}^j \binom{u+l}{j-l} \binom{v-l}{l} + \binom{v+u+1}{j} \\
& \; \; \; \; = \sum_{l=0}^j \binom{u+l}{j-l} \binom{v-1-l}{l} + \sum_{l=1}^j \binom{u+l}{j-l}\binom{v-1-l}{l-1} + \binom{v+u}{j} + \binom{v+u}{j-1} \\
& \; \; \; \; = \delta(u,v-1,j) + \sum_{l=1}^{j} \binom{(u+1)+(l-1)}{(j-1)-(l-1)}\binom{(v-2)-(l-1)}{l-1} + \binom{(v-2)+(u+1)+1}{j-1} \\
& \; \; \; \; = \delta(u,v-1,j) + \delta(u+1,v-2,j-1).
\end{aligned}
\end{equation*}
Hence if any two of $P(u,v,j), P(u,v-1,j)$ and $P(u+1,v-2,j-1)$ hold, then so does the third. Next, using a change of variables $l \mapsto j - l$, we find:
\begin{equation*}
\begin{aligned}
&\delta(u,v,j) = \sum_{l=0}^j \binom{u+l}{j-l} \binom{v-l}{l} + \binom{v+u+1}{j} \\
& \; \; \; \; = \sum_{l=0}^j \binom{u+j-l}{l} \binom{v-j+l}{j-l} + \binom{v+u+1}{j} \\
& \; \; \; \; = \delta(v-j , u+j , j).
\end{aligned}
\end{equation*}
Hence $P(u,v,j)$ holds if and only if $P(v-j,u+j,j)$ holds.

Setting $j=0$ in (\ref{equ:recur5}), we see that $P(u,v,0)$ holds trivially for all $u,v$. Next, we consider the case $v=-1$. Using (\ref{equ:binomidentity}), we find:
\[
\sum_{l=0}^j \binom{u+l}{j-l} \binom{-1-l}{l} = \sum_{l=0}^j \binom{u+l}{j-l} \binom{2l}{l}.
\]
From Equation (\ref{equ:binom2}) it follows that $\binom{2l}{l}$ is even for all $l > 0$, hence the above sum equals $\binom{u}{j}$ mod $2$. This shows that $P(u,-1,j)$ holds for all $j \ge 0$ and all integers $u$. Similarly, we consider the case $v=0$. Using (\ref{equ:binomidentity}), we find
\[
\sum_{l=0}^j \binom{u+l}{j-l} \binom{-l}{l} = \sum_{l=0}^j \binom{u+l}{j-l} \binom{2l-1}{l}.
\]
So $P(u,0,j)$ is equivalent to
\[
\sum_{l=0}^j \binom{u+l}{j-l} \binom{2l-1}{l} = \binom{u+1}{j} \; ({\rm mod} \; 2).
\]
But this is precisely what was shown in Lemma \ref{lem:vzero}. Hence we have verified $P(u,0,j)$ for all $u$ and all $j \ge 0$.

As an intermediate step towards proving $P(u,v,j)$ for all $u,v$ and all $j \ge 0$, we will now prove it for all $u$, all $v \ge -1$ and all $j \ge 0$. For $v \ge -1$ and $j \ge 0$, let $Q(v,j)$ be the Proposition that $P(u,v,j)$ holds for all values of $u$. We have shown that $Q(v,0)$ holds for all $v \ge -1$, that $Q(-1,j)$ holds for all $j \ge 0$ and that $Q(0,j)$ holds for all $j \ge 0$. Now suppose that $j \ge 1$ and $v \ge 1$. Recall that we have shown that if any two of $P(u,v,j), P(u,v-1,j)$ and $P(u+1,v-2,j-1)$ hold, then so does the third. In particular $P(u,v,j)$ is implied by $P(u,v-1,j)$ and $P(u+1,v-2,j-1)$. It follows that $Q(v,j)$ holds for all $v \ge -1$ and all $j \ge 0$, or equivalently that $P(u,v,j)$ holds for all $u$, for all $v \ge -1$ and all $j \ge 0$.

Recall that $P(u,v,j)$ holds if and only if $P(v-j,u+j,j)$ holds. Thus we have shown that $P(u,v,j)$ holds for all $u,v,j$ such that $j \ge 0$ and $u+j \ge -1$. Let $\lambda = -u-j$. So we have shown $P(u,v,j)$ holds whenever $j \ge 0$ and $\lambda \le 1$. Now we proceed by induction on $\lambda$ (with $\lambda \le 1$ serving as the starting point of the induction). Suppose now that $P(u',v',j')$ holds whenever $j' \ge 0$ and $-u'-j' \le \lambda$ for some $\lambda \ge 1$. Now let $u,v,j$ be such that $j \ge 0$ and $-u-j = \lambda+1$. Recall that we have shown that if any two of $P(u',v',j'), P(u'-1,v',j')$ and $P(u'-1,v',j'-1)$ hold, then so does the third. In particular $P(u+1,v,j+1)$ and $P(u,v,j+1)$ imply $P(u,v,j)$. But note that $P(u+1,v,j+1)$ holds by induction because $-(u+1)-(j+1) = -u-j-2 = \lambda - 1$ and similarly $P(u,v,j+1)$ holds by induction because $-u-(j+1) = -u-j+1 = \lambda$. This shows that $P(u,v,j)$ also holds and so the inductive step is complete. This completes the proof of the result.
\end{proof}

\begin{theorem}\label{thm:steenrod}
The Steenrod squares of the families Seiberg-Witten invariants are given by:
\begin{equation*}
\begin{aligned}
Sq^{2j}( SW_m(f,\phi) ) &= \sum_{l=0}^j \sum_{k=0}^{j-l} \binom{d-1-m+l+k}{l} s_k(D) w_{2j-2l-2k}(H^+) SW_{m+l}(f , \phi), \\
Sq^{2j+1}( SW_m(f,\phi) ) &= \sum_{l=0}^j \sum_{k=0}^{j-l} \binom{d-1-m+l+k}{l} s_k(D) w_{2j-2l-2k+1}(H^+) SW_{m+l}(f , \phi).
\end{aligned}
\end{equation*}
for all $m,j \ge 0$.
\end{theorem}
\begin{proof}
By Lemma \ref{lem:recur0}, it suffices to show that $\theta^j_m, \pi^j_m$ satisfy Equations (\ref{equ:recur1}), (\ref{equ:recur2}), where $\theta^j_m$, $\pi^j_m$ are given by Equation (\ref{equ:thetapi}) with $f^m_{k,l} = \binom{d-1-m+l+k}{l}$. But Lemmas \ref{lem:recurf} and \ref{lem:recurf2} show that equations (\ref{equ:recur1}), (\ref{equ:recur2}) are indeed satisfied.
\end{proof}

For convenience, we give the explicit formulas for the first few even Steenrod squares:
\begin{equation*}
\begin{aligned}
& Sq^2(SW_m(f,\phi)) = \\ 
& \; \; \; \; (d+m)SW_{m+1}(f,\phi) + (s_1(D) + w_2(H^+))SW_m(f,\phi). \\
& Sq^4(SW_m(f,\phi)) = \\
& \binom{d-m+1}{2}SW_{m+2}(f,\phi) + \left( (d+m+1)s_1(D) + (d+m)w_2(H^+) \right)SW_{m+1}(f,\phi) \\
& \quad\quad\quad
+\left(s_2(D) + s_1(D)w_2(H^+) + w_4(H^+)\right)SW_m(f,\phi). \\
& Sq^6(SW_m(f,\phi)) = \\
& \binom{d-m+2}{3}SW_{m+3}(f,\phi) + \left( \binom{d-m+2}{2}s_1(D) + \binom{d-m+1}{2}w_2(H^+) \right) SW_{m+2}(f,\phi) \\
& \quad\quad\quad
+\left( (d+m)s_2(D) + (d+m+1)s_1(D)w_2(H^+) + (d+m)w_4(H^+) \right) SW_{m+1}(f,\phi) \\
& \quad\quad\quad\quad\quad\quad
+\left( s_3(D) + s_2(D)w_2(H^+) + s_1(D)w_4(H^+) + w_6(H^+) \right) SW_m(f,\phi).
\end{aligned}
\end{equation*}
Similar formulas hold for the odd Steenrod squares (to go from $Sq^{2j}(SW_m(f,\phi))$ to $Sq^{2j+1}(SW_m(f,\phi))$ just replace each occurrence of $w_{2i}(H^+)$ with $w_{2i+1}(H^+)$).

\begin{corollary}\label{cor:steenrod1}
Suppose that $2j > 2m -(2d-b^+-1)$. Then:
\begin{equation*}
\begin{aligned}
\sum_{l=0}^j \sum_{k=0}^{j-l} \binom{d-1-m+l+k}{l} s_k(D) w_{2j-2l-2k}(H^+) SW_{m+l}(f , \phi) &= 0 \; ({\rm mod} \; 2), \\
\sum_{l=0}^j \sum_{k=0}^{j-l} \binom{d-1-m+l+k}{l} s_k(D) w_{2j-2l-2k+1}(H^+) SW_{m+l}(f , \phi) &=0 \; ({\rm mod} \; 2).
\end{aligned}
\end{equation*}
\end{corollary}
\begin{proof}
Start with Theorem \ref{thm:steenrod} and recall that $Sq^j(x) = 0$ whenever $j > deg(x)$.
\end{proof}

For instance if $2m < 2d-b^+ + 1$, then:
\[
(d+m)SW_{m+1}(f,\phi) + (s_1(D) + w_2(H^+))SW_m(f,\phi) = 0 \; ({\rm mod} \; 2).
\]
and if $2m < 2d-b^+ + 3$, then:
\begin{equation*}
\begin{aligned}
& \binom{d-m+1}{2}SW_{m+2}(f,\phi) + \left( (d+m+1)s_1(D) + (d+m)w_2(H^+) \right)SW_{m+1}(f,\phi) \\
& \quad \quad \quad\quad\quad\quad\quad\quad\quad
+\left(s_2(D) + s_1(D)w_2(H^+) + w_4(H^+)\right)SW_m(f,\phi) = 0 \; ({\rm mod} \; 2).
\end{aligned}
\end{equation*}

\begin{corollary}
Suppose that $b^+ = 2p+1$ is odd and suppose that for some $j > 0$ we have that $H^{2l}(B , \mathbb{Z}_2) = 0$ for $0 < l < j$. Then
\[
\binom{p+j}{j} SW_{m+j}(f,\phi) + \left( w_{2j}(H^+) + s_j(D) \right) SW_m(f,\phi) = 0 \; ({\rm mod} \; 2).
\]
where $m = d-p-1$.
\end{corollary}
\begin{remark}
Suppose that $f$ is a finite dimensional approximation of the Seiberg-Witten monopole map of a spin$^c$ family of $4$-manifolds with fibres $(X , \mathfrak{s}_X)$ with $b^+(X) = 2p+1>1$. Then for $m = d-p-1$, we see that $SW_m(f,\phi) \in H^0(B,\mathbb{Z}_2) \cong \mathbb{Z}_2$ is the mod $2$ reduction of the ordinary Seiberg-Witten invariant $SW(X,\mathfrak{s}_X)$ of $X$.
\end{remark}
\begin{proof}
Note that $2d-b^+-1 = 2(d-p-1)$ and hence $2j > 0 = 2m - (2d-b^+-1)$. Hence Corollary \ref{cor:steenrod1} holds with $m = d-p-1$. So
\[
\sum_{l=0}^j \sum_{k=0}^{j-l} \binom{d-1-m+l+k}{l} s_k(D) w_{2j-2l-2k}(H^+) SW_{m+l}(f , \phi) = 0 \; ({\rm mod} \; 2).
\]
But $H^{2l}(B , \mathbb{Z}_2) = 0$ for $0 < l < j$, so all terms in the above sum are zero except for $(l,k) = (0,0), (0,j)$ or $(j,0)$. The result now follows easily.
\end{proof}

\begin{theorem}
Suppose that $b^+ = 2p+1$ is odd and that the Stiefel-Whitney classes of $H^+$ are trivial. Write $p$ as $p = 2^a p'$, where $a \ge 0$ and $p'$ is odd. Let $m = d-p-1$ and note that $SW_m(f,\phi) \in H^0(B ; \mathbb{Z}) \cong \mathbb{Z}$. Then for $0 \le b < a$ we have
\[
SW_{m + 2^b}(f,\phi) = s_{2^b}(D)SW_m(f,\phi) \; ({\rm mod} \; 2)
\]
for all $b$ such that $0 \le b < a$. Furthermore, we have
\[
s_{2^a}(D)SW_m(f,\phi) = 0 \; ({\rm mod} \; 2).
\]
\end{theorem}
\begin{proof}
Since the Stiefel-Whitney classes of $H^+$ are trivial, Theorem \ref{thm:steenrod} gives
\begin{equation}\label{equ:steenrodreduced}
\sum_{l=0}^j \binom{p + j}{l} s_{j-l}(D) SW_{m+l}(f,\phi) = 0 \; ({\rm mod} \; 2).
\end{equation}
Next we note that if $p = 2^a p'$, where $p'$ is odd and if $0 \le b < a$, then $\binom{p+2^b}{l}$ is even for $l$ in the range $0 \le l \le 2^b$, except when $l = 0$ or $2^b$ in which case it is odd. Thus Equation (\ref{equ:steenrodreduced}) for $j=2^b$ simplifies to:
\[
s_{2^b}(D) SW_m(f,\phi) + SW_{m+2^b}(f,\phi) = 0 \; ({\rm mod} \; 2).
\]
Similarly $\binom{p+2^a}{l}$ is even for $l$ in the range $0 \le l \le 2^a$, except when $l=0$ in which case it is odd. So Equation (\ref{equ:steenrodreduced}) for $j=2^a$ simplifies to:
\[
s_{2^a}(D)SW_m(f,\phi) = 0 \; ({\rm mod} \; 2).
\]
\end{proof}

\begin{theorem}
Suppose that $H^+$ is orientable and that $2d-b^+-1 = 0$. Suppose also that $SW_0(f,\phi) \in H^0(B ; \mathbb{Z}) \cong \mathbb{Z}$ is odd.
\begin{itemize}
\item{If $SW_j(f,\phi) = 0 \; ({\rm mod} \; 2)$ for all $j > 0$, then $w(H^+) = c(D) \; ({\rm mod} \; 2)$.}
\item{Conversely if $w(H^+) \neq c(D) \; ({\rm mod} \; 2)$, then $SW_j(f,\phi) \neq 0 \; ({\rm mod} \; 2)$ for some $j > 0$.}
\end{itemize}
\end{theorem}
\begin{proof}
Assume that $SW_j(f,\phi)$ vanishes mod $2$ for all $j > 0$ and that $SW_0(f,\phi)$ is odd. We will show that $w(H^+) = c(D)$. The second result follows by taking the contrapositive. From the $Sq^{2j}$ and $Sq^{2j+1}$ cases of Theorem \ref{thm:steenrod} with $m=0$ we obtain:
\begin{equation*}
\begin{aligned}
\sum_{k=0}^j s_k(D)w_{2j-2k}(H^+) &= 0 \; ({\rm mod} \; 2), \\
\sum_{k=0}^j s_k(D)w_{2j-2k+1}(H^+) &= 0 \; ({\rm mod} \; 2).
\end{aligned}
\end{equation*}
this says that the $2j$ and $2j+1$-th Stiefel-Whitney classes of the virtual bundle $H^+-D$ vanish. Thus $w(H^+-D) = 1$. But $w(H^+-D) = w(H^+)c(-D) = w(H^+)s(D) \; ({\rm mod} \; 2)$. Multiplying both sides of $w(H^+)s(D) = 1 \; ({\rm mod} \; 2)$ by $c(D)$, we get $w(H^+) = c(D) \; ({\rm mod} \; 2)$.
\end{proof}

\subsection{Application to $K3$ surfaces}\label{sec:k3}
In this Subsection we give an application of the Steenrod squares computation to $K3$ surfaces. For this purpose we first need to review the notion of topological spin structures.

Let $B$ be a topological manifold. An {\em $n$-microbundle} $\xi$ on $B$ consists of \cite{mil}:
\[
\xi = \{  B \buildrel i \over \longrightarrow E \buildrel p \over \longrightarrow B \}
\]
where:
\begin{itemize}
\item{$E$ is a topological space}
\item{$i : B \to E$ is called the zero-section}
\item{$p : E \to B$ is called the projection}
\item{$p \circ i = id$}
\item{$E$ is locally trivial around the zero section. That is, for each $b\in B$ there exists a neighbourhood $U$ of $b$ in $B$, a neighbourhood $V$ of $i(b)$ in $E$ and a homeomorphism $V \to U \times \mathbb{R}^n$ such that the restriction of $p$ to $V$ coincides with the projection $U \times \mathbb{R}^n \to U$.}
\end{itemize}

Two microbundles $\xi,\xi'$ are considered isomorphic if there exists a homeomorphism between neighbourhoods of the zero sections of $\xi,\xi'$ respecting the zero sections and projection maps.

Every topological manifold $X$ has a tangent microbundle $\tau X$, given by
\[
\tau X = \{ X \buildrel \Delta \over \longrightarrow X \times X \buildrel pr_1 \over \longrightarrow X\}
\]
where $\Delta$ is the diagonal and $pr_1$ is the projection to the first factor. If $X$ is smooth then $\tau X$ is isomorphic to the tangent bundle:
\[
TX = \{ X \buildrel i \over \longrightarrow TX \buildrel p \over \longrightarrow X \}.
\]

It is known as the Kister-Mazur theorem \cite{kis} that a microbundle over a locally-finite complex can be represented by a fibre bundle, and it was generalised by Holm~\cite{holm} to a microbundle over a paracompact base space:

\begin{theorem}[\cite{kis,holm}]
Let $\xi = \{  B \buildrel i \over \longrightarrow E \buildrel p \over \longrightarrow B \}$ be an $n$-microbundle over $B$. Then there exists an open neighbourhood $U \subseteq E$ of the zero section such that $p|_U : U \to B$ is a locally trivial fibre bundle with fibre $\mathbb{R}^n$. The fibre bundle is unique up to isomorphism.
\end{theorem}

Thus, by restriction every $n$-microbundle $\xi$ can be represented by an honest fibre bundle $p : E \to B$ with fibres homeomorphic to $\mathbb{R}^n$ and equipped with a section $i : B \to E$. This fibre bundle is unique up to fibre bundle isomorphism. Let $Top(n)$ denote the group of homeomorphisms of $\mathbb{R}^n$ preserving the origin. Then the fibre bundle $p : E \to B$ has structure group $Top(n)$ and thus determines a principal $Top(n)$ bundle $\mathcal{F}(\xi) \to B$. This principal bundle is well-defined up to principal bundle isomorphism and we call it the {\em frame bundle of $\xi$}.

There is an obvious notion of an oriented microbundle. We can define the {\em oriented frame bundle} $\mathcal{F}^+(\xi)$ of an oriented $n$-microbundle $\xi$, which is a principal $STop(n)$-bundle where $STop(n)$ is the subgroup of $Top(n)$ preserving the orientation of $\mathbb{R}^n$.

It is known that the natural inclusion $SO(n) \to STop(n)$ induces an isomorphism of fundamental groups, and both groups are connected (for $n \neq 4$, this follows from \cite[V, \textsection 5]{ks}. For $n=4$, this follows from \cite[Theorem 8.7A]{fq} together with \cite[V , \textsection 5]{ks}). Hence there is a unique connected double covering $\phi : SpinTop(n) \to STop(n)$. A {\em spin structure} on an oriented $n$-microbundle $\xi$ is a double cover $\tilde{\mathcal{F}}^+(\xi) \to \mathcal{F}^+(\xi)$ whose restriction to each fibre is isomorphic to $\phi$ as covering spaces. If $\xi$ is represented by an oriented vector bundle $V$, this is equivalent to the usual notion of a spin structure on $V$.

Let $X$ be a topological manifold. We define a {\em topological spin structure} on $X$ to be a spin structure on $\tau X$.

Let $E \to B$ be a continuous locally trivial fibre bundle with fibres homeomorphic to $X$. We define the {\em vertical tangent microbundle} of $E$ to be the microbundle
\[
\tau(E/B) = \{ E \buildrel \Delta \over \longrightarrow E \times_B E \buildrel pr_1 \over \longrightarrow E \}.
\]
If $E,B$ are smooth manifolds and $E \to B$ a smooth fibre bundle, then $\tau(E/B)$ is isomorphic to the vertical tangent bundle $T(E/B)$.

Suppose that $E \to B$ is fibrewise oriented, or equivalently, the transition functions are valued in $Homeo^+(X)$, the group of orientation preserving homeomorphisms of $X$. Then associated to $\tau(E/B)$ is its principal $STop(n)$-oriented frame bundle $\mathcal{F}^+(\tau(E/B)) \to E$. We define a {\em families topological spin structure} for $E \to B$ to be a spin structure on $\tau(E/B)$, that is, a double covering $\tilde{\mathcal{F}}^+(\tau(E/B)) \to \mathcal{F}^+(\tau(E/B))$ which restricts to $\phi$ on each fibre.

Let $Y_0$ be an oriented topological $4$-manifold and let $Y_1, \dots, Y_k$ be smooth, oriented, $4$-manifolds. Assume that $Y_0 , \dots , Y_k$ are spin and let $\mathfrak{s}_i$ be a spin structure on $Y_i$ for $0 \le i \le k$. Fix a choice of open neighbourhood $U_0$ of the diagonal in $Y_0 \times Y_0$ such that the restriction of $pr_1 : Y_0 \times Y_0 \to Y_0$ to $U_0$ defines a locally trivial $\mathbb{R}^4$-bundle. We think of $U_0$ as playing the role of the tangent bundle of $Y_0$. Let $\mathcal{F}^+(Y_0)$ denote the principal $STop(4)$-frame bundle associated to $U_0$. The spin structure $\mathfrak{s}_0$ determines a double cover $\tilde{\mathcal{F}}^+(Y_0)$ of $\mathcal{F}^+(Y_0)$ which is a principal $SpinTop(4)$-bundle.

For each $j \in \{1, \dots , k\}$ suppose that $f_j : Y_j \to Y_j$ is an orientation preserving diffeomorphism which acts as the identity in some neighbourhood $B_j$ of a point $p_j \in Y_j$. Suppose also that $f_j$ preserves the spin structure $\mathfrak{s}_j$. Form the connected sum $Y_+ = Y_1 \# \cdots \# Y_k$ in the following manner. Start with $S^4$ and choose $k$ disjoint open balls $C_1, \dots , C_k$ in $S^4$. For each $j \in \{1 , \dots , k\}$, attach $Y_j$ to $S^4$ by removing $C_j$ from $S^4$, removing an open ball $B'_j$ from $Y_j$ such that the closure of $B'_j$ is contained in $B_j$ and then attaching $Y_j \setminus B'_j$ to $S^4 \setminus ( C_1 \cup \cdots C_k )$ along the $j$-th boundary component. Let $\mathfrak{s}_+ = \mathfrak{s}_1 \# \cdots \# \mathfrak{s}_k$ be the resulting spin structure on $Y_+$. Since $f_j$ acts as the indentity on $B_j$ it determines a diffeomorphism $h_j : Y_+ \to Y_+$ which acts as the identity outside of $Y_j \setminus B'_j$. Clearly $h_1, \dots , h_k$ commute and preserve the isomorphism class of $\mathfrak{s}_+$. Let $C_+$ denote an open ball in $S^4$ disjoint from $C_1, \dots , C_k$. We can regard $C_+$ as an open ball in $Y_+$ and we observe that $h_1, \dots , h_k$ act as the identity on $C_+$.

Let $Fr^+(Y_+)$ denote the usual $SO(4)$-frame bundle of $Y_+$ and let $\mathcal{F}^+(Y_+) = Fr^+(Y_+) \times_{SO(4)} STop(4)$ be the associated $STop(4)$-frame bundle. The spin structure $\mathfrak{s}_+$ determines a double cover $\widetilde{Fr}^+(Y_+)$ of $Fr^+(Y_+)$ which is a principal $Spin(4)$-bundle and we set $\tilde{\mathcal{F}}^+(Y_+) = \widetilde{Fr}^+(Y_+) \times_{Spin(4)} SpinTop(4)$, which is the double cover of $\mathcal{F}^+(Y_+)$ corresponding to the spin structure $\mathfrak{s}_+$.

By differentiation, $h_i$ lifts to an automrphism $(h_i)_* : Fr^+(Y_+)$ of the $SO(4)$-frame bundle. Let $\tilde{h}_i$ denote the unique lift of $(h_i)_*$ to $\widetilde{Fr}^+(Y_+)$ such that $\tilde{h}_i$ is the identity over $C_+$. Then by extension of structure group we can also regard $\tilde{h}_i$ as an automorphism of $\tilde{\mathcal{F}}^+(Y_+)$.

We claim that $\tilde{h}_1, \dots , \tilde{h}_k$ commute. To see this, note that since $h_1, \dots , h_j$ commute (and hence so do their lifts to the $SO(4)$-frame bundle), it follows that the commutator $[ \tilde{h}_i , \tilde{h}_j ]$ projects to the identity on $\mathcal{F}^+(Y_+)$, hence  $[ \tilde{h}_i , \tilde{h}_j ] = \pm 1$. But $\tilde{h}_j$ is the identity over $C_+$ and hence $[ \tilde{h}_i , \tilde{h}_j ] = 1$.

Let $Y = Y_0 \# Y_+ = Y_0 \# Y_1 \# \cdots \# Y_k$ be the connected sum, where we attach $Y_0$ to $Y_+$ as follows. Choose an open ball $C_0$ in $Y_0$ and identify $Y_0 \setminus C_0$ and $Y_+ \setminus C_+$ along their boundary $\partial C_0 \cong \partial C_+ \cong S^3$.

The $STop(4)$-frame bundle of $Y$ is obtained by identifying $\mathcal{F}^+(Y_0)|_{\partial C_0}$ and $\mathcal{F}^+(Y_+)|_{\partial C_+}$ by some isomorphism $\psi : \mathcal{F}^+(Y_0)|_{\partial C_0} \to \mathcal{F}^+(Y_+)|_{\partial C_+}$. Note that $\mathcal{F}^+(Y_0)|_{\partial C_0}$ and $\mathcal{F}^+(Y_+)|_{\partial C_+}$ are both trivial bundles and upon choosing trivialisations $\mathcal{F}^+(Y_0)|_{\partial C_0} \cong S^3 \times STop(4)$, $\mathcal{F}^+(Y_+)|_{\partial C_+} \cong S^3 \times STop(4)$, we have that $\psi$ is given by a map $S^3 \to STop(4)$. Since $S^3$ is simply-connected such a map lifts to $S^3 \to SpinSTop(4)$ and thus $\psi$ lifts to an isomorphism $\tilde{\psi} : \tilde{\mathcal{F}}^+(Y_0)|_{\partial C_0} \to \tilde{\mathcal{F}}^+(Y_+)|_{\partial C_+}$ which is unique up to $\tilde{\psi} \mapsto -\tilde{\psi}$. Let $\mathfrak{s} = \mathfrak{s}_0 \# \mathfrak{s}_+$ denote the resulting topological spin structure on $Y$ (the isomorphism class of $\mathfrak{s}$ is easily seen to be independent of the choice of lift of $\psi$).

Since the $h_i$ act as the identity over $C_+$, they extend to commuting diffeomorphisms of $Y$. Likewise, since the $\tilde{h}_i$ act as the identity over $C_+$, they extend to commuting automorphisms of the $SpinTop(4)$-frame bundle of $Y$. Consider now the mapping torus $E = Y \times_{h_1, \dots , h_k} \mathbb{R}^k \to T^k$ associated to the commuting homeomorphisms $h_1, \dots , h_k$. Since the $h_i$ lift to commuting automorphisms of the $SpinTop(4)$-frame bundle of $Y$, it follows that the vertical tangent microbundle $\mathcal{F}^+(\tau(E/T^k))$ admits a spin structure. We have thus proven the following result.

\begin{proposition}\label{prop:spin}
Let $Y_0, Y_1, \dots , Y_k$, $f_1, \dots , f_k$, $\mathfrak{s}_1 , \dots , \mathfrak{s}_k$ be as above. Let $h_1, \dots , h_k$ be the diffeomorphisms on the connected sum $Y = Y_0 \# \cdots \# Y_k$ determined by $f_1, \dots , f_k$. Then the vertical tangent microbundle of the mapping torus of $h_1, \dots , h_k$ admits a spin structure.
\end{proposition}

Let $Homeo^+(X)$ denote the group of orientation preserving homeomorphisms of $X$ with the $\mathcal{C}^0$-topology and let $Diff^+(X)$ denote the group of orientation preserving diffeomorphisms of $X$ with the $\mathcal{C}^\infty$-topology.
Note that, if $X$ is a closed oriented $4$-manifold with non-zero signature, we have $Homeo^+(X)=Homeo(X)$ and $Diff^+(X)=Diff(X)$.

\begin{remark}
\label{rem: continuous and smooth isotopies}
Let $f_{0}, f_{1} \in Diff(X)$.
If there is a continuous map $F_{\bullet} : X \times[0,1] \to X$ satisfying that $F_{0}=f_{0}, F_{1}=f_{1}$, and $F_{t} \in Diff(X)$ for each $t \in [0,1]$,
by approximating $F_{\bullet}$ by a smooth family, one can show that there exists a smooth map $F'_{\bullet} : X \times [0,1] \to X$ satisfying that $F'_{0}=f_{0}, F'_{1}=f_{1}$, and $F'_{t} \in Diff(X)$ for each $t \in [0,1]$.
\end{remark}

\begin{definition}
\label{defi: smoothable as families}
Let $B$ be a smooth manifold. 
\begin{itemize}
\item {For an oriented topological manifold $X$, a {\em continuous family} over $B$ with fibres homeomorphic to $X$ is a fibre bundle $E \to B$ whose fibres are homeomorphic to $X$ and whose transitions functions are valued in $Homeo^+(X)$ (so $E \to B$ is fibrewise oriented).}
\item{For an oriented smooth manifold $X$, a {\em smooth family} over $B$ with fibres diffeomorphic to $X$ is a fibre bundle $E \to B$ whose fibres are diffeomorphic to $X$ and whose transition functions are valued in $Diff^+(X)$.}
\item{For an oriented topological manifold $X$ admitting a smooth structure, a continuous family $E \to B$ over $B$ with fibres homeomorphic to $X$ is called {\em smoothable} if there exists a smooth family $E' \to B$ with fibres diffeomorphic to $X$, equipped with a smooth structure, such that $E$ is isomorphic to the underlying continuous family associated to $E'$.}
\end{itemize}
\end{definition}

\begin{remark}
If $E \to B$ is a smooth family with fibres diffeomorphic to $X$ in the above sense, by replacing $E$ with an isomorphic bundle, we may assume that the total space of $E$ is a smooth manifold and the projection $E \to B$ is a surjective submersion.
This is based on M\"{u}ller--Wockel~\cite{mw} as follows.
A priori, the transition functions of $E$ are continuous maps from double-indexed open sets $\{U_{ij}\}_{ij}$ of $B$ to $Diff^{+}(M)$, regarded as a Fr\'{e}chet manifold.
However, the main theorem of \cite{mw} implies that there exists a bundle $E' \to B$ such that $E'$ is isomorphic to $E$ and that the transition functions $U_{ij} \to Diff(X)$ of $E'$ are smooth for some choice of local trivializations of $E'$.
Recall that the evaluation map $X \times Diff(X) \to X$ is a smooth map.
This follows from the fact that $Diff(X)$ is an open set of the smooth mapping space $Map(X,X)$ with the $\mathcal{C}^\infty$-topology, and the evaluation map $Map(X,X) \times X \to X$ is smooth (see, for example, 42.13. Theorem in p.444 of \cite{km}).
Therefore we can deduce that $E'$ admits a smooth manifold structure and the projection $E' \to B$ is a surjective submersion.
\end{remark}

For any continuous family $E \to B$ with fibres $X$, the bundle $H^+ \to B$ is well-defined up to isomorphism (we can define it to be a maximal positive definite sub-bundle of the bundle with fibres $H^2(X ; \mathbb{R})$).

For any smooth family with fibres diffeomorphic to $X$, equipped with a families spin structure, we may define the virtual index bundle $D \in K^0(B)$ of the Dirac operator.

Let $X$ be the underlying topological $4$-manifold of a $K3$ surface.
Then:
\begin{itemize}
\item{$\pi_0(Homeo^+(X))$ is isomorphic to the isometry group of $H^2(X ; \mathbb{Z})$ by results of Freedman \cite{fre} and Quinn \cite{qui}.}
\item{For the standard smooth structure of $X$, the image of $\pi_0(Diff^+(X))$ in $\pi_0(Homeo^+(X)) \cong Aut( H^2(X ; \mathbb{Z}))$ is precisely the index $2$ subgroup of elements which preserve the orientation of $H^+(X)$ \cite{mat,don}.}
\end{itemize}

\begin{proposition}\label{prop:obstruction}
Let $E \to B$ be a continuous family with fibres $X$ and suppose that $E$ admits a families topological spin structure. If $E$ is smoothable for some smooth structure on $X$, then $w_2(H^+) = 0$.
\end{proposition}
\begin{proof}
Suppose $E$ is smoothable for some smooth structure on $X$. Then $E$ admits a families spin structure. 
Let $\mathfrak{s}$ be the unique spin$^c$ structure on $X$ of zero characteristic. Then $SW(X,\mathfrak{s})$ is odd by \cite{ms}.
By Corollary \ref{cor:w2}, we have
\[
c_1(D) + w_2(H^+) = 0 \in H^2(B ; \mathbb{Z}_2).
\]
To prove the proposition, it suffices to show that $c_1(D)=0$. But $D$ is the families index of a spin structure on a family of $4$-manifolds, so $D$ lies in the image of $KO^{-4}(B) \to K^0(B)$. The group $KO^{-4}(B)$ can be identified with the Grothendieck group of quaternionic vector bundles on $B$. In particular, all such bundles have trivial first Chern class.
\end{proof}

Up to homeomorphism, we may identify $X$ with the connected sum 
\[
X = 2(-E_8) \# 3(S^2 \times S^2).
\]
We label the three copies of $S^2 \times S^2$ as $(S^2 \times S^2)_i$, $i=1,2,3$.

Fix an orientation preserving diffeomorphism $f$ on $S^2 \times S^2$ which acts as $-1$ on $H^2(S^2 \times S^2 ; \mathbb{Z})$ and such that $f$ acts as the identity on some disc $U$. To be specific, we start with the diffeomorphism which acts as a reflection about the equator on each factor of $S^2$ and after performing an isotopy, we can assume the diffeomorphism acts as the identity in a neighbourhood of a fixed point. Now in the connected sum decomposition $X = 2(-E_8) \# 3(S^2 \times S^2)$, we assume that each copy of $(S^2 \times S^2)$ is attached to $2(-E_8)$ via a handle that ends in the fixed disc. Then we set
\[
f_1 = id_{2(-E_8)} \# f \# f \# id, \quad f_2 = id_{2(-E_8)} \# id \# f \# f.
\]
Then $f_1, f_2 : X \to X$ are commuting, orientation preserving homeomorphisms. Let $E_{f_1,f_2} \to T^2$ be the mapping cylinder.

\begin{lemma}\label{lem:topspin}
The family $E \to T^2$ admits a families topological spin structure.
\end{lemma}
\begin{proof}
Define $h_1,h_2,h_3$ by
\[
h_1 = id_{2(-E_8)} \# f \# id \# id, \quad h_2 = id_{2(-E_8)} \# id \# f \# id, \quad h_3 = id_{2(-E_8)} \# id \# id \# f.
\]
Then $f_1 = h_1 h_2$ and $f_2 = h_2 h_3$. It follows that the family $E \to T^2$ is the pullback of the mapping cylinder of $h_1,h_2,h_3$ under a map $T^2 \to T^3$. Hence the result follows from Proposition \ref{prop:spin}.
\end{proof}

\begin{theorem}\label{thm:nonsmoothable}
The family $E \to T^2$ is not smoothable, but its restriction to any $1$-dimensional embedded submanifold $S^1 \subset T^2$ is smoothable.
\end{theorem}
\begin{proof}
Let $E \in H^2( S^2 \times S^2 ; \mathbb{Z})$ be the class Poincar\'e dual to the diagonal $S^2 \to S^2 \times S^2$. Then $f(E) = -E$. For $i=1,2,3$, let $E_i$ denote the corresponding class in $H^2( (S^2 \times S^2)_i ; \mathbb{Z})$. Then $E_1,E_2,E_3$ form a basis for $H^+(X)$ and
\begin{equation*}
\begin{aligned}
f_1(E_1) &= -E_1, && \quad f_1(E_2) &= -E_2, && \quad f_1(E_3) &= E_3, \\ 
f_2(E_1) &=  E_1,  && \quad f_2(E_2) &= -E_2, && \quad f_2(E_3) &= -E_3.
\end{aligned}
\end{equation*}
Clearly $f_1,f_2$ preserve orientation on $H^+$. Let $x,y \in H^1(T^2 ; \mathbb{Z}_2)$ be a basis corresponding to the two $S^1$-factors of $T^2$. Using the splitting of $H^+$ into the real line bundles spanned by $E_1,E_2,E_3$ we find:
\begin{equation*}
\begin{aligned}
w(H^+) &= (1+x)(1+x+y)(1+y) \\
&= (1 + y + xy)(1+y) \\
&= 1 + xy.
\end{aligned}
\end{equation*}
Hence $w_2(H^+) = xy \neq 0$. This shows that $E \to B$ is not smoothable by Proposition~\ref{prop:obstruction}. On the other hand, consider the restriction of $E \to B$ to some embedded circle $j : S^1 \to T^2$. Then $j^*(E) \to S^1$ is the mapping cylinder of the homeomorphism $g = f_1^a f_2^b$, where the underlying homology class of $j(S^1)$ is $(a,b)$. But since $f_1,f_2$ preserve orientation of $H^+(X)$, so does $g$. As remarked earlier, this means that $g$ is continuously isotopic to a diffeomorphism and hence $j^*(E)$ is smoothable with respect to the standard smooth structure on $X=K3$.
\end{proof}

\begin{proposition}
The total space of the family $E$ is smoothable as a manifold.
\end{proposition}

\begin{proof}
The proof is totally the same with Section~6 of \cite{kkn}.
\end{proof}

\section{Wall crossing formula for families Seiberg-Witten invariants}\label{sec:wallcrossing}

In this section we use the monopole map to give a new derivation of the wall crossing formula for families Seiberg-Witten invariants. The families wall crossing formula was originally proven in \cite{liliu} using parametrised Kuranishi models and obstruction bundles. In contrast our approach is purely cohomological and considerably simpler.

\subsection{Wall crossing formula}\label{sec:wcf1}
As usual our starting point is a finite dimensional monopole map $f : (S_{V,U} , B_{V,U} ) \to (S_{V',U'} , B_{V',U'})$. For the wall crossing formula it is most convenient to stabilise $f$ such that the bundles $V',U'$ are trivial, say $V' = \mathbb{C}^{a'}$, $U' = \mathbb{R}^{b'}$. According to Assumption 1, we may assume that $U' = U \oplus H^+$ and that $f|_U : U \to U'$ is the inclusion $U \to U'$. Recall that the set $\mathcal{CH}(f)$ of chambers for $f$ is the set of homotopy class of sections of $U' \setminus U$, or equivalently homotopy classes of sections of $S(H^+)$, the unit sphere bundle in $H^+$ with respect to some choice of metric on $H^+$. Let $[\phi] \in \mathcal{CH}(f)$ be a chamber represented by a section $\phi : B \to S(H^+)$. Since $\phi$ avoids $U$, we obtain a pushforward map:
\[
\phi_* : H^j_{S^1}(B ; \mathbb{Z}) \to H^{j+2a'+b'}_{S^1}( S_{V',U'} , S_U ; \mathbb{Z} ).
\]
Recall that for $m \ge 0$ the $m$-th Seiberg-Witten invariant of $(f , \phi)$ is a cohomology class $SW_m(f , \phi) \in H^{2m-(2d-b^+-1)}(B , \mathbb{Z}_{w_1})$ where $w_1 = w_1(H^+)$. By Theorem \ref{thm:pdclass}, $SW_m(f,\phi)$ is given by:
\[
SW_m(f,\phi) = (\pi_Y)_*( x^m \smallsmile f^*\phi_*(1) ).
\]
Let $[\psi] \in \mathcal{CH}(f)$ be a second chamber for $f$ represented by a section $\psi : B \to S(H^+)$. Then similarly
\[
SW_m(f,\psi) = (\pi_Y)_*( x^m \smallsmile f^*\psi_*(1) ).
\]
It follows that the difference $SW_m(f,\phi) - SW_m(f,\psi)$ can be computed provided we know the difference $\psi_*(1) - \psi_*(1) \in H^{2a'+b'}_{S^1}( S_{V',U'} , S_U ; \mathbb{Z})$. As before, we use the following notation: $\widetilde{\tau}_{V',U'}^\phi = \phi_*(1)$, $\widetilde{\tau}_{V',U'}^\psi = \psi_*(1)$. Since $\phi$ and $\psi$ are homotopic within the total space of $U$, we have that the images of $\widetilde{\tau}_{V',U'}^\phi$ and $\widetilde{\tau}_{V',U'}^\psi$ under the natural map $H^{2a'+b'}_{S^1}( S_{V',U'} , S_U ; \mathbb{Z} ) \to H^{2a'+b'}_{S^1}( S_{V',U'} , B_{V',U'} ; \mathbb{Z} )$ coincide. Then from Proposition \ref{prop:equivthom2} we have that
\[
\widetilde{\tau}_{V',U'}^\phi - \widetilde{\tau}_{V',U'}^\psi = \widetilde{\theta}(\phi,\psi) \smallsmile \delta \tau_{0,U}
\]
for some $\widetilde{\theta}(\phi,\psi) \in H^{2a'+b^+-1}_{S^1}(B,\mathbb{Z}_{w_1})$.\\

Next we observe that $\phi : B \to S_{V',U'}$ factors as:
\[
B \buildrel \phi^0 \over \longrightarrow S(H^+) \buildrel \iota \over \longrightarrow S_{V',U'}.
\]
A similar factorisation holds for $\psi$ and therefore we have:
\begin{equation}\label{equ:taudifference}
\widetilde{\tau}_{V',U'}^\phi - \widetilde{\tau}_{V',U'}^\psi = \phi_*(1) - \psi_*(1) = \iota_*( \phi^0_*(1) - \psi^0_*(1)).
\end{equation}

\begin{lemma}\label{lem:iota1}
We have $\iota_*(1) = x^{a'} \delta \tau_{0,U}$.
\end{lemma}
\begin{proof}
We factor  $S(H^+) \buildrel \iota \over \longrightarrow S_{V',U'}$ as follows:
\[
S(H^+) \buildrel \iota_1 \over \longrightarrow S_{0,U'} \buildrel \iota_2 \over \longrightarrow S_{V',U'}.
\]
Then it suffices to show that $(\iota_1)_*(1) = \delta \tau_{0,U}$ and $(\iota_2)_*(\delta \tau_{0,U}) = x^{a'} \delta \tau_{0,U}$. First consider $\iota_1$. Since $S^1$ acts trivially on $S(H^+)$ and $S_{0,U'}$, it suffices to compute $(\iota_1)_*(1)$ in non-equivariant cohomlogy. We view $(\iota_1)_*$ as a homomorphism
\[
(\iota_1)_* : H^*( S(H^+) ; \mathbb{Z}) \to H^{*+b+1}( S_{0,U'} , S_U ; \mathbb{Z}_{w_1}).
\]
But any element of $H^*( S_{0,U'} , S_U ; \mathbb{Z}_{w_1})$ has the form $\alpha \smallsmile \delta \tau_{0,U} + \beta \smallsmile \widetilde{\tau}^\phi_{0,U'}$ for some $\alpha \in H^*(B ; \mathbb{Z})$ and some $\beta \in H^*(B ; \mathbb{Z}_{w_1})$. In particular, we have $(\iota_1)_*(1) = \alpha \smallsmile \delta \tau_{0,U} + \beta \smallsmile \widetilde{\tau}^\phi_{0,U'}$ for some $\alpha \in H^0(B ; \mathbb{Z}) = \mathbb{Z}$ and some $\beta \in H^{1-b^+}(B ; \mathbb{Z}_{w_1})$. We claim that the image of $(\iota_1)_*(1)$ under the natural map $H^*( S_{0,U'} , S_U ; \mathbb{Z}_{w_1}) \to H^*( S_{0,U'} , B_{0,U'} ; \mathbb{Z}_{w_1})$ is zero. To see this note that $(\iota_1)_*(1)$ is Poincar\'e dual to $S(H^+) \subset S_{0,U'}$. But $S(H^+)$ is the boundary of the unit disc bundle $D(H^+)$ and $D(H^+)$ is disjoint from $B_{0,U'}$. This proves the claim. Thus $(\iota_1)_*(1) = \alpha \delta \tau_{0,U}$ for some $\alpha \in \mathbb{Z}$. To compute $\alpha$, it suffices to restrict the the fibres of $S(H^+)$ and $S_{0,U'}$ over a point in $B$. Upon restriction, $\iota_1$ is the just the inclusion map of spheres:
\[
S( \mathbb{R}^{b^+}) \to S(\mathbb{R} \oplus \mathbb{R}^{b^+} \oplus \mathbb{R}^b).
\]
Then it is straightforward to verify in this situation that $\alpha = 1$.\\

Next, we verify that $(\iota_2)_*(\delta \tau_{0,U}) = x^{a'} \delta \tau_{0,U}$. But $(\iota_2)_*(\delta \tau_{0,U}) = \iota_*(1)$ is Poincar\'e dual to $S(H^+)$, so arguing as we did for $\iota_1$, this implies that the image of $(\iota_2)_*(\delta \tau_{0,U})$ under the map $H^*_{S^1}(S_{V',U'} , S_U ; \mathbb{Z}_{w_1}) \to H^*_{S^1}( S_{V',U'} , B_{V',U'} ; \mathbb{Z}_{w_1})$ is zero. Thus $(\iota_2)_*(\delta \tau_{0,U}) = \gamma \smallsmile \delta \tau_{0,U}$ for some $\gamma \in H^{2a'}_{S^1}(B ; \mathbb{Z})$. But $(\iota_2)^* (\iota_2)_* (\delta \tau_{0,U}) = e_{V'} \smallsmile \delta \tau_{0,U}$, where $e_{V'}$ denotes the equivariant Euler class of $V'$. It follows that $\gamma = e_{V'}$. Then since we are assuming $V'$ is a trivial bundle, we have $e_{V'} = x^{a'}$ and the result follows.
\end{proof}

Let $\pi : S(H^+) \to B$ denote the projection to $B$. From $\pi \circ \phi^0 = id$, we see that $\pi_*( \phi^0_*(1) ) = 1$ and similarly $\pi_*( \psi^0_*(1) ) = 1$. It follows that $\pi_*( \phi^0_*(1) - \psi^0_*(1)) = 0$. The Gysin sequence for $S(H^+) \to B$ then implies that 
\begin{equation}\label{equ:theta0}
\phi^0_*(1) - \psi^0_*(1) = \pi^*(\theta(\phi,\psi))
\end{equation}
for some $\theta(\phi,\psi) \in H^{b^+-1}(B ; \mathbb{Z}_{w_1})$. Moreover since $S(H^+)$ admits sections, the Gysin sequence splits and the pullback map $\pi^* : H^*(B ; \mathbb{Z}_{w_1}) \to H^*( S(H^+) ; \mathbb{Z}_{w_1})$ is injective, so that $\theta(\phi,\psi)$ is uniquely characterised by Equation (\ref{equ:theta0}).

\begin{theorem}[Wall crossing formula for families Seiberg-Witten invariants]\label{thm:wcf}
Given $\phi, \psi \in \mathcal{CH}(f)$, let $\theta(\phi,\psi) \in H^{b^+-1}(B ; \mathbb{Z}_{w_1})$ be defined as in Equation (\ref{equ:theta0}). Then:
\[
SW_m(f,\phi)-SW_m(f,\psi) = \begin{cases} 0 & \text{if} \; \; m < d-1, \\ \theta(\phi,\psi)s_{m-(d-1)}(D) & \text{if} \; \; m \ge d-1. \end{cases}
\]
\end{theorem}
\begin{proof}
From Equations (\ref{equ:taudifference}) and (\ref{equ:theta0}) we have:
\begin{equation*}
\begin{aligned}
\widetilde{\tau}_{V',U'}^\phi - \widetilde{\tau}_{V',U'}^\psi &= \iota_*( \phi^0_*(1) - \psi^0_*(1)) \\
&= \iota_*( \pi^*(\theta(\phi,\psi))) \\
&= \theta(\phi,\psi) \cdot \iota_*(1) \\
&= \theta(\phi,\psi) \cdot x^{a'} \delta \tau_{0,U},
\end{aligned}
\end{equation*}
where the last equality follows from Lemma \ref{lem:iota1}. Combined with Theorem \ref{thm:pdclass}, we now have:
\begin{equation*}
\begin{aligned}
SW_m(f,\phi) - SW_m(f,\psi) &= (\pi_Y)_*( x^m \smallsmile f^*( \widetilde{\tau}_{V',U'}^\phi - \widetilde{\tau}_{V',U'}^\psi ) ) \\
&= (\pi_Y)_*( x^{m+a'} \smallsmile f^*( \theta(\phi,\psi)) \cdot \delta \tau_{0,U} ) ) \\
&= \theta(\phi,\psi) \cdot (\pi_Y)_*( x^{m+a'} \smallsmile \delta \tau_{0,U}).
\end{aligned}
\end{equation*}
The result now follows by applying Proposition \ref{prop:pushforwardsegre}.
\end{proof}

\subsection{Relative obstruction class}\label{sec:relativeobstruction}

In the previous subsection we gave a wall crossing formula for families Seiberg-Witten invariants in terms of a class $\theta(\phi , \psi)$. In this subsection we examine the class $\theta(\phi,\psi)$ more carefully and identify with the primary obstruction class for the existence of a homotopy between $\phi$ and $\psi$.

Let $[\phi], [\psi] \in \mathcal{CH}(f)$ be two chambers, which we represent by sections $\phi , \psi : B \to S(H^+)$ of the unit sphere bundle $S(H^+)$ associated to $H^+$. By definition, $\phi$ and $\psi$ represent the same chamber if and only if there is a homotopy between them. While the general problem of computing homotopy classes of sections of sphere bundles is complicated, we can use obstruction theory to gain some insights. Since the fibres of $S(H^+) \to B$ are spheres of dimension $b^+-1$, we see there is no obstruction to finding a homotopy between $\phi$ and $\psi$ on the $(b^+-2)$-skeleton of $B$. Extending such a homotopy to the $(b^+-1)$-skeleton is in general obstructed. The obstruction to extending to the $(b^+-1)$-skeleton is a cohomology class called the {\em primary difference} in \cite[\textsection 36]{ste}:
\[
Obs(\phi,\psi) \in H^{b^+-1}(B ; \mathbb{Z}_{w_1}).
\]
One may define $Obs(\phi,\psi)$ as follows. Consider the pullback of $H^+$ to $B \times I$, where $I = [0,1]$ is the unit interval. The pair $(\phi,\psi)$ define a non-vanishing section of $H^+$ over $B \times {0,1}$. In such a situation we have a relative Euler class \cite{ker}
\[
e_{rel}( H^+, \phi,\psi) \in H^{b^+}( B \times I , B \times \{0,1\} ; \mathbb{Z}_{w_1}).
\]
Let $\iota : B \to B \times I$ be given by $\iota(b) = (b , 1/2)$. Then the pushforward
\[
\iota_* : H^{*-1}( B ; \mathbb{Z}_{w_1} ) \to H^*( B \times I , B \times \{0,1\} ; \mathbb{Z}_{w_1})
\]
is easily seen to be an isomorphism. The class can $Obs(\phi,\psi)$ may be defined by the relation (c.f. \cite[Lemma 3.8]{liliu}):
\[
\iota_*( Obs(\phi , \psi) ) = e_{rel}(H^+ , \phi , \psi).
\]
Note that $\iota_*$ coincides with the coboundary map $\delta : H^{*-1}( B ; \mathbb{Z}_{w_1} ) \to H^*( B \times I , B \times \{0,1\} ; \mathbb{Z}_{w_1})$ in the long exact sequence of the pair $(B \times I , B \times \{0,1\})$.

A more geometric definition of $e_{rel}(H^+ , \phi , \psi)$ is as follows: choose a section $\tilde{s}$ of the pullback of $H^+$ to $B \times I$ such that $\tilde{s}|_{B \times \{0\}} = \phi$, $\tilde{s}|_{B \times \{1\}} = \psi$ and such that $\tilde{s}$ is transverse to the zero section. Then the zero locus $\tilde{s}^{-1}(0)$ is Poincar\'e dual to $e_{rel}(H^+ , \phi , \psi)$.

\begin{lemma}
Let $-\psi$ denote the composition of $\psi$ with the antipodal map. Suppose that $\phi$ and $-\psi$ meet transversally. Then $Obs(\phi,\psi)$ is Poincar\'e dual to the intersection locus $S = \{ b \in B \; | \; \phi(b) = -\psi(b) \}$.
\end{lemma}
\begin{proof}
Let us define a section $\tilde{s}$ of the pullback of $H^+$ to $B \times I$ by:
\[
\tilde{s}(b,t) = (1-t)\phi + t\psi.
\]
Then $\tilde{s}|_{B \times \{0\}} = \phi$, $\tilde{s}|_{B \times \{1\}} = \psi$. Moreover, since $\phi$ and $\psi$ are both sections of the unit sphere bundle $S(H^+)$, we see that $\tilde{s}(b,t) = 0$ only when $t = 1/2$. Moreover $\tilde{s}(b,1/2) = \frac{1}{2}\left( \phi + \psi \right)$, so $\tilde{s}(b,1/2) = 0$ if and only if $\phi(b) = -\psi(b)$. So the zero locus of $\tilde{s}$ is precisely $S \times \{1/2\} = \iota(S)$. If $\phi$ and $-\psi$ intersect transversally then it is easy to see that $\tilde{s}$ meets the zero section transversally. Thus $\iota(S)$ is Poincar\'e dual to $e_{rel}(H^+ , \phi , \psi)$. Then since $\iota_*( Obs(\phi,\psi)) = e_{rel}(H^+,\phi,\psi)$ we have that $S$ is Poincar\'e dual to $Obs(\phi,\psi)$.
\end{proof}

The above lemma translates into the following cohomological formula for $Obs(\phi,\psi)$:
\begin{equation}\label{equ:obs}
Obs(\phi , \psi) = \pi_*( \phi_*(1) \smallsmile (-\psi)_*(1) ).
\end{equation}
Indeed $\phi_*(1) \smallsmile (-\psi)_*(1)$ is Poincar\'e dual to the intersection locus of $\phi(B)$ and $-\psi(B)$ in $S(H^+)$ and thus $\pi_*( \phi_*(1) \smallsmile (-\psi)_*(1) )$ is Poincar\'e dual to $S = \{ b \in B \; | \; \phi(b) = -\psi(b) \}$.

\begin{proposition}
We have an equality $Obs(\phi,\psi) = \theta(\phi,\psi)$.
\end{proposition}
\begin{proof}
Recall that $\theta(\phi,\psi)$ is characterised by the identity $\pi^*(\theta(\phi,\psi)) = \phi_*(1) - \psi_*(1)$ while on the other hand $Obs(\phi,\psi)$ is given by Equation (\ref{equ:obs}). To relate these two expressions we will need to carry out some cohomological computations on $S(H^+)$. Consider the Gysin sequence for $S(H^+) \to B$:
\[
\cdots \to H^*(B ; \mathbb{Z}_{w_1}) \buildrel \pi^* \over \longrightarrow H^*(S(H^+) ; \mathbb{Z}_{w_1}) \buildrel \pi_* \over \longrightarrow H^{*-(b^+-1)}(B ; \mathbb{Z}) \to \cdots
\]
This sequence may be split using a section of $S(H^+)$. For instance if we take the section $\phi$ then $\phi_* : H^{*-(b^+-1)}(B ; \mathbb{Z}) \to H^*(S(H^+) ; \mathbb{Z}_{w_1})$ is a right inverse to $\pi_*$ and $\phi^* : H^*(S(H^+) ; \mathbb{Z}_{w_1}) \to H^*(B ; \mathbb{Z}_{w_1})$ is a left inverse of $\pi^*$. Thus $\phi$ induces an isomorphism
\[
H^*(S(H^+) ; \mathbb{Z}_{w_1}) \cong H^*(B ; \mathbb{Z}_{w_1}) \oplus H^{*-(b^+-1)}(B ; \mathbb{Z}),
\]
which can moreover be thought of as an isomorphism of $H^*(B ; \mathbb{Z})$-modules. It follows in particular that to each section $\phi$ there is a uniquely determined class $\tau_\phi \in H^{b^+-1}(S(H^+) ; \mathbb{Z})$ satisfying $\pi_*(\tau_\phi) = 1$ and $\phi^*(\tau_\phi)=0$. By the above splitting we also have that
\[
\phi_*(1) = \pi^*(\alpha_\phi) \tau_\phi + \pi^*(\beta_\phi)
\]
for uniquely determined classes $\alpha_\phi \in H^0(B ; \mathbb{Z}) \cong \mathbb{Z}$ and $\beta_\phi \in H^{b^+-1}(B ; \mathbb{Z}_{w_1})$. Then since $\pi \circ \phi = id$, we have
\[
1 = \pi_* \phi_*(1) = \alpha_\phi \pi_*(\tau_\phi) + \pi_*( \pi^*(\beta_\phi)) = \alpha_\phi.
\]
So $\alpha_\phi = 1$. Next, let $e_\phi \in H^{b^+-1}(B ; \mathbb{Z}_{w_1})$ denote the Euler class of $\langle \phi \rangle^\perp$, the orthogonal complement of $\langle \phi \rangle$ in $H^+$. Then since $\langle \phi \rangle^\perp$ can be identified with the normal bundle of $\phi(B)$, we have
\[
e_\phi = \phi^*(\phi_*(1)) = \phi^*( \tau_\phi) + \phi^* \pi^*(\beta_\phi) = \beta_\phi.
\]
So $\beta_\phi = e_\phi$ and we have shown that
\begin{equation}\label{equ:phitau}
\phi_*(1) = \tau_\phi + \pi^*(e_\phi).
\end{equation}
Now if $\psi$ is another section we have that $\pi_*( \tau_\phi - \tau_\psi) = 0$ and hence
\begin{equation}\label{equ:tauphipsi}
\tau_\phi - \tau_\psi = \pi^*(\lambda(\phi,\psi))
\end{equation}
for a uniquely determined class $\lambda(\phi,\psi) \in H^{b^+-1}(B ; \mathbb{Z}_{w_1})$. Combining equations (\ref{equ:phitau}) and (\ref{equ:tauphipsi}) with the definition of $\theta(\phi,\psi)$ we see that:
\begin{equation}\label{equ:thetalambda}
\theta(\phi,\psi) = \lambda(\phi,\psi) + e_\phi - e_\psi.
\end{equation}
Next, consider $\tau^2_\phi$. By the Gysin sequence we have
\[
\tau^2_\phi = \pi^*(a_\phi) \smallsmile \tau_\phi + \pi^*(b_\phi)
\]
for uniquely determined classes $a_\phi \in H^{b^+-1}(B ; \mathbb{Z}_{w_1})$ and $b_\phi \in H^{2(b^+-1)}(B ; \mathbb{Z})$. From $\phi^*(\tau_\phi) = 0$ we see that $b_\phi = 0$. To compute $a_\phi$ we first note that $\tau^2_\phi$ is Poincar\'e dual to the self-intersection of $\phi(B)$. But the normal bundle to $\phi(B)$ is isomorphic to $\langle \phi \rangle^\perp$, which implies that
\[
\pi_*( \phi_*(1) \smallsmile \phi_*(1) ) = e_\phi.
\]
On the other hand, using (\ref{equ:phitau}), we find
\begin{equation*}
\begin{aligned}
e_\phi &= \pi_*( \phi_*(1) \smallsmile \phi_*(1) ) \\
&= \pi_* \left( (\tau_\phi + \pi^*(e_\phi) ) \smallsmile (\tau_\phi + \pi^*(e_\phi) ) \right) \\
&= \pi_* \left( \tau_\phi^2 + \pi^*(e_\phi) \smallsmile \tau_\phi + (-1)^{b^+-1} \pi^*(e_\phi) \smallsmile \tau_\phi + \pi^*( e_\phi^2) \right) \\
&= a_\phi + e_\phi + (-1)^{b^+-1}e_\phi.
\end{aligned}
\end{equation*}
Therefore $a_\phi = (-1)^{b^+} e_\phi$ and we have shown that:
\[
\tau_\phi^2 = (-1)^{b^+} \pi^*(e_\phi) \smallsmile \tau_\phi.
\]
Let $(-1) : S(H^+) \to S(H^+)$ denote the antipodal map. Note first of all that $(-1)$ preserves or reverses orientation according to the parity of $b^+$. Thus
\[
(-1)_* \circ \pi^*(x) = (-1)_*( \pi^*(x) \smallsmile 1) = \pi^*(x) \smallsmile (-1)_*(1) = (-1)^{b^+} \pi^*(x)
\]
for any $x \in H^*(B ; \mathbb{Z}_{w_1})$. Using the splitting of the Gysin sequence we have that
\[
(-1)_* \tau_\phi = c_\phi \tau_\phi + \pi^*(\gamma_\phi)
\]
for some $c_\phi \in \mathbb{Z}$ and some $\gamma_\phi \in H^{b^+-1}(B ; \mathbb{Z}_{w_1})$. Since $\pi \circ (-1) = \pi$, we have that $\pi_* \circ (-1)_* \tau_\phi = \pi_* \tau_\phi = 1$ and hence $c_\phi = 1$. To compute $\gamma_\phi$, we note that $\phi(B)$ and $-\phi(B)$ are disjoint so $\phi_*(1) \smallsmile -\phi_*(1) = 0$. Therefore,
\begin{equation*}
\begin{aligned}
0 &= \pi_*( \phi_*(1) \smallsmile -\phi_*(1) ) \\
&= \pi_*\left( (\tau_\phi + \pi^*(e_\phi)) \smallsmile (-1)_* (\tau_\phi + \pi^*(e_\phi)) \right) \\
&= \pi_*\left( (\tau_\phi + \pi^*(e_\phi)) \smallsmile ( \tau_\phi + \pi^*(\gamma_\phi + (-1)^{b^+}e_\phi) ) \right) \\
&= \pi_*\left( \tau_\phi^2 + \pi^*(e_\phi) \smallsmile \tau_\phi + (-1)^{b^+-1} \pi^*(\gamma_\phi + (-1)^{b^+}e_\phi) \smallsmile \tau_\phi \right) \\
&= (-1)^{b^+} e_\phi + e_\phi + (-1)^{b^+-1}(\gamma_\phi + (-1)^{b^+} e_\phi) \\
&= (-1)^{b^+}e_\phi + (-1)^{b^+-1}\gamma_\phi.
\end{aligned}
\end{equation*}
Hence $\gamma_\phi = e_\phi$ and we have shown that
\begin{equation}\label{equ:antipode}
(-1)_* \tau_\phi = \tau_\phi + \pi^*(e_\phi).
\end{equation}
Putting all of these calculations together, we have:
\begin{equation*}
\begin{aligned}
Obs(\phi,\psi) &= \pi_*( \phi_*(1) \smallsmile -\psi_*(1)) \\
&= \pi_* \left( (\tau_\phi + \pi^*(e_\phi)) \smallsmile (-1)_* ( \tau_\psi + \pi^*(e_\psi)) \right) \\
&= \pi_* \left( (\tau_\phi + \pi^*(e_\phi) ) \smallsmile ( \tau_\psi + \pi^*( e_\psi + (-1)^{b^+}e_\psi)) \right) \\
&= \pi_* \left( \tau_\phi \tau_\psi + \pi^*(e_\phi) \smallsmile \tau_\psi + (-1)^{b^+-1} \pi^*(e_\psi)\smallsmile \tau_\phi -\pi^*(e_\psi) \smallsmile \tau_\phi \right)\\
&= \pi_* \left( (\tau_\psi + \pi^*(\lambda(\phi,\psi) ) \smallsmile \tau_\psi \right) + e_\phi + (-1)^{b^+-1} e_\psi - e_\psi \\
&= \pi_*( \tau_\psi^2 ) + \lambda(\phi,\psi) + e_\phi + (-1)^{b^+-1}e_\psi - e_\psi \\
&= (-1)^{b^+}e_\psi + \lambda(\phi,\psi) + e_\phi + (-1)^{b^+-1}e_\psi - e_\psi \\
&= \lambda(\phi,\psi) + e_\phi - e_\psi.
\end{aligned}
\end{equation*}
Comparing with Equation (\ref{equ:thetalambda}), we see that $Obs(\phi,\psi) = \theta(\phi,\psi)$.
\end{proof}

Using this result, the wall crossing formula for the families Seiberg-Witten invariants (Theorem \ref{thm:wcf}) can be re-written in terms of the obstruction class $Obs(\phi,\psi)$:

\begin{corollary}
Given $\phi, \psi \in \mathcal{CH}(f)$, we have:
\[
SW_m(f,\phi)-SW_m(f,\psi) = \begin{cases} 0 & \text{if} \; \; m < d-1, \\ Obs(\phi,\psi)s_{m-(d-1)}(D) & \text{if} \; \; m \ge d-1. \end{cases}
\]
\end{corollary}

The following proposition gives some further properties of the obstruction class $Obs(\phi,\psi)$ depending on the parity of $b^+$:
\begin{proposition}
We have the following:
\begin{itemize}
\item{Suppose $b^+$ is even. Then $e_\phi = e_\psi$ for any two sections $\phi,\psi$ and hence $Obs(\phi,\psi) = \lambda(\phi,\psi)$.}
\item{Suppose $b^+$ is odd. Then $2\lambda(\phi,\psi) = -e_\phi + e_\psi$ and $2 \cdot Obs(\phi , \psi) = e_\phi - e_\psi$. Thus if $H^{b^+-1}(B ; \mathbb{Z}_{w_1})$ has no $2$-torsion we may write $Obs(\phi,\psi) = \frac{1}{2}( e_\phi - e_\psi)$.}
\end{itemize}
\end{proposition}
\begin{proof}
Recall that $\tau_\phi = \tau_\psi + \pi^*(\lambda(\phi,\psi))$. Applying $(-1)_*$ and using (\ref{equ:antipode}) we find
\begin{equation*}
\begin{aligned}
\tau_\phi + \pi^*(e_\phi) &= \tau_\psi + \pi^*(e_\psi) + (-1)^{b^+} \pi^*(\lambda(\phi,\psi) \\
&= \tau_\phi + \pi^*(e_\psi + (-1)^{b^+}\lambda(\phi,\psi) -\lambda(\phi,\psi) )
\end{aligned}
\end{equation*}
and hence
\[
-\lambda(\phi,\psi) + (-1)^{b^+}\lambda(\phi,\psi) = e_\phi - e_\psi.
\]
If $b^+$is even, this gives $e_\phi = e_\psi$ and thus $Obs(\phi,\psi) = \lambda(\phi,\psi) + e_\phi - e_\psi = \lambda(\phi,\psi)$. If $b^+$ is odd, then we get $2\lambda(\phi,\psi) = -e_\phi + e_\psi$ and hence
\[
2 \cdot Obs(\phi,\psi) = 2\lambda(\phi,\psi) + 2(e_\phi - e_\psi) = e_\phi - e_\psi
\]
as claimed.
\end{proof}

\begin{proposition}
Suppose $H^+$ is trivialisable. Choose a trivialisation and identify sections $\phi,\psi$ with maps $\phi, \psi : B \to S^{b^+-1}$. Choose an orientation on $H^+$ and hence on $S^{b^+-1}$ so that $Obs(\phi,\psi)$ may be considered as a class in $H^{b^+-1}(B ; \mathbb{Z})$. Then
\[
Obs(\phi,\psi) = (-1)^{b^+-1}(\phi^*(\nu) - \psi^*(\nu))
\]
where $\nu \in H^{b^+-1}( S^{b^+-1} ; \mathbb{Z})$ is the generator compatible with the chosen orientation.
\end{proposition}
\begin{proof}
Using $S(H^+) = B \times S^{b^+-1}$ one easily checks that $\tau_\phi = \nu - \phi^*(\nu)$. Therefore $\lambda(\phi,\psi) = -\phi^*(\nu) + \psi^*(\nu)$. For any section $\phi$, the orthogonal complement $\langle \phi \rangle^\perp$ of $\phi$ in $H^+$ is isomorphic to the pullback of the tangent bundle of $S^{b^+-1}$ under $\phi : B \to S^{b^+-1}$. Thus $e_\phi = \phi^*( e( TS^{b^+-1} ) )$, which is zero if $b^+$ is even and $2 \phi^*(\nu)$ is $b^+$ is odd. Thus is $b^+$ is odd we find
\[
Obs(\phi,\psi) = \lambda(\phi,\psi) + e_\phi - e_\psi = -\phi^*(\nu) + \psi^*(\nu) + 2\phi^*(\nu) - 2\psi^*(\nu) = \phi^*(\nu) - \psi^*(\nu).
\]
On the other hand if $b^+$ is even, then $Obs(\phi,\psi) = \lambda(\phi,\psi) = -\phi^*(\nu) + \psi^*(\nu)$ and the result has been proven.
\end{proof}

\subsection{Unparametrised wall crossing formula}\label{sec:uwcf}

In this subsection we will see how to recover the general wall crossing formula for the unparametrised Seiberg-Witten invariants of a $4$-manifold with $b_1$ even, as in \cite{liliu0}.

Let $X$ be a compact, oriented smooth $4$-manifold with $b^+(X) = 1$ and $b_1(X)$ even. Let $\mathfrak{s}$ denote a spin$^c$-structure on $X$. Fix an orientation of $H^+$ and let $\omega \in H^+$ denote the unique class for which a reducible solution exists. Having fixed such an orientation, the set $\mathcal{CH}(f)$ consist of two chambers which we denote by $+,-$ and define as follows: we say that a class $\eta \in H^+$ is in the $+$ chamber if $\eta - \omega$ agrees with the given orientation on $H^+$ and similarly $\eta$ is in the $-$ chamber if $\eta - \omega$ gives the opposite orientation. Choose a metric $g$ and a generic perturbation $\eta$ representing the $+$ chamber. Let $\mathcal{M}_+$ denote the moduli space of solutions to the Seiberg-Witten equations with respect to $(X , \mathfrak{s} , g , \eta)$ and let $\mathcal{L}$ denote the natural line bundle over $\mathcal{M}_+$. Then $\mathcal{M}_+$ is a smooth compact manifold of dimension $2d-b^+(X)-1+b_1(X) = 2(d-1+b_1(X)/2)$, where $d = \frac{c(\mathfrak{s})^2 - \sigma(X)}{8}$. Assume that this dimension is non-negative, i.e. that $(d-1)+b_1(X)/2 \ge 0$. Let $\mathbb{T}$ denote the torus of gauge equivalence classes of $U(1)$-connections on the determinant line bundle of $\mathfrak{s}$ with harmonic curvature. This is a torsor for the Jacobian $Jac(X) = H^1(X ; \mathbb{R}/\mathbb{Z})$, so in particular $\mathbb{T}$ is a torus of dimension $b_1(X)$. The inclusion of $\mathcal{M}_+$ into the configuration space defines a natural map $\pi : \mathcal{M}_+ \to \mathbb{T}$ and this map corresponds to the fact that when we take the finite dimensional approximation of the monopole map $f$, the base space is taken as $B = \mathbb{T}$.

The ordinary Seiberg-Witten invariant $SW^+(X,\mathfrak{s})$ of $(X,\mathfrak{s})$ with respect to the chamber $+$ is given by:
\begin{equation*}
\begin{aligned}
SW^+(X,\mathfrak{s}) &= \int_{\mathcal{M}_+} c_1(\mathcal{L})^{(d-1)+b_1(X)/2} \\
&= \int_{\mathbb{T}} \pi_*( c_1(\mathcal{L})^{(d-1)+b_1(X)/2} ) \\
&= \int_{\mathbb{T}} SW_{(d-1)+b_1(X)/2}(f,+),
\end{aligned}
\end{equation*}
where $f$ denotes a finite dimensional approximation of the monopole map. A similar equality holds for the $-$ chamber.

A short calculation shows that $Obs(+,-) = 1$. Now the wall-crossing formula (Theorem \ref{thm:wcf}) implies that:
\begin{equation*}
\begin{aligned}
SW^+(X,\mathfrak{s}) - SW^-(X,\mathfrak{s}) = \int_{\mathbb{T}} s_{ b_1(X)/2}(D),
\end{aligned}
\end{equation*}
where $s_{b_1(X)/2}(D)$ is the $b_1(X)/2$-th Segre class of the virtual index class $D \in K^0( \mathbb{T} )$ of the natural family of Dirac operators on the fibres of $X \times \mathbb{T} \to \mathbb{T}$.

Let $\{ y_i \}$ denote a basis of $H^1(X ; \mathbb{Z})$ and let $\{ x_i \}$ be the corresponding dual basis of $H^1( \mathbb{T} ; \mathbb{Z}) \cong Hom( H^1(X,\mathbb{Z}) ; \mathbb{Z})$. Let $\Omega = \sum_i x_i \smallsmile y_i \in H^2( X \times \mathbb{T} ; \mathbb{Z})$ denote the first Chern class of the Poincar\'e line bundle. Then the families index theorem gives:
\[
Ch(D) = \int_{X} e^{\frac{1}{2}( c(\mathfrak{s})) + \Omega } \wedge \left( 1 - \frac{\sigma(X)}{8} vol_X \right),
\]
where $\int_{X}$ is understood to mean integration over the fibres of $X \times \mathbb{T} \to \mathbb{T}$. Next we recall that if $X$ is a compact oriented $4$-manifold with $b^+(X)=1$ then $y_1 \smallsmile y_2 \smallsmile y_3 \smallsmile y_4 = 0$ for any $y_1,y_2,y_3,y_4 \in H^1(X ; \mathbb{Z})$ (see \cite[Lemma 2.4]{liliu0}). It follows that $\Omega^4 = 0$ and thus
\begin{equation*}
\begin{aligned}
Ch(D) &= \int_{X} \left( 1 + \frac{c(\mathfrak{s})}{2} + \frac{c(\mathfrak{s})^2}{8} \right) \left( 1 + \Omega + \frac{1}{2}\Omega^2 + \frac{1}{6}\Omega^3 \right)\left( 1 - \frac{\sigma(X)}{8}vol_X \right)\\
&= \int_{X} \left( 1 + \frac{c(\mathfrak{s})}{2} + \frac{c(\mathfrak{s})^2}{8} \right) \left( 1 + \frac{1}{2}\Omega^2 \right)\left( 1 - \frac{\sigma(X)}{8}vol_X \right)\\
&= \frac{c(\mathfrak{s})^2-\sigma(X)}{8} + \frac{1}{4} \int_X c \cdot \Omega^2 \\
&= d + \frac{1}{4} \int_X c \cdot \Omega^2.
\end{aligned}
\end{equation*}
Hence $c_0(D) = d$, $c_1(D) = \frac{1}{4} \int_X c(\mathfrak{s}) \cdot \Omega^2$ and $Ch_i(D) = 0$ for $i > 1$ (where $Ch_i(D)$ denotes the degree $2i$ part of $Ch(D)$). Using the splitting principle one can verify that if
\[
Ch(V) = \sum_{n \ge 0} \frac{1}{n!} p_n, \; \text{where } p_n \in H^{2n}(\mathbb{T} ; \mathbb{Z})
\]
for a virtual bundle $V$, then
\[
s(V) = exp\left( \sum_{n \ge 1} \frac{(-1)^n}{n} p_n \right).
\]
In particular in the case $V = D$, we immediately get
\[
s_j(D) = \frac{(-1)^j}{j!} \alpha^j,
\]
where $\alpha = \int_X \frac{c(\mathfrak{s})}{2} \cdot \frac{\Omega^2}{2}$. Putting $j = b_1(X)/2$ we then get:
\[
SW^+(X,\mathfrak{s}) - SW^-(X,\mathfrak{s}) = \frac{1}{(b_1(X)/2)!} \int_{\mathbb{T}} (-\alpha)^{b_1(X)/2},
\]
which is precisely the general wall crossing formula \cite[Theorem 1.2]{liliu0}.

\section{Divisibility conditions on families Seiberg-Witten invariants from $K$-theory}\label{sec:ktheory}

We have seen how mod $2$ conditions on the families Seiberg-Witten invariants arise via computation of their Steenrod squares. In this section we consider a different approach using $K$-theory and the Chern character, which yields various divisibility conditions on the families Seiberg-Witten invariants modulo torsion. This approach was used by Bauer and Furuta in \cite{bafu} to obtain divisibility conditions of the Seiberg-Witten invariants in the unparametrised setting.

\subsection{$K$-theoretic Seiberg-Witten invariants and divisibility conditions}\label{sec:divcond}

We will work with the $S^1$-equivariant K-theory and cohomology of various spaces. Recall that $H^*_{S^1}(pt ; \mathbb{Z})$ is isomorphic to $\mathbb{Z}[x]$, the ring of polynomials in $x$ with integer coefficients. Let $\widehat{H}^*_{S^1}(pt ; \mathbb{Z})$ denote the completion of $H^*_{S^1}(pt ; \mathbb{Z})$ with respect to the filtration by degree. Then $\widehat{H}^*_{S^1}(pt ; \mathbb{Z}) = \widehat{\mathbb{Z}}[x]$, the ring of formal power series in $x$ with integer coefficients. Let $R[S^1] = K^*_{S^1}(pt)$ denote the $S^1$-equivariant $K$-theory of a point, which is also the representation ring of $S^1$. We let $S^1$ act on $\mathbb{C}$ by scalar multiplication. This defines a representation of $S^1$ and we let $\xi = [\mathbb{C}] \in R[S^1]$. Then $R[S^1] = \mathbb{Z}[\xi , \xi^{-1}]$ is the ring of Laurent polynomials in $\xi$.

For any space $X$ and abelian group $A$, we let $\widehat{H}^*(X ; A)$ denote the completion of $H^*(X ; A)$ with respect to the filtration by degree. We use similar notation for the completion of the equivariant cohomology groups of a $G$-space. The Chern character defines a ring homomorphism $Ch : K^*(X) \to \widehat{H}^{ev/odd}(X ; \mathbb{Q})$. We use this to define the equivariant Chern character as follows. For any $G$-space $X$, let $X_G = X \times_G EG$ denote the homotopy quotient. We define the equivariant Chern character as the composition:
\[
K^*_G( X ) \buildrel \pi_1^* \over \longrightarrow K^*_G( X \times EG ) = K^*(X_G) \buildrel Ch \over \longrightarrow \widehat{H}^{ev/odd}(X_G ; \mathbb{Q}) = \widehat{H}^{ev/odd}_G( X ; \mathbb{Q})
\]
where $\pi_1 : X \times EG \to X$ is the projection to $X$.

\begin{example}
For $G = S^1$ and $X = pt$, we obtain a Chern character homomorphism $Ch : R[S^1] \to \widehat{H}^*_{S^1}(pt ; \mathbb{Q})$. Using $BS^1 = \mathbb{CP}^\infty$, one easily checks that the Chern character is given by:
\[
Ch : \mathbb{Z}[\xi , \xi^{-1}] \to \widehat{\mathbb{Q}}[x], \quad \quad Ch( \xi^k ) = e^{kx}.
\]
\end{example}

Let $S^1$ act trivially on $B$ and recall that we defined $\mathcal{H}^* = H^*_{S^1}(B ; \mathbb{Z}) = H^*(B ; \mathbb{Z})[x]$. Similarly we define
\[
\mathcal{K}^* = K^*_{S^1}(B) = R[S^1] \otimes K^*(B) = K^*(B)[\xi , \xi^{-1}].
\]

The equivariant Chern character for $B$ takes the form:
\[
Ch : \mathcal{K}^* \to \widehat{ \mathcal{H}}^{ev/odd}_\mathbb{Q} \quad Ch( \alpha \otimes \xi^k ) = Ch(\alpha) e^{kx}, \; \; \alpha \in K^*(B),
\]
where $\widehat{\mathcal{H}}^{ev/odd}_\mathbb{Q} = \widehat{\mathcal{H}}^{ev/odd} \otimes \mathbb{Q}$.

Recall the cohomological description of the families Seiberg-Witten invariants: we have a finite dimensional monopole map $f : (S_{V,U} , B_{V,U} ) \to (S_{V',U'} , B_{V',U'})$ satisfying Assumption $1$. Let $[\phi] \in \mathcal{CH}(f)$ be a chamber represented by a section $\phi : B \to U'$ whose image is disjoint from $U$. Assume also that $H^+$ is orientable and fix an orientation. Since $\phi$ avoids $U$, we obtain a pushforward map:
\[
\phi_* : H^j_{S^1}(B ; \mathbb{Z}) \to H^{j+2a'+b'}_{S^1}( S_{V',U'} , S_U ; \mathbb{Z} ).
\]
Then for $m \ge 0$, the $m$-th Seiberg-Witten invariant of $(f,\phi)$ is given by:
\[
SW_m(f,\phi) = (\pi_Y)_*( x^m \smallsmile f^*\phi_*(1) ) \in H^{2m-(2d-b^+-1)}(B , \mathbb{Z}).
\]
Imitating this definition in $K$-theory, we will obtain $K$-theoretic Seiberg-Witten invariants. To do this, we will assume that $H^+$ admits a spin$^c$-structure. Then we can also assume that the bundles $U$ and $U'$ admit spin$^c$-structures such that $U' = H^+ \oplus U$ holds at the level of spin$^c$ vector bundles. Then since the normal bundle of $\phi(B) \subset S_{V',U'}$ is equivariantly $K$-oriented, we obtain a pushforward map in equivariant $K$-theory (see \cite[Section 3]{atsi}):
\[
\phi_* : K^j_{S^1}(B ) \to K^{j+b'}_{S^1}( S_{V',U'} , S_U ).
\]
Similarly, the vertical tangent bundle of $\pi_Y : Y \to B$ has a spin$^c$-structure and we obtain a pushforward map
\[
(\pi_Y)_* : K^*(Y , \partial Y) \to K^{*+b-1}(B).
\]
\begin{definition}
Suppose that $H^+$ is equipped with a spin$^c$-structure. Then for each $m \ge 0$, we define the {\em $m$-th $K$-theoretic Seiberg-Witten invariant of $(f,\phi)$} by:
\[
SW^K_m(f,\phi) = (\pi_Y)_*( \xi^m \smallsmile f^*\phi_*(1) ) \in K^{b^+-1}(B).
\]
\end{definition}

\begin{remark}\label{rem:kty}
We make some remarks on this definition:
\begin{itemize}
\item{It is possible to define $SW^K_m(f,\phi)$ even when $H^+$ is not spin$^c$ provided one works with {\em twisted $K$-theory}. For simplicity we will not discuss this generalisation.}
\item{One can check that $SW^K_m(f,\phi)$ depends only on the stable homotopy class of $f$. The proof is analogous to that of Proposition \ref{prop:swsuspension}.}
\item{One can give a more geometric interpretation of $SM^K_m(f,\phi)$ analogous to the cohomological Seiberg-Witten invariants. Let $\mathcal{M}_\phi$ be the families Seiberg-Witten moduli space (after perturbing $f$ to obtain a smooth moduli space) and let $\pi_{\mathcal{M}_\phi} :\mathcal{M}_\phi \to B$ be the projection. Then arguing in much the same way as we did for the cohomological invariants, we see that:
\[
SW^K_m(f,\phi) = (\pi_{\mathcal{M}_\phi})_*( \mathcal{L}^m )
\]
where $\mathcal{L} \to \mathcal{M}_\phi$ is the natural line bundle over $\mathcal{M}_\phi$. Note that a spin$^c$-structure on $H^+$ induces a relative spin$^c$-structure on $\mathcal{M}_\phi$, that is, a spin$^c$-structure on $T \mathcal{M}_\phi \oplus \pi_{\mathcal{M}_\phi}^*(TB)$. This is exactly the data requried to define the $K$-theoretic push-forward with respect to $\pi_{\mathcal{M}_\phi}$.}
\item{In the case that $B = \{pt\}$ is a point and $b^+$ is odd, it follows that $SW^K_m(f,\phi) \in K^0(pt) = \mathbb{Z}$ is the index of the spin$^c$ Dirac operator on $\mathcal{M}_\phi$ coupled to the line bundle $\mathcal{L}^m$.}
\end{itemize}

\end{remark}

For a complex vector bundle $V$ over $B$, let $Td^{S^1}(V) \in \widehat{\mathcal{H}}^{ev}_\mathbb{Q}$ denote the $S^1$-equivariant Todd class of $V$, where as usual $S^1$ acts on $V$ by scalar multiplication. This can be expanded as a formal power series in $x$:
\[
Td^{S^1}(V) = \sum_{j \ge 0} Td_j(V) x^j
\]
for some characteristic classes $Td_j(V) \in H^{ev}(B ; \mathbb{Q})$.

Recall that $Spin^c(n) = Spin(n) \times_{\pm} U(1)$ and therefore we obtain a homomorphism $\varphi : Spin^c(n) \to U(1)$ which sends $(g,z) \in Spin(n) \times U(1)$ to $z^2$. Suppose that $E$ is a rank $n$ vector bundle with a spin$^c$ structure. The homomorphism $\varphi$ applied to the principal $Spin^c(n)$-bundle of $E$ determines a line bundle. We call this the {\em spin$^c$ line bundle of $E$}.

\begin{theorem}\label{thm:swk}
Let $\kappa \in H^2(B ; \mathbb{Z})$ be the first Chern class of the spin$^c$ line bundle associated to $H^+$. Then the $K$-theoretic and cohomological Seiberg-Witten invariants are related by:
\[
Ch( SW_m^K(f,\phi)) = e^{-\kappa/2}\hat{A}(H^+)^{-1} \sum_{j \ge 0} Td_j(D) \sum_{k \ge 0} \frac{m^k}{k!}SW_{j+k}(f,\phi) \in H^*(B ; \mathbb{Q}).
\]
\end{theorem}
\begin{proof}
Recall from Subsection \ref{sec:cohomological} the definition of the spaces $Y$ and $\widetilde{Y}$. Let $\pi_Y : Y \to B$ and $\pi_{\widetilde{Y}} : \widetilde{Y} \to B$ be the projections to $B$. Let $\mathbb{V} = Ker( (\pi_Y)_* )$ denote the kernel of $(\pi_Y)$ and similarly let $\widetilde{\mathbb{V}} = Ker( (\pi_{\widetilde{Y}})_* )$. By definition $\widetilde{Y}$ is the complement of a tubular neighbourhood of $S_U$ in $S_{V,U}$. Let $\hat{A}^{S^1}( E )$ denote the $S^1$-equivariant $\hat{A}$-genus of an equivariant vector bundle $E$. Then $\hat{A}^{S^1}(\widetilde{\mathbb{V}}) = \hat{A}^{S^1}( T_{vert} S_{V,U})|_{\widetilde{Y}}$, where $T_{vert} S_{V,U}$ denotes the vertical tangent bundle of $S_{V,U} \to B$. But $T_{vert} S_{V,U} \oplus \mathbb{R} \cong V \oplus U \oplus \mathbb{R}$ and so it follows that $\hat{A}^{S^1}( \widetilde{\mathbb{V}} ) = \hat{A}^{S^1}(V \oplus U \oplus \mathbb{R}) = \hat{A}^{S^1}(V \oplus U)$. Next recall that $Y = \widetilde{Y}/S^1$ is the quotient of $\widetilde{X}$ by a free $S^1$-action. It follows that
\[
\hat{A}(\mathbb{V}) = \hat{A}^{S^1}( \widetilde{\mathbb{V}}) = \hat{A}^{S^1}(V \oplus U)
\]
where we have identified $S^1$-equivariant cohomology on $\widetilde{Y}$ with ordinary cohomology on $Y$.

Since assuming that $H^+$ admits a spin$^c$ structure, we can after stabilisation assume that $U' = H^+ \oplus U$ as spin$^c$ vector bundles. For a spin$^c$ vector bundle $W$ we let $\kappa_W$ denote the first Chern class of the associated line bundle of the spin$^c$-structure. In particular $\kappa = \kappa_{H^+}$ and
\[
\kappa_{U'} = \kappa + \kappa_{U}.
\]
Recall that $SW^K_m(f,\phi)$ is defined by $SW^K_m(f,\phi) = (\pi_Y)_*( \xi^m \smallsmile f^*\phi_*(1) )$. Now we proceed by direct calculation using Grothendieck-Riemann Roch:
\begin{equation*}
\begin{aligned}
& Ch( SW_m^K(f,\phi)) = Ch \left( (\pi_Y)_*( \xi^m \smallsmile f^* \phi_*(1) )\right) \\
&= (\pi_Y)_* \left( Ch( \xi^m \smallsmile f^* \phi_*(1) ) e^{\kappa_{\mathbb{V}}/2} \hat{A}(\mathbb{V}) \right) \\
&= (\pi_Y)_* \left( Ch( \xi^m \smallsmile f^* \phi_*(1) ) Td^{S^1}(V) e^{\kappa_U/2} \hat{A}(U) \right) \\
&= e^{\kappa_U/2} \hat{A}(U) \cdot (\pi_Y)_* \left( e^{mx} \smallsmile f^*( Ch(\phi_*(1))) Td^{S^1}(V) \right) \\
&= e^{\kappa_U/2} \hat{A}(U) \cdot (\pi_Y)_* \left( e^{mx} \smallsmile f^*( \phi_*( 1 ) ) Td^{S^1}(V')^{-1} e^{-\kappa_{U'}/2} \hat{A}(U')^{-1} Td^{S^1}(V) \right) \\
&= e^{(\kappa_U - \kappa_{U'})/2}\hat{A}(U)\hat{A}(U')^{-1} \cdot (\pi_Y)_* \left( e^{mx} \smallsmile f^*(\phi_*(1)) Td^{S^1}(V)Td^{S^1}(V')^{-1} \right) \\
&= e^{-\kappa/2}\hat{A}(H^+)^{-1} \cdot (\pi_Y)_* \left( e^{mx} \smallsmile f^*(\phi_*(1)) Td^{S^1}(D) \right) \\
&= e^{-\kappa/2}\hat{A}(H^+)^{-1} \cdot \sum_{j \ge 0} Td_j(D) \sum_{k \ge 0} \frac{m^k}{k!} \cdot (\pi_Y)_* \left( x^{j+k} \smallsmile f^*(\phi_*(1)) \right) \\
&= e^{-\kappa/2}\hat{A}(H^+)^{-1} \cdot \sum_{j \ge 0} Td_j(D) \sum_{k \ge 0} \frac{m^k}{k!} \cdot SW_{j+k}(f,\phi).
\end{aligned}
\end{equation*}
\end{proof}

\begin{example}\label{ex:point}
Let $X$ be a compact, oriented, smooth $4$-manifold with $b_1(X) = 0$ and $b^+(X) = 2p+1$ odd. Let $\mathfrak{s}$ be a spin$^c$-structure on $X$ and put $d = \frac{c(\mathfrak{s})^2-\sigma(X)}{8}$. Then we take $B = \{pt\}$. So $\hat{A}(H^+) = 1$, $\kappa = 0$ and
\[
Td^{S^1}(D) = \left(\frac{x}{1-e^{-x}} \right)^d.
\]
So $Td_j(D) = c_{j,d}$, where $c_{j,d}$ is the coefficient of $x^j$ in $\left(\frac{x}{1-e^{-x}} \right)^d$. Assume that $d \ge p+1$. The only non-zero Seiberg-Witten invariant is $SW_{d-p-1}(f,\phi)$ is the ordinary Seiberg-Witten invariant of $X$ (with respect to the chamber $\phi$ if $b^+(X)=1$). On the other hand $SW^K_m(f,\phi) \in \mathbb{Z}$ is the index of the spin$^c$ Dirac operator coupled to $\mathcal{L}^m$ on the Seiberg-Witten moduli space, which we denote as $Ind(D \otimes \mathcal{L}^m)$. Theorem \ref{thm:swk} gives:
\[
SW_m^K(f,\phi) = \left( \sum_{j = 0}^{d-p-1} c_{d-p-1-j,d} \frac{m^j}{j!} \right) SW(X , \mathfrak{s} , \phi).
\]
However, we notice that $\sum_{j = 0}^{d-p-1} c_{d-p-1-j,d} \frac{m^j}{j!}$ is the $x^{d-p-1}$-coefficient of 
\[
\left(\frac{x}{1-e^{-x}} \right)^d e^{mx}.
\]
But this is the same as:
\[
\int_{\mathbb{CP}^{d-1}} \left(\frac{x}{1-e^{-x}} \right)^d x^p e^{mx} = \int_{\mathbb{CP}^{d-1}}  Td( \mathbb{CP}^{d-1}) x^p Ch( \mathcal{O}(m))
\]
where now $x = c_1( \mathcal{O}(1))$. Let us denote this quantity by $n(d,m,p)$. Clearly when $p=0$, $n(d,m,p)$ is the Dolbeault index of $\mathcal{O}(m)$, which is $\binom{m+d-1}{m}$.

Define rational numbers $a_{p,l}$ to be the coefficients of the Taylor expansion (c.f. \cite[Theorem 3.7]{bafu}):
\[
log(1-y)^p = \sum_{l=0}^\infty a_{p,l} y^{p+l}.
\]
Then
\[
x^p = (-1)^p \sum_{l \ge 0} a_{p,l} (1-e^{-x})^{p+l}
\]
and thus
\begin{equation*}
\begin{aligned}
n(d,m,p) &= (-1)^p \sum_{l \ge 0} a_{p,l} \int_{\mathbb{CP}^{d-1}} Td(\mathbb{CP}^{d-1}) Ch ((1 - \xi^{-1})^{p+l} \xi^m) \\
&= (-1)^p \sum_{l \ge 0} a_{p,l} Res|_{x=0} \left( \frac{1}{(1-e^{-x})^d} (1-e^{-x})^{p+l} e^{mx} \right) \\
&= (-1)^p \sum_{l = 0}^{d-p-1} a_{p,l} Res|_{x=0} \left( \frac{1}{(1-e^{-x})^{d-p-l}} e^{mx} \right) \\
&= (-1)^p \sum_{l=0}^{d-p-1} a_{p,l} \binom{m+d-p-l-1}{m}.
\end{aligned}
\end{equation*}
So we have shown that:
\[
Ind(D \otimes \mathcal{L}^m) = \left( (-1)^p \sum_{l=0}^{d-p-1} a_{p,l} \binom{m+d-p-l-1}{m} \right) SW(X , \mathfrak{s} , \phi).
\]
Let us set $n = d-p-1$ and $q(m) = Ind(D \otimes \mathcal{L}^m)$, so that
\[
q(m) = (-1)^p \left( \sum_{l=0}^{n} a_{p,n-l} \binom{m+l}{l} \right) SW(X , \mathfrak{s} , \phi).
\]
Evidently $q(m)$ is a polynomial in $m$ with rational coefficients. Moreover $q(m) \in \mathbb{Z}$ whenever $m$ is a non-negative integer. Introduce the difference operator $\Delta$ on polynomial functions of $m$ by: $(\Delta f)(m) = f(m+1)-f(m)$. Then $\Delta \binom{m+k}{j} = \binom{m+k}{j-1}$ by Pascal's identity. Iterating, we get that $\Delta^k \binom{m+k}{k} = 1$ and $\Delta^{r} \binom{m+k}{k} = 0$ for all $r > k$. Now since $\Delta^l q(m)$ is integer valued for all $l,m \ge 0$, we see by induction on $l$ that $a_{p,n-l} SW(X , \mathfrak{s} , \phi) \in \mathbb{Z}$ for $l = 0, 1 , \dots , n$. Thus $SW(X , \mathfrak{s} , \phi)$ is divisible by the denominators of $a_{p,l}$ for $l = 0,1, \dots , n$. This is precisely the statement of \cite[Theorem 3.7]{bafu}.
\end{example}

\begin{proposition}\label{prop:torsion}
Suppose that $H^+$ is equipped with a spin$^c$-structure and let $\kappa$ be the first Chern class of the spin$^c$ line bundle of $H^+$. Suppose that the Chern classes of $D$ are torsion. Then
\[
e^{\kappa/2} \hat{A}(H^+) Ch(SW^K_m(f,\phi)) = \sum_{r \ge 0} (-1)^{d-r-1} \sum_{l=0}^r a_{d-r-1,r-l} \binom{m+l}{l}SW_r(f,\phi) \in H^*(B ; \mathbb{Q})
\]
where the rational numbers $a_{p,l}$ are defined as in Example \ref{ex:point} (note that $a_{p,l}$ is defined even when $p < 0$).
\end{proposition}
\begin{proof}
Since the Chen classes of $D$ are torsion it follows that $Td^{S^1}(D) = \left( \frac{x}{1-e^{-x}} \right)^d$ and thus $Td_j(D) = c_{j,d}$, where $c_{j,d}$ are defined as in Example \ref{ex:point}. Substituting into Theorem \ref{thm:swk} we get:
\begin{equation*}
\begin{aligned}
e^{\kappa/2} \hat{A}(H^+) Ch(SW^K_m(f,\phi)) &= \sum_{j,k \ge 0} c_{j,d} \frac{m^k}{k!} SW_{j+k}(f,\phi)\\
&= \sum_{r \ge 0} \left( \sum_{j+k=r} c_{j,d} \frac{m^k}{k!} \right) SW_r(f,\phi) \\
&= \sum_{r \ge 0} n(d,m,d-r-1) SW_r(f,\phi)
\end{aligned}
\end{equation*}
where the rational numbers $n(d,m,p)$ are defined by 
\[
n(d,m,p) = \sum_{j+k=d-p-1} c_{d-p-1-j,d} \frac{m^j}{j!}.
\]
But from Example \ref{ex:point}, we find that
\[
n(d,m,d-r-1) = (-1)^{d-r-1} \sum_{l=0}^r a_{d-r-1,r-l} \binom{m+l}{l}
\]
and thus
\[
e^{\kappa/2} \hat{A}(H^+) Ch(SW^K_m(f,\phi)) = \sum_{r \ge 0} (-1)^{d-r-1} \sum_{l=0}^r a_{d-r-1,r-l} \binom{m+l}{l}SW_r(f,\phi).
\]
\end{proof}

\begin{corollary}\label{cor:divsphere}
Suppose that $B = S^{2r}$ is an even dimensional sphere and that $f$ is a finite dimensional approximation of the monopole map of a spin$^c$ family of $4$-manifolds $\pi : E \to B$ whose fibres are diffeomorphic to a $4$-manifold $X$ with $b_1(X)=0$ and $b^+(X) = 2p+1$ odd. Suppose that $n = r+d-p-1 \ge 0$ and that $c_r(D) = 0$. Then $SW_{n}(f,\phi) \in H^{2r}(S^{2r} ; \mathbb{Z}) \cong \mathbb{Z}$ is divisible by the denominators of $a_{p-r,l}$, for $l = 0,1, \dots n$.
\end{corollary}
\begin{proof}
Since $B$ is simply-connected, it follows that the bundle $H^2 = R^2 \pi_* \mathbb{R}$ over $B$ whose fibres are degree $2$ cohomology of the fibres of $E \to B$ is trivial. It follows then that $H^+$ is also trivial. Thus $H^+$ admits a spin structure and therefore a spin$^c$-structure with $\kappa = 0$. Also $\hat{A}(H^+) = 1$ since $H^+$ is trivial. From Proposition \ref{prop:torsion} we get:
\[
Ch(SW^K_m(f,\phi)) = \sum_{j=0}^d (-1)^{d-j-1} \sum_{l=0}^j a_{d-j-1,j-l} \binom{m+l}{l}SW_j(f,\phi) \in H^*(S^{2r} ; \mathbb{Q}).
\]
But for a sphere, the Chern character takes values in integral cohomology. In particular, extracting the degree $2r$ part, we find:
\[
\sum_{l=0}^{n} a_{p-r,n-l} \binom{m+l}{l} SW_n(f, \phi) \in H^{2r}(S^{2r} ; \mathbb{Z})
\]
for all $m \ge 0$, where $n = r+d-p-1$. Arguing as in Example \ref{ex:point}, this implies that $SW_n(f,\phi)$ is divisible by the denominators of $a_{p-r,l}$ for $l = 0,1, \dots , n$.
\end{proof}

\begin{remark}
In the setting of Corollary \ref{cor:divsphere}, if $r \ge 5$, then it follows by the families index theorem that $c_r(D) = 0$ holds automatically. To see this, note that if $r \ge 5$ then $H^6(E ; \mathbb{Z}) = H^8(E ; \mathbb{Z})=0$ by the Leray--Serre spectral sequence. The families index theorem gives
\[
Ch(D) = \int_{E/B} e^{c_1(\mathfrak{s})/2} \hat{A}( T_{vert} E )
\]
where $T_{vert}E = ker( \pi_* : TE \to TB )$ denotes the vertical tangent bundle and $\int_{E/B}$ denotes fibre integration. Let $p_1,p_2$ denote the first and second Pontryagin classes of $T_{vert}E$. Then $p_1^2 = p_2 = 0$, since $H^8(E ; \mathbb{Z}) = 0$. So
\[
\hat{A}(T_{vert}E) = 1 - \frac{1}{24} p_1.
\]
Similarly
\[
e^{c_1(\mathfrak{s})/2} = 1 + \frac{c_1(\mathfrak{s})}{2} + \frac{c_1(\mathfrak{s})^2}{8}
\]
as $c_1(\mathfrak{s})^3 \in H^6(E ; \mathbb{Z}) = 0$. It follows that all non-zero terms in $e^{c_1(\mathfrak{s})/2} \hat{A}(T_{vert}E)$ have degree $\le 8$ and thus all non-zero terms in $Ch(D)$ have degree $\le 4$. But  since $B = S^{2r}$ is a sphere, we have $Ch(D) = rk(D) + (-1)^r c_r(D)/(r-1)!$. So $c_r(D)$ is a non-zero multiple of the degree $2r$ term in $Ch(D)$, which is zero since $2r \ge 10$.
\end{remark}

\begin{remark}
To be in the range where there is no wall crossing, we need $b^+ > dim(B)+1$, which in the above corollary is equivalent to $p > r$. On the other hand if $r > p$ then the primary difference class $Obs(\phi,\psi)$ is easily seen to vanish, so there can only be wall crossing if $p=r$.
\end{remark}

\subsection{$K$-theoretic wall crossing formula}\label{sec:kwcf}

In this subsection we state a wall crossing formula for the $K$-theoretic Seiberg-Witten invariants. As usual, let $f : (S_{V,U} , B_{V,U}) \to (S_{V',U'} , B_{V',U'})$ be a finite dimensional monopole map. Let $U' = U \oplus H^+$ and suppose that $H^+$ is equipped with a spin$^c$-structure. As in the previous subsection, let $\kappa \in H^2(B ; \mathbb{Z})$ denote the first Chern class of the associated line bundle of this spin$^c$-structure. We will use a superscript $K$ to indicate pushforwards in $K$-theory.

Let $[\phi],[\psi] \in \mathcal{CH}(f)$ be chambers of $f$, represented by maps $\phi, \psi : B \to S(H^+)$.
\begin{lemma}\label{lem:exactk}
We have a short exact sequence
\[
0 \to K^*(B) \buildrel \pi^* \over \longrightarrow K^*( S(H^+) ) \buildrel \pi^K_* \over \longrightarrow K^{*-(b^+-1)}(B) \to 0.
\]
\end{lemma}
\begin{proof}
Use the section $\phi$ to identify $B$ with a subspace of $S(H^+)$. Then $\pi : S(H^+) \to B$ may be thought of as a retraction. Hence the long exact sequence in $K$-theory for the pair $( S(H^+) , \phi(B) )$, splits into short exact sequences
\[
0 \to K^*(B) \buildrel \pi^* \over \longrightarrow K^*( S(H^+) ) \to K^{*}( S(H^+) , B ) \to 0.
\]
By the Thom isomorphism, $K^{*}( S(H^+) , B ) \cong K^{*-(b^+-1)}( B )$ and under this isomorphism the map $K^*( S(H^+) ) \to K^{*}( S(H^+) , B )$ in the above sequence corresponds to $\pi^K_* : K^*( S(H^+) ) \to K^{*-(b^+-1)}( B )$.
\end{proof}

Now since $\phi$ and $\psi$ are sections of $\pi$, we have $\pi \circ \phi = \pi \circ \psi$ and so $\pi^K_*( \phi^K_*(1) - \psi^K_*(1) ) = 0$. Hence by Lemma \ref{lem:exactk}, there exists a unique class $Obs^K(\phi , \psi) \in K^{b^+-1}(B)$ such that 
\[
\phi^K_*(1) - \psi^K_*(1) = \pi^*( Obs^K(\phi , \psi) ).
\]
Clearly $Obs^K(\phi,\psi)$ is the $K$-theoretic analogue of the class $Obs(\phi,\psi)$ defined in \textsection \ref{sec:relativeobstruction}. In fact, the two are related as follows:

\begin{proposition}
For any $\phi,\psi : B \to S(H^+)$, we have
\[
Ch( Obs^K(\phi,\psi)) = e^{-\kappa/2}\hat{A}(H^+)^{-1}Obs(\phi , \psi) \in H^*(B ; \mathbb{Q}).
\]
\end{proposition}
\begin{proof}
The normal bundle of $\phi(B)$ in $S(H^+)$ is isomorphic to $\langle \phi \rangle^\perp$, the orthogonal complement of $\phi$ in $H^+$. But this is stably equivalent to $H^+$. Therefore Grothendieck-Riemann-Roch gives:
\[
Ch(\phi^K_*(1)) = \phi_*( Ch(1) e^{-\kappa/2} \hat{A}(H^+)^{-1} ) = e^{-\kappa/2}\hat{A}(H^+) \phi_*(1).
\]
Similarly, 
\[
Ch(\psi^K_*(1)) = e^{-\kappa/2}\hat{A}(H^+)\psi_*(1).
\]
Taking differences gives:
\begin{equation*}
\begin{aligned}
\pi^*(Ch( Obs^K(\phi,\psi))) &= Ch( \pi^*(Obs^K(\phi,\psi))) \\
&= Ch( \phi^K_*(1) - \psi^K_*(1)) \\
&= e^{-\kappa/2}\hat{A}(H^+)^{-1}( \phi_*(1) - \psi_*(1)) \\
&= e^{-\kappa/2}\hat{A}(H^+)^{-1}\pi^*(Obs(\phi,\psi)).
\end{aligned}
\end{equation*}
But $\pi^* : K^*(B) \to K^*( S(H^+) )$ is injective since $S(H^+)$ admits sections. Hence $Ch( Obs^K(\phi,\psi)) = e^{-\kappa/2}\hat{A}(H^+)^{-1}Obs(\phi , \psi)$.
\end{proof}

\begin{lemma}\label{lem:sym}
Let $V \to B$ be a complex vector bundle over $B$ of rank $a \ge 2$ and let $\pi_{\mathbb{P}(V)} : \mathbb{P}(V) \to B$ be the associated projective bundle. Set $\xi = \mathcal{O}_V(1)$. Let $a' \ge 0$ be a non-negative integer and set $d = a-a'$. For any $W \in K^0(B)$ and any integer $m$, define $S^+_m(W), S^-_m(W) \in K^0(B)$ by:
\[
S^+_m(W) = \begin{cases} Sym^m(W^*) & \text{if } m \ge 0, \\ 0 & \text{if } m < 0, \end{cases} \quad S^-_m(W) = \begin{cases} 0 & \text{if } m > -d, \\ (-1)^{d-1}Sym^{-m-d}(W) & \text{if } m \le -d. \end{cases}
\]
Also let $S_m(W) = S^+_m(W) + S^-_m(W)$. Then:
\[
(\pi_{\mathbb{P}(V)})_*^K( \xi^m(1-\xi^{-1})^{a'}) = S_m(V - \mathbb{C}^{a'}).
\]
\end{lemma}
\begin{proof}
For any integer $m$, let us set $q_m = (\pi_{\mathbb{P}(V)})_*^K( \xi^m)$ and $h_m = (\pi_{\mathbb{P}(V)})_*^K( \xi^m(1-\xi^{-1})^{a'} )$. Then by expanding $(1-\xi^{-1})^{a'}$, we find:
\begin{equation}\label{equ:hm}
h_m = \sum_{j=0}^{a'} (-1)^j q_{m-j} \binom{a'}{j}.
\end{equation}
Next we observe that $q_m = (\pi_{\mathbb{P}(V)})_*^K( \mathcal{O}_V(m)  )$ is the families index of $\mathcal{O}_V(m)$. But $\mathbb{P}(V) \to B$ has a fibrewise complex structure and the families index corresponds to taking fibrewise index of the Dolbeault complex coupled to $\mathcal{O}_V(m)$. Hence for $m \ge 0$, $q_m$ is represented by the vector bundle over $B$ whose fibres consist of homogeneous degree $m$ polynomials, so $q_m = Sym^m(V^*)$ for $m \ge 0$. Combined with Serre duality we find that for any $m \in \mathbb{Z}$ $q_m = q_m^+ + q_m^-$, where
\[
q_m^+ = \begin{cases} Sym^m(V^*) & \text{if } m \ge 0 \\ 0 & \text{if } m < 0, \end{cases} \quad q^-_m = \begin{cases} 0 & \text{if } m > -a, \\ (-1)^{a-1}Sym^{-m-a}(V) & \text{if } m \le -a. \end{cases}
\]
The from Equation (\ref{equ:hm}) we have $h_m = h_m^+ + h_m^-$, where
\begin{equation}\label{equ:hpm}
h_m^\pm = \sum_{j=0}^{a'} (-1)^j q_{m-j}^\pm \binom{a'}{j}.
\end{equation}
To complete the lemma, it suffices to show that $h_m^\pm = S^\pm_m(V - \mathbb{C}^{a'})$. Consider the $+$ case first. Introduce the following formal power series in $t$ with coefficients in $K^0(B)$:
\[
Q^+ = \sum_{m \ge 0} q_m^+ t^m = \sum_{m \ge 0} Sym^m(V^*) t^m.
\]
By the splitting principle we may write $V = L_1 + \cdots + L_a$ for some line bundles. Then we have
\[
Q^+ = \prod_{j=1}^a \frac{1}{(1-tL_j^*)}.
\]
Next we note that since $q^+_m = 0$ for $m < 0$, then $h^+_m = 0$ for $m < -a'$. Hence setting
\[
H^+ = \sum_{m \in \mathbb{Z}} h^+_m t^m,
\]
we have that $H^+$ is a formal Laurent series in $t$ with coefficients in $K^0(B)$. Then from Equation (\ref{equ:hpm}), one finds that
\[
H^+ = Q^+(1-t)^{a'} = \left( \prod_{j=1}^a \frac{1}{(1-tL_j^*)} \right) (1-t)^{a'}.
\]
Equating coefficients one sees that $h_m^+ = S^+_m(V - \mathbb{C}^{a'})$.

Now we consider the $-$ case. Introduce the formal power series
\[
Q^- = \sum_{m < 0} q_m^- t^{-m} = (-1)^{a-1} t^a \left( \sum_{m \ge 0} Sym^m(V)t^m \right) = (-1)^{a-1} t^a \left( \prod_{j=1}^a \frac{1}{(1-tL_j)} \right).
\]
Similar to the $+$ case, we find that $h^-_m = 0$ for $m > -a$, so we may define the following formal Laurent series:
\[
H^- = \sum_{m \in \mathbb{Z}} h^-_m t^{-m}.
\]
Then from Equation (\ref{equ:hpm}), one finds that
\begin{equation*}
\begin{aligned}
H^- &= Q^-(1-t^{-1})^{a'} \\
&= (-1)^{a-1} t^a \left( \prod_{j=1}^a \frac{1}{(1-tL_j)} \right) (-1)^{a'} t^{-a'} (1-t)^{a'} \\
&= (-1)^{d-1} t^d \left( \prod_{j=1}^a \frac{1}{(1-tL_j)} \right) (1-t)^{a'}.
\end{aligned}
\end{equation*}
Equating coefficients, one sees that $h^-_m = S^-_m(V - \mathbb{C}^{a'})$.
\end{proof}

\begin{theorem}[$K$-theoretic wall crossing formula]
For any $\phi, \psi \in \mathcal{CH}(f)$ and any integer $m$, we have:
\[
SW^K_m(f,\phi)-SW^K_m(f,\psi) = Obs^k(\phi,\psi) S_m(D).
\]
where $S_m$ is defined as in Lemma \ref{lem:sym}.
\end{theorem}
\begin{proof}
Since the proof is very similar to the cohomological wall crossing formula we will give only a brief sketch. By stabilisation, we can assume $U',V'$ are trivial. Then $U$ and $U'$ are both spin$^c$. We can also assume the rank of $V$ is at least $2$. Repeating the arguments of \textsection \ref{sec:wcf1}, we find
\[
SW_m^K(f,\phi) - SW_m^K(f,\psi) = (\pi_Y)_*^K( \xi^m \cdot f^*( \phi^K_*(1) - \psi^K_*(1)))
\]
where $\phi, \psi : (B, \emptyset) \to (S_{V',U'} , S_U)$ are representatives for the chambers. We can assume that $\phi,\psi$ take values in $S(H^+)$. Then as before we find
\[
\phi^K_*(1) - \psi^K_*(1) = \iota^K_*( (\phi^0)_*^K(1) - (\psi^0)^K_*(1) ) = Obs^K(\phi,\psi) \iota^K_*(1)
\]
where $\phi^0,\psi^0$ are the corresponding maps $\phi^0,\psi^0 : B \to S(H^+)$ and $\iota : (S(H^+) , \emptyset) \to (S_{V',U'} , S_U)$ is the inclusion. The $K$-theoretic analogue of Lemma \ref{lem:iota1} is $\iota^K_*(1) = (1-\xi^{-1})^{a'} \delta \tau^K_{0,U}$. Therefore we have:
\begin{equation*}
\begin{aligned}
SW_m^K(f,\phi) - SW_m^K(f,\psi) &= Obs^K(\phi,\psi) (\pi_Y)_*^K( \xi^m(1-\xi^{-1})^{a'} \delta \tau^K_{0,U}) \\
&= Obs^K(\phi,\psi) (\pi_{\mathbb{P}(V)})_*^K( \xi^m(1-\xi^{-1})^{a'}) \\
&= Obs^K(\phi,\psi) S_m(V - \mathbb{C}^{a'}) \\
&= Obs^K(\phi,\psi) S_m(D),
\end{aligned}
\end{equation*}
where we used Lemma \ref{lem:sym}.
\end{proof}


\bibliographystyle{amsplain}

\end{document}